%% file: main.tex
\documentclass[microtype]{gtpart}
\usepackage[utf8]{inputenc}
\usepackage{amsmath}
\usepackage{amssymb}
\usepackage[numbers]{natbib}
\usepackage{graphicx}
\usepackage{color}
\usepackage{amsthm}
\usepackage{subcaption}
\usepackage{import}
\usepackage{array}

\theoremstyle{plain}

\makeatletter
\newtheorem*{rep@theorem}{\rep@title}
\newcommand{\newreptheorem}[2]{%
\newenvironment{rep#1}[1]{%
 \def\rep@title{#2 \ref{##1}}%
 \begin{rep@theorem}}%
 {\end{rep@theorem}}}
\makeatother

\newtheorem{thm}{Theorem}[section]
\newreptheorem{thm}{Theorem}
\newtheorem{lem}[thm]{Lemma}
\newtheorem{prop}[thm]{Proposition}
\newtheorem{cor}[thm]{Corollary}

\theoremstyle{definition}
\newtheorem{defn}[thm]{Definition}

\theoremstyle{remark}
\newtheorem*{rem}{Remark}

\title{The existence of a universal transverse knot}

\author{Jesús Rodríguez-Viorato}
\givenname{Jesus}
\surname{Rodríguez Viorato}
\address{Centro de Investigación en Matemáticas}
\email{jesusr@cimat.mx}
\urladdr{https://www.cimat.mx/~jesusr/}

\arxivreference{2001.03663}


\begin{document}
\begin{abstract}
    We prove that there is a knot $K$ transverse to $\xi_{std}$, the tight contact structure of $S^3$, such that every contact 3-manifold $(M, \xi)$ can be obtained as a contact covering branched along $K$. By contact covering we mean a map $\varphi: M \to S^3$ branched along $K$ such that $\xi$ is contact isotopic to the lifting of $\xi_{std}$ under $\varphi$.
\end{abstract}
\maketitle

\section{Introduction}

In 1982, W. Thurston \cite{Thurston} first introduced the universality of links. He showed that there is a link $K$ in $S^3$ such that any compact orientable 3-manifold can be obtained as a covering space of $S^3$ branched along $K$ . Such links are called \emph{universals}.

Later, J. Gonz\'alo \cite{Gonzalo-Branch-Covers} used branched covering in the context of contact 3-manifolds. He noticed that it is possible to ``lift" or  ``pull-back" the contact structure on the base space $S^3$  up to the domain $M$ through a covering $\phi: M  \to S^3$ branched along a transverse link $L \subset S^3$ (see Definition \ref{def:contact-branch-covering}). After that,  H. Geiges \cite{geiges-article} generalized the result to other dimensions, and \"{O}zt\"{u}rk and Niederkr\"{u}ger \cite{branch-covering-and-contact} showed that this "pull back" is unique up to contact isotopy. 

In 2013, M. Casey raised the question about the existence of transverse universal links \cite{Casey}. That is, if there is a link $L$ transverse to the standard contact structure $\xi_{std}$ of $S^3$ such that any contact 3-manifold is a contact covering of $S^3$ branched along $L$. She also showed that the figure-eight knot is not of this kind.

Recently, R. Casals and J. Etnyre \cite{TUL} proved that there is a transverse universal link. They also asked if it is possible to find one which is connected, \emph{a.k.a} a transverse universal knot. 

In the present work, we obtain an affirmative answer to that question, with the proof of the following theorem.

\begin{thm}\label{thm:main-corolary}
There is a universal transverse knot.
\end{thm}
\begin{proof}
 The proof is just a corollary of Theorem \ref{thm:main-theorem}. From it, we know that there exists a contact branch covering $\varphi: (S^3, \xi_{std}) \to (S^3, \xi_{std})$ branched along some transverse knot $K$  such that $\varphi^{-1}(K)$ contains a sublink that is contact isotopic to the universal transverse link $L$ given in \cite{TUL}.
 
Because $L$ is transverse universal, we know that for any contact manifold $(M, \xi)$ there exists a contact covering $\varphi': (M, \xi) \to (S^3, \xi_{std})$ branched along $L$. Then, $ \varphi \circ \varphi': (M, \xi) \to (S^3, \xi_{std}) $ is a contact covering branched along $K$. Hence,  $K$ is transverse universal.
\end{proof}


In fact, Theorem \ref{thm:main-theorem} says even more, it says that we can construct such a branch covering for any transverse link $L$ in $(S^3, \xi_{std})$, being $L$ universal or not. This theorem can be considered as the contact version of a previous theorem from  Hilden, Lozano, and Montesinos in \cite{2-universal}. 

\begin{thm}\label{thm:main-theorem}
Given any transverse link $L$ in $(S^3, \xi_{std})$, there is a contact covering $\varphi:(S^3, \xi_{std}) \to  (S^3, \xi_{std})$ branched along some transverse link $L'$ such that $L$ is contained in $\varphi^{-1}(L')$. If $L$ is disconnected, $L'$ can be chosen connected or with fewer components than $L$.
\end{thm}

We will give a constructive proof of this theorem.  To get the universal link $K$, one needs to follow our construction replacing $L$ with the transverse universal link provided by Casals and Etnyre. The algorithm is long, and it will create a knot with a lot of crossings.  If we want to obtain a knot with fewer crossings, a possible way is to simplify the algorithm presented here for the specific case of the link $L$ given in \cite{TUL}. 

The paper is organized as follows. In Section \ref{sec:contact-and-openbooks}, we review some known facts about contact structures and open books. At the end of this section, we prove that the following theorem implies Theorem \ref{thm:main-theorem}.

\begin{repthm}{thm:main-theorem-translation-to-openbooks}
Let $L$ be a braid link in $S^3$. Then there is braid knot $K$ in $S^3$ and a covering $\varphi:S^3 \to S^3$ map branched along $K$ such that $\varphi$ maps disk pages into disk pages (of the corresponding standard open books of $S^3$ ) such that a positive stabilization of $L$ is a conjugate of some sublink of $\varphi^{-1}(K)$.   
\end{repthm}

We can think of this theorem as a translation of Theorem \ref{thm:main-theorem} from contact structures, and transverse knots into open-books, and braids. So, after Section \ref{sec:contact-and-openbooks} we expend most of our time proving Theorem \ref{thm:main-theorem-translation-to-openbooks}. 

First, in Section \ref{sec:description-with-graphs} we set up some notation to describe branch coverings between surfaces using graphs. These coverings will be the main tool to construct open-book branch coverings $\varphi: S^3 \to S^3$. We also describe how it is possible to modify these coverings through half-twists (see Subsection \ref{sec:half-twist-lifting}). By the end of this section, we construct a special branch covering $\varphi_m: S^3 \to S^3$ for every $ m\geq 8$ that admits a bunch of such special modifications.

In Section \ref{sec:constructing-branch-coverings} we prove Theorem \ref{thm:main-theorem-translation-to-openbooks}. In the first subsection (\ref{sec:arcs-gammai-definition}),  we describe another set of moves specially crafted to modify the branch covering map $\varphi_m$ in a finer way. The definition is again through a set of half-twist. In Subsection \ref{sec:constructing-any-braid} we show that these moves are enough to construct any given closed braid $L$ as a $2m$-strand sub-braid of $\varphi_m^{-1}(K)$ for some $m$. 

Finally, in Subsection \ref{sec:proof-of-main-theorem-translation-to-openbooks} we write the detailed proof of Theorem \ref{thm:main-theorem-translation-to-openbooks} in a constructive way. We split the proof into five steps. The first four steps create a closed braid link $K$, and the given link $L$ as a sub-braid of $\varphi_m^{-1}(K)$. The link $L$ is carefully setup inside of $\varphi_m^{-1}(K)$ for the fifth step to work. The final fifth step is just a simple cubic power of a half-twist that reduces the number of components of $K$ by one, and changes $L \subset \varphi_m^{-1}(K)$ only by a set of positive stabilizations.  This final step is what finally proves Theorem \ref{thm:main-theorem-translation-to-openbooks}.




\textbf{Acknowledgment}. Part of this work was carried out while the author was visiting the Mathematics Department at the University of Iowa and he wishes to thank them for their hospitality. Especially to Maggy Tomova for insightful conversations and guidance. He also likes to thank Roman Aranda, Thomas Kindred, Puttipong Pongtanapaisan, and Daniel Rodman for helpful discussions at the Topology Secret Seminar organized by Maggy. The author is greatly indebted to Puttinpong for letting him know about this problem. 

\section{From contact structures to open books}\label{sec:contact-and-openbooks}
An $n$-fold branch covering between a 3-manifold $N$ by another manifold $M$ is a continuous map $\varphi: M \to N$ such that there are links $L' \subset M$ and $L \subset N$ where $L' = \varphi^{-1}(L)$ and  $\varphi: M - L' \to N - L$ is an ordinary covering map. We will also include in our definition that, around each component of $L'$, there is a neighborhood $U = S^1 \times D^2$ such that $\varphi(U) = S^1 \times D^2$ and in these coordinates $\varphi$ is of the form $\varphi(\theta, z) =(\theta^k, z^m) $ for some positive integers $k$ and $m$. We will say that $\varphi$ is an \textit{$n$-sheet branch covering map} if $\varphi: M - L' \to N - L$ is an $n$-sheet covering map.

If $N$ has a contact structure $\xi$, it will be possible to give $M$ a contact structure ``induced" by $\varphi$ thanks to the following theorem.

\begin{thm}[J. Gonzalo \cite{Gonzalo-Branch-Covers},  H. Geiges \cite{geiges-article}, G\"{O}zt\"{u}rk and Niederkr\"{u}ger \cite{branch-covering-and-contact}]\label{thm:branch-covering-and-contact}
Let $\varphi:M \to N$ be a branched covering map with branch locus $L$. Given a contact structure $\xi$ on $N$ with contact form $\alpha$ such that $L$ is transverse to $\xi$, then $\varphi^{*}(\alpha)$ may be deformed by an arbitrarily small amount near $\varphi^{-1}(L)$ to give a contact form defining a unique, up to contact isotopy, contact structure $\xi_L$ on $M$.
\end{thm}

\begin{defn}\label{def:contact-branch-covering}
A branch covering map $\varphi:(M,\xi_L) \to (N,\xi)$  with the contact structure $\xi_L$ defined as above is called a \emph{contact branch covering} of $(N,\xi)$ along $L$.
\end{defn}

\subsection{Open Books}
An open book decomposition for a 3-manifold $M$ is given by a link $L$ embedded in $M$ such that $M - L $ fibers over $S^1$ with the condition that $\pi^{-1}(t)$ is a Seifert surface of $L$ for every $t \in S^1$; where $\pi: M-L \to S^1$ is the fibration map. We often say that $(L,\pi)$ is an open book decomposition of $M$.  The fibers $F_t = \pi^{-1}(t)$ are called \emph{pages}, and the link $L$ is called the \emph{binding} of the open book decomposition. The monodromy of the fibration $\pi$ is called the \emph{monodromy} of the open book decomposition. 

Given a surface $E$ and automorphism $f$ fixing $\partial E$, we can construct a 3-manifold $M_{(E,f)}$ by taking the mapping torus of $f$ and gluing some solid tori at the boundary (see \cite{Ozabci} for a more detailed description). The manifold $M_{(E,f)}$ has a natural open book decomposition given by $(E,f)$. The union of the cores of the glued tori is the binding, and the pages are copies of $E$. The function $f$ is the corresponding monodromy. We will refer to the pair $(E,f)$ as an \emph{abstract open book}.

\subsection{Compatible open books} \label{sec:compatible-open-books}

It has been shown by Giroux \cite{giroux2003g} that there is a correspondence between open books and contact-structures. The corresponding open book for a contact structure is often called a \emph{compatible} or \emph{supported} open book. 

We will refer to $(S^3, \xi_{std})$ as the standard contact 3-sphere. It is well known that a compatible open book decomposition for the standard contact 3-sphere is given by disk a fibration of a trivial knot transverse to $\xi_{std}$ with self linking $-1$. In this case, the pages are disks, and the monodromy is the identity. 

So, by taking any abstract open book $(E^2,f)$, we can construct a contact structure $\xi_{(E^2,f)}$ for the 3-manifold $M_{(E^2,f)}$ compatible with the natural open book decomposition given by  $(E^2,f)$. By Giroux's correspondence theorem \cite{giroux2003g}, all the possible contact structures $\xi_{(E^2,f)}$ are contact isotopic. In particular, if we take any automorphisms $f \in Aut(D^2, \partial D^2)$, the contact structure $\xi_{(D^2,f)}$ is contact isotopic to the standard contact structure $\xi_{std}$ of  $M_{(D^2,f)} \cong S^3$.

\subsection{Open books branch coverings}\label{sec:open-books-branch-coverings}

For this part, we are going to review some known facts about branch coverings and open books (for a detailed exposition see Casey's Thesis \cite[Section 3.4]{Casey}).  We also take the opportunity to introduce some handy notation, inspired by the non-branching case introduced in \cite{keiko}.  

Now, let us take a covering $\phi: \tilde{F}^2 \to F^2$ branched along some finite set of points $P = \{p_1, \dots, p_m\} \subset F^2$.   

\begin{defn}\label{def:open-book-branch-covering-map}
If there is an automorphism $\tilde{f} \in Aut(\tilde{F},\partial \tilde{F})$ satisfying that $$\phi \circ \tilde{f} = f \circ \phi$$ we will say that  $(\tilde{F}, \tilde{f})$ branch covers $(F,f)$. We will also say that $\phi: (\tilde{F}^2, \tilde{f}) \to (F,f)$ is an open book branch covering map.  
\end{defn}

On the above definition, the branch covering map is defining an actual branch covering map between 3-manifolds, namely $\varphi: M_{(\tilde{F}, \tilde{f})} \to M_{(F,f)}$ satisfying that its restriction to the pages is equivalent to $\phi$. The branching set $B$ of $\varphi$ is the projection of $P\times [0,1] \subset F \times [0,1]$ into $M_{(F,f)}$. Observe that $B$ is a 1-dimensional submanifold of $M_{(F,f)}$, which is transverse to the pages. In other words, $B$ is a link transverse to the pages of $M_{(F,f)}$.

In particular, when $F$ is a disk, $F = D^2$. Any open book branch covering over $(D^2, f)$ branches along a closed braid in $S^3 \cong M_{(D,f)}$.  

On the other way around, if we have any covering map $\varphi: M^3 \to S^3$ branched along some closed braid  $L$ in $S^3$, then we can give $M^3$ an "induced" open book decomposition. As a closed braid, $L$ is transverse to the pages of an open book decomposition $(D^2, Id)$ of $S^3$ (where $Id$ is the identity map). The way we can create an open book decomposition for $M$ is by using $\varphi^{-1}( \partial D^2)$ as a binding, and $F_t = \varphi^{-1}(D^2 \times t)$ as the pages. Which then makes $\varphi$ an open book covering map.

\subsection{Braids and transverse links}
There is a correspondence between braids in $S^3$ and transverse links in the standard contact 3-sphere (see Theorem \ref{thm:markov-contact} below). To construct the correspondence, we need to recall that $\xi_{std}$ restricted to $S^3 - \{point\}$ is contactomorphic to $(\mathbb{R}^3,\xi_{sym})$ where $\xi_{sym} = ker(xdy-ydx +dz) $. This means that any link transverse to $\xi_{std}$ can be identified with a transverse link in $(\mathbb{R}^3,\xi_{sym})$ and vice-versa. 

Now, it is easy to show that any braid can be drawn transversely to $\xi_{sym}$ by pushing it away from the $z$-axis. This way, any braid represents a transverse link. And the following theorem from Bennequin says that any transverse knot comes from a braid.

\begin{thm}[Bennequin, \cite{bennequin}]\label{thm:transverse-links-as-braids}
Any transverse link in $(\mathbb{R}^3,\xi_{sym})$ is transversely isotopic to a closed braid.
\end{thm}

It is natural to ask whether two braids represent the same transverse link. The answer is given in similar terms as the famous Markov's Theorem for braids \cite{markov-original} and links.  Before the statement of the theorem, we need to recall the following definition.

\begin{defn} 
Given a braid $B$ in the braid group $B_n$, we can get another braid $B' \in B_{n+1}$ by adding a crossing between the extra strand and the last strand of $B$. If the added crossing is positive (or negative) we say that $B'$ is a \emph{positive (or negative) stabilization} of $B$, or that $B$ is a \emph{positive (or negative) destabilization} of $B$.
\end{defn}

\begin{thm}[Orevkov and Shevchishin, \cite{markov-contact}] \label{thm:markov-contact} Two braids represent the same transverse link if and only if they are related by a sequence of positive stabilization/destabilization and braid isotopy.
\end{thm}

Let $(D^2, f)$ be an abstract open book and let  $\phi: (\tilde{F}^2, \tilde{f}) \to (D^2,f)$ be an open book branch covering map as in Definition \ref{def:open-book-branch-covering-map}.  We observed in Subsection \ref{sec:open-books-branch-coverings} that $\phi$ induces an actual branch covering map $\varphi: M_{(\tilde{F}, \tilde{f})} \to M_{(D^2,f)}$ branched along a closed braid $L \subset M_{(D^2,f)} \cong S^3$ transverse to all pages of the open book decomposition given by $(D^2,f)$. 

Now, we have seen that $L$ can be thought of as a transverse link in $(S^3, \xi_{std})$.  And we know that $\xi_{std}$ is compatible with the open book decomposition described by $(D^2, f)$. We can also take the contact structure $\xi_{(\tilde{F}^2, \tilde{f})}$ in domain space $M_{(\tilde{F}^2, \tilde{f})}$ of $\varphi$ that is compatible with the open book  $(\tilde{F}^2, \tilde{f})$ (see Subsection \ref{sec:compatible-open-books}). It is natural to ask if $\varphi: (M_{(\tilde{F}^2, \tilde{f})}, \xi_{(\tilde{F}^2, \tilde{f})}) \to ( M_{(D^2,f)}, \xi_{(D^2,f)}) \cong (S^3, \xi_{std})$ is a contact branch covering map. The answer is affirmative, and it is given in the following theorem.

\begin{thm}[M. Cassey  in \cite{Casey}]\label{thm:cassey-openbooks-and-contact-structures}
Let $L$ be a link braided transversely through the pages of the open book decomposition $(D^2, id)$, which supports $(S^3,\xi_{std})$. Let $(M,\xi)$ be the covering contact manifold obtained by branching over $L$. The open book constructed as described in Subsection \ref{sec:open-books-branch-coverings} supports the contact manifold $(M,\xi)$.
\end{thm}

In particular, if $\tilde{F}^2 \cong D^2$ then $M_{(\tilde{F}^2, \tilde{f})}$ is contactomorphic to $(S^3, \xi_{std})$. This way, we get a contact branch covering $(S^3, \xi_{std}) \to (S^3, \xi_{std})$. The construction of many branch coverings of this kind is crucial to prove Theorem \ref{thm:main-theorem}, and we can accomplish this through open book branch coverings $(D^2, \tilde{f}) \to (D^2, f)$.

\subsection{Proof of Theorem \ref{thm:main-theorem}}

To prove the theorem is enough to show that we can construct a contact covering $\varphi: (S^3, \xi_{std})  \to (S^3, \xi_{std})$  branched along some transverse knot $K$ such that $\varphi^{-}(K)$ contains a sublink transverse isotopic to the already known universal link $L$ (the one given by R. Casals and J. Etnyre in \cite{TUL}.). This way any contact manifold $(M^3, \xi)$ would be a contact covering of $(S^3, \xi_{std})$ branched along $K$, just by taking the composition of the already known contact covering $h:(M^3, \xi) \to (S^3, \xi_{std}) $ branched along $L$ followed by $\varphi$.

As we are looking for a branch covering between two $3-$spheres we can use open-books. We start by taking an open book branch covering between 2-disk $\phi: (\tilde{D}^2, \tilde{f}) \to (D^2, f)$ as explained before. This covering corresponds to a contact branch covering map given by $\varphi: (M_{(\tilde{D}^2, \tilde{f})}, \xi_{(\tilde{D}^2, \tilde{f})}) \to ( M_{(D^2,f)}, \xi_{(D^2,f)})$. In both spaces, the domain and codomain of $ \varphi$,  we have that their contact structures are supported by the standard open-book decomposition (the one with disk pages). This means that both spaces are contactomorphic to the 3-sphere with the standard contact structure.

This way of constructing branch covering between standard contact 3-spheres, naturally place the branching set $B \subset M_{(D^2,f)}$ as a closed braid. This braid can be described by decomposing $f$ as a composition of half-twists around arcs with ends in $P \subset D^2$ (see Subsection \ref{sec:open-books-branch-coverings}). And the same happens with the preimage of $\tilde{B} = \varphi^{-1}(B)$, it is naturally in braid form and its braid word is given by the decomposition of $\tilde{f}$ in half-twists along arcs with ends in $\phi^{-1}(P) \subset \tilde{D}^2$.


Observe that the braid $B$ and its preimage $\tilde{B}$ are transverse knots to the corresponding contact structures on the space where they live; which are in both cases tight contact spheres. We can then contact isotope this braids using Markov's like moves thanks to Theorem \ref{thm:markov-contact}. 

By Theorem \ref{thm:transverse-links-as-braids} we know that we can transverse isotope $L$ to be a closed braid $L'$.  We can then finish the proof if we can prove that $L'$ is a subbraid of some $\tilde{B}$. Remember that $\tilde{B}$ depends on $\phi: (\tilde{D}^2, \tilde{f}) \to (D^2, f)$, that can be chosen in many different ways (we will see that in Section \ref{sec:constructing-branch-coverings}).

Again, by Theorem \ref{thm:markov-contact}, we can relax the condition of being a subbraid to just being equivalent under positive stabilization/destabilization to a subbraid of some $\tilde{B}$. This is what we accomplish by proving Theorem \ref{thm:main-theorem-translation-to-openbooks}, but for any possible closed braid $L'$, not just the one obtained from $L$.

This completes the proof. For the rest of this document, we will focus on proving Theorem \ref{thm:main-theorem-translation-to-openbooks}.

\section{Description with graphs}\label{sec:description-with-graphs}
As we have seen in the previous section, a way to construct a contact branch covering $(S^3, \xi_{std}) \to (S^3, \xi_{std})$ is through open book branch covering maps between disks $(D^2, \tilde{f}) \to (D^2,f)$. In this section, we are going to study how to construct many of these such coverings. For that, we are going to use edge-colored graphs.

Let $G$ be a finite graph, we say that $G$ is edge-colored if all its edges are labeled, and consecutive edges (edges with a common vertex) are labeled differently. We will refer to labels as colors.  A coloring with the numbers $\{1, 2, \dots, n\}$ is called an $n$-edge-coloring. If there is no confusion, we will only say $n$-coloring instead of $n$-edge-coloring.

Given an $n$-colored graph $G$ with $m$ vertices, we can associate a presentation $\rho: F_n \to S_m$, where $F_n$ is the free group of rank $n$, and $S_m$ is the symmetric group on $m$ symbols. The presentation is constructed by assigning to each generator (a color of $G$) in $F_n$ a product of transpositions given by the edges. More formally, if $x_1, \dots, x_n$ is a set of generators of $F_n$, we define $\rho(x_i)$ as the product of all the transpositions $(a\ b)$ where $[a,b]$ is an $i$-colored edge in $G$. As no two consecutive edges have the same color, $\rho(x_i)$ is a product of disjoint transpositions. For example, for the graph in Figure  \ref{fig:disk-covering-example} the corresponding representation $\rho: F_4 \to S_6$ is the one given by $\rho(x_1) = (1\ 2)(3\ 4)$, $\rho(x_2) = (2\ 6)$, $\rho(x_3) = (2\ 3)$, and $\rho(x_4) = (2\ 5)$.

As the fundamental group of a disk minus $n$ points is a free group of rank $n$, $\rho$ can be thought of as a representation of  $\pi_1(D^2 - \{p_1, \dots, p_n\}) \to S_m$. Then we take the covering map $\hat{\phi}: \hat{E} \to D^2 - \{p_1, \dots, p_n\}$ corresponding to $\rho$.  Now, after taking the Stone-\v{C}hech compactification of $\hat{\phi}$, we obtain a covering map $\phi:E ^2 \to D^2$ branched along $\{p_1, \dots, p_n\}$.  We will say that $G$ is the graph representation of $\phi$. 

Observe that one important step to construct $\phi$ is to give an isomorphism between $\pi_1(D^2 - \{p_1, \dots , p_n\})$ and $F_n$. This isomorphism is determined by a set of generators of $\pi_1(D^2 - \{p_1, \dots , p_n\})$. So, to fix the covering $\phi$ corresponding to $G$ and remove any ambiguity  we choose the generators $\{\sigma_1, \sigma_2, \dots, \sigma_n\}$ depicted in Figure \ref{fig:disk-generators}.

\begin{figure}
    \centering
    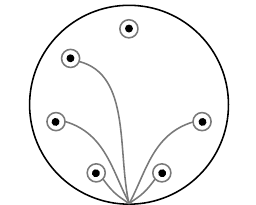
    \caption{Standard $n$-punctured disk generators}
    \label{fig:disk-generators}
\end{figure}

Observe that $G$ can be embedded in $E$.  Take the base point $V \in D^2$, this is the common point of the $\sigma$'s  shown in Figure \ref{fig:disk-generators}.  Then the vertices of $G$ correspond to $\phi^{-1}(V) = \{V_1, V_2, \dots, V_m\}$ and the $i$-colored edge connecting $V_a$ with $V_b$ will correspond to the lifts of the generator $\sigma_i$ at a vertex $V_a$. Basically, $G$ corresponds to the subset of $\phi^{-1}(\sigma_1 \vee \dots \vee \sigma_n)$ without lifts of $\sigma_i$ that starts and ends at the same vertex $V_a$ for some $a$.

Observe that $E$ strongly retracts to $G$. As a consequence, we have the following remark.

\begin{rem}
$E$ is a disk if and only if $G$ is a tree.
\end{rem}

From now on, we are going to focus on the case that $G$ is a tree. In order to visualize $E$ from the graph $G$, we first embed $G$ into the real plane with the following condition at each vertex $v \in G$ : when listing the colors of the edges incident to a $v$, starting from the edge of smallest color and in a clockwise direction, the resulting sequence is strictly increasing. After having $G$ embedded, just thicken it up (or more formally, take a small closed regular neighborhood of $G$ inside the plane) to get $E$. 

The graph $G$ in Figure \ref{fig:disk-covering-example-graph} can be thought of as an embedding into the Euclidean plane (thinking the paper sheet as the plane).  It satisfies the condition around the vertices that we said before. For example, around $V_2$, the list of colors of the edges is $1, 2, 3, 4$: remember that we need to start from the edge $[V_1, V_2]$ (which is the edge with the smallest color around $V_2$) and move in a clockwise direction around the edges incidents to $V_2$. After thickening $G$ up, we get the picture that is above the arrow in Figure \ref{fig:disk-covering-example-covering}. 

\begin{figure}
    \centering
    \begin{subfigure}[b]{0.5\textwidth}
        \centering
        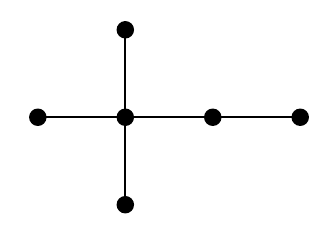
        \caption{Four colored graph with six vertices}
        \label{fig:disk-covering-example-graph}
    \end{subfigure}%
    
    \begin{subfigure}[b]{0.5\textwidth}
        \centering
        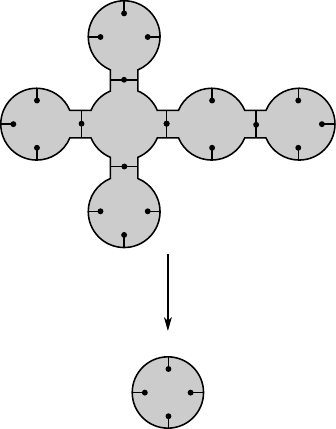
        \caption{Branch covering associated with the graph above}
        \label{fig:disk-covering-example-covering}
    \end{subfigure}%
    
    \caption{An example of a covering described by a graph.}
    \label{fig:disk-covering-example}
\end{figure}

A more precise way of getting the branch covering map, and not only the covering space, is by the following process. 

\begin{enumerate}
    \item We mark $n$ points in the base disk $D^2$ and label them as $1, 2, \dots n$ (remember that $n$ is the number of colors). Instead of taking any random set of points, we are going to choose a circular array of $n$ points as in Fig \ref{fig:disk-generators}. Then we label the points in clockwise order (the disk below the arrow in Figure \ref{fig:disk-covering-example-covering} is an example of a base disk with 4 colors). 
    \item Then we take as many copies of the base disk $D^2$ as the number of vertices of $G$, and place them over the vertices of $G$.
    \item  Later, for each edge $e = [V_j, V_k]$ of color $i$ in $G$ we do as follows.
     \begin{itemize}
         \item We cut out $D_j$ and $D_k$ (the disks at $V_j$ and $V_k$ respectively) along an oriented radial arc $\kappa_i$ coming out of the $i$-marked point and ending at the disk boundary (in Figure \ref{fig:disk-covering-example-covering}, the arcs $\kappa_1, \kappa_2, \kappa_3$, and $\kappa_4$ are depicted as the four small arcs connecting the points).
         \item After cutting along $\kappa_i$, we get two copies of $\kappa_i$ in $D^2_j$ (as well as in $D^2_k$). We name one copy as positive or negative according to the orientation of $D^2_j$ (as well as in $D^2_k$).
         \item  Finally we glue the positive arc in $D^2_j$ with the negative arc in $D^2_k$, and vice-versa.
     \end{itemize}
\end{enumerate}
  
At the end of this process, we get a surface $E$ contractible to $G$ and made of copies of the base $D^2$. As $G$ is a tree, $E$ is homeomorphic to a disk. But moreover, we get a quotient map $\phi: E \to D^2$ that identifies each copy $D^2_j \subset E$ with $D^2$. The map $\phi$ is a covering map branched along $n$ point in $D^2$. A complete example of this construction is shown in Figure \ref{fig:disk-covering-example}.

As consideration for the reader, we are going to draw $G$ embedded in the plane, taking care of the colors' order around the vertices as described above. This way, the thickening up of that embedding will correspond to the covering space. But, for most of the arguments, drawing $G$ in this way is not required.

We finalize this section with some notation that relates colored graphs with representations of free groups in the symmetric group. 

\begin{defn}\label{def:graph-and-permuations} In the following definitions, we denote with $S_m$ the symmetric group of $m$ elements. Let $\rho$ be an element of $S_m$, $R = \{\rho_1, \rho_2, \dots, \rho_n\}$ be a finite subset of $S_m$, and $\phi: F_n \to S_m$ be a representation of $F_n$ into $S_m$.  We now define the graphs $G_\rho$, $G_R$, and $G_\phi$.
\begin{enumerate}
    \item We denote as $G_\rho$ the directed graph associated to $\rho$. The graph consists of $m$ vertices $v_1, \dots v_m$, and directed edges $v_i \to v_j$ if and only if $\rho(i) = j$ with $i \neq j$. 
    
    If $\rho \in S_m$ is a product of disjoint transpositions, $G_{\rho}$ can be thought of as an undirected graph.
    
    \item We denote as $G_R$ the directed $n$-colored graph with $m$ vertices constructed by taking the union of directed graphs $G_{\rho_i}$ for all $i$. Where we colored each edge $v_i \to v_j$ with $k$ if and only if $\rho_k(i) = j$ with $i \neq j$. Multiple parallel edges may appear. 
    
    If $R$ is a set of permutations of order 2 (product of transpositions), then $G_R$ is an undirected $n$-colored graph.
    
    \item We denote $G_\phi = G_R$ where $R =\{\phi(\sigma_1), \dots, \phi(\sigma_n)\}$ where  $\{\sigma_1, \sigma_2, \dots, \sigma_n\}$ is a base of generators for $F_n$. So, $G_\phi$ depends on the choice of generators. 
\end{enumerate}

\end{defn}

Some important graphs that we are going to use are the following:

\begin{defn}\label{def:linear_graph_Lm}
We define the following graphs:

\begin{enumerate}
    \item A \emph{linear graph} is a connected tree graph (not necessarily colored) with valences no greater than 2. 
    \item We denote as $L_m$ to the $m$-colored linear graph with exactly $m+1$ vertices named $v_0, v_1, \dots,$ $v_m$, and $m$ edges $[v_i, v_{i+1}]$ with $i=0, \dots, m-1$. Each edge $[v_i, v_{i+1}]$ is colored with $i+1$.
    \item Let $L_{(m,2)}$ be a $2$-colored linear graph with exactly $m$ edges. Using the same notation as above, we color each edge $[v_i,v_{i+1}]$ as 1 if $i+1$ is odd or 2 otherwise. 
    
\end{enumerate}
\end{defn}

\subsection{Half-twist lifting}\label{sec:half-twist-lifting}
Recall that we want to construct open book branch coverings between disks (as in Definition \ref{def:open-book-branch-covering-map}) We already know how to construct many branch covering maps $\phi: D^2 \to D^2$ using graphs. What is missing to get an open book branch covering is a pair of disk automorphism $f,\tilde{f}$ such that $f\circ \phi = \phi \circ \tilde{f}$. That is why we are interested in automorphisms $f:D^2 \to D^2$ where such $\tilde{f}$ exists. 
\begin{defn}
Let $\phi: \tilde{E}^2 \to E^2 $ be a branch covering map with branching set $P = \{p_1, \dots, p_n\} \subset E^2$. We will say that a map $f:(E,P) \to(E,P)$ is \emph{liftable} if there is map $\tilde{f}: \tilde{E}^2 \to \tilde{E}^2$ such that $f\circ \phi = \phi \circ \tilde{f}$. 

We may also say that $f$ \emph{lifts} to $\tilde{f}$ with respect to $\phi$. When $\phi$ can be inferred from the context, we will often say that $f$ \emph{lifts} to $\tilde{f}$ or we may simply say $f$ lifts.
\end{defn}

We are particularly interested in liftable automorphisms $f:(D^2,P) \to (D^2,P)$ of a punctured disk. It is a well-known fact that any automorphism $f$ of the punctured disks $D^2$ is homotopic to a product of half-twist (see definition below) relative to $\partial D^2$. Then, it is natural to ask whether or not a half-twist (or a power of it) is liftable.

A \emph{half-twist} is a map that takes an arc $\alpha$ on a surface $E$ and twists everything around a neighborhood of $\alpha$ (a disk) and leaves the rest of $E$ intact (see Figure \ref{fig:arc-twist}). 

To define half-twist maps more precisely, one must decompose $E$ as $E = D \cup \hat{E}$ where $D$ is a regular neighborhood of $\alpha$ (homeomorphic to a disk) and $\hat{E}=\overline{E - D}$. Then, we define a function $c_\alpha: E \to E$ that is the identity on $\hat{E}$ and on $D$ does the following:
\begin{enumerate}
    \item Rotates a half-turn an embedded disk containing $\alpha$ and disjoint from the boundary,
    \item In the annulus that remains, it is the function $(\theta, t) \to (\theta + t\pi/2, t)$. 
\end{enumerate}

\begin{figure}
    \centering
    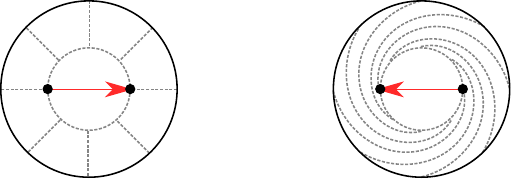
    \caption{Arc twist image}
    \label{fig:arc-twist}
\end{figure}

Let $\phi$ be the branch covering given by a graph $G$, we would like to give a criterion to decide if a power (usually a square or a cube) of a half-twist lifts with respect to $\phi$.

Let us start with the simplest example that will be the source of many others. Let $L_2$ be the 2-colored connected graph made of three vertices and two edges (see Definition \ref{def:linear_graph_Lm} ). Consider the 3-fold branch covering map $\phi_3:D^2 \to D^2 $ represented by $L_2$ (see Figure \ref{fig:example-L2-phi3}).  The map $\phi_3$ is branched along two points in $D^2$; say $a, b \in D^2$. Take $\alpha$ a simple arc connecting $a$ and $b$. It is known (see for example \cite{some_universal_knots}) that $\tau_{\alpha}^3$ lifts to a half-twist along the arc $\tilde{\alpha} =\phi_3^{-1}(\alpha)$. This observation can be generalized as stated on the following lemma.
\begin{figure}
    \centering
    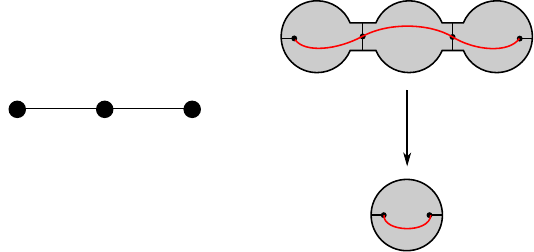
    \caption{Left: The linear graph $L_2$. Right: The 3-fold branch covering $\phi_3$ represented by $L_2$. The arc $\widetilde{\alpha}$ is the preimage of $\alpha$ under $\phi_3$ .}
    \label{fig:example-L2-phi3}
\end{figure}

\begin{lem}\label{lem:on-Lm_m-powers-lifts}
Let $\phi: \tilde{D}^2 \to D^2$ be the $m$-fold branch covering between 2-disks represented by the graph $L_{(m,2)}$ (see Definition \ref{def:linear_graph_Lm} ) with $m \geq 2$. Let $\alpha$ be an arc connecting the two branching points $a, b \in D^2$ of $\phi$. Then, there is a map $T:\tilde{D}^2 \to \tilde{D}^2$  such that $\tau^k_{\alpha} \circ \phi = \phi \circ T $ if and only if $k$ is multiple of $m+1$.

Moreover, the function $T$ is homotopic to the $k/(m+1)$-th power of a half-twist around the arc $\phi^{-1}(\alpha)$. 
\end{lem}
\begin{proof}
Consider the covering map resulting from removing the branching set $\{a,b \}$, $\phi: \tilde{D}^2 - \phi^{-1}(\{a,b\}) \to D^2 -\{a,b\}$. Now we make use of the lifting theorem for covering maps. This means we need to take the image of $F = \pi_1(\tilde{D}^2 - \phi^{-1}(\{a,b\}))$ under  $\tau^{m+1}_{\#} \circ \phi_{\#}$ and prove that it is a subgroup of $\phi_{\#}(F)$. This is enough to show that $T:\tilde{D}^2 \to \tilde{D}^2$ exists for every $k$ multiple of $m+1$. 

Now, we need to show that if $k$ is not a multiple of $m+1$, then there is no such $T$. For that, take any point $P$ on the boundary of the base disk $D^2$ and connect it with its antipodal $Q$ using a diameter $\beta$ passing between $a$ and $b$. 

Let $P_1, P_2, \dots, P_{m+1}$ be the preimages of $P$ under $\phi$. These preimages belong to the boundary of the domain of $\phi$, the disk $\tilde{D}$. By reindexing, we can make the sequence $P_1, P_2, \dots, P_{m+1}$ appear in that order when traveling around $\partial \tilde{D}$. For every $i$, we name $\beta_i$ to be the component of $\phi^{-1}(\beta)$ that contains $P_i$; the arc $\beta_i$ is also known as the lift of $\beta$ at $P_i$. We call $Q_i$ the other end of $\beta_i$.

Observe that the arc $\overline{\beta} = \tau^k_{\alpha}(\beta)$ has the same ends as $\beta$. So, $\phi^{-1}(\overline{\beta})$ will be a set of arcs connecting $P_i$'s with  $Q_i$'s. This creates a permutation of the sub-indexes $i$'s. By drawing the covering, we can observe that the permutation is exactly $i \to i+k$; the lift of $\overline{\beta}$ at $P_i$ ends at $Q_{i+k}$. If $\tau^k_{\alpha}$ lifts, then the mentioned permutation has to be the identity. This proves the first part.

To prove that $T$ is a half-twist around $\gamma = \phi^{-1}(\alpha)$ we proceed as follows. Observe fist that $T$ leaves $\gamma$ invariant; in fact, $T|_{\gamma}$ is either the identity or a reflection depending on the parity of $k$. Now, let $\tilde{A}$ and $A$ be the two annuli obtained after cutting $\tilde{D^2}$ along $\gamma$ and $D^2$ along $\alpha$, respectively. After cutting, we obtain an induced covering map $\hat{\phi}: \tilde{A} \to A$ and two induced automorphisms $\hat{T}: \tilde{A} \to \tilde{A}$ and $\hat{\tau}^k_{\alpha}: A \to A$ such that $\hat{\phi} \circ \hat{T} = \hat{\tau}^k_{\alpha} \circ  \hat{\phi}$.  


Observe now that we can find a coordinate system for  $A \cong S^1 \times I$ of the form $(\theta, t) $ where $\theta \in [0, 2\pi]$ and $t \in [0,1]$, such that $\hat{\tau}^k_{\alpha}(\theta, t) = (\theta + t\cdot k \cdot \pi, t)$. And we can lift the coordinate system through $\phi$ to a coordinate system $(\tilde{\theta}, \tilde{t})$ for $\tilde{A} \cong S^1 \times I$. With this new coordinate system, we have that  $\phi(\tilde{\theta}, \tilde{t}) =  ((m+1)\cdot \tilde{\theta}, \tilde{t})$. By the fact that $\hat{\phi} \circ \hat{T} = \hat{\tau}^k_{\alpha} \circ  \hat{\phi}$ we got that $$\hat{T}(\tilde{\theta}, \tilde{t}) = (\tilde{\theta}+ \frac{\tilde{t}\cdot k \cdot \pi}{(m+1)}, \tilde{t})$$

As we already prove, $k$ is a multiple of $m+1$,  which means that $\hat{T}$ is a Denh twist around the annulus $A$ with angle  $\frac{k}{m+1} \cdot \pi $.  By taking the quotient $\tilde{A} \to \tilde{D}^2$, we can prove that $T$ is the $\frac{k}{m+1}$-th power of a half-twist around $\gamma$. 
\end{proof}

In the previous lemma, we can also include $L_1$, which is 1-colored instead of 2-colored as the rest. In this case, the corresponding covering $\phi_2: \tilde{D}^2 \to D^2$ branches along one point and it satisfies that  $\tau^2_{\alpha}$ lifts to a half-twist around $\phi^{-1}(\alpha) = \tilde{\alpha}$. Here, $\alpha$ is an arc connecting any point in the interior of $D$ with the branching point.

In an even more general setting, let us take a branch covering $\phi: E^2 \to D^2$ with a tree $G$ as graph representation. Let $\{p_1, p_2, \dots, p_n\} \subset D^2$ be the branching set of $\phi$. And let $\alpha$ be a properly embedded arc in $D^2 - \{p_1, \dots, p_n\}$ connecting two points $p_i$ and $p_j$. We would like to know all numbers $n \in \mathbb{Z}$ such that $\tau_{\alpha}^n$ lifts to $E^2$ under $\phi$. 

Let $\beta_1, \beta_2, \dots, \beta_k$ be the connected components of  $\phi^{-1}(\alpha)$. For each $\beta_i$, the map $\phi |_{\beta_i}$ covers $\alpha$ in such a way that the preimage of every point $p$ in the interior of $\alpha$ has the same number of points. We call this number the degree of $\beta_i$ and we will denote it as $deg(\beta_i) = |\phi^{-1}(p) \cap \beta_i|$ for any $p \in Int(\alpha)$. 

\begin{prop}\label{prop:lifting-of-tk-and-degree}
On the setting described above, $\tau^N_{\alpha}$ lifts if and only if $N$ is a multiple of $deg(\beta_i)$ for all $i$.
Moreover, the lifting of $\tau^N_{\alpha}$ is the composition of $\tau_{\beta_i}^{n_i}$ where $n_i = N/deg(\beta_i)$ for  $i=1, \dots, k$. 
\end{prop}
\begin{proof}
Let $B = N(\beta_i)$ be a regular neighborhood of $\beta_i$. Observe that $A=\phi(N(\beta_i))$ is a regular neighborhood of $\alpha$.  Now, $\phi |_{B}:B \to A$ is a covering map branched along two points (the ends of $\alpha$), the number of sheets of $\phi|_B$ is exactly $deg(\beta_i)$. If $\tau^N_{\alpha}$ lifts under $\phi$ it will lift under $\phi|_B$, and in the other way around.  So, it will be enough to prove that $\tau^N_{\alpha}$ lifts under $\phi|_B$ if and only if $N$ is a multiple of $deg(\beta_i)$. 

As the graph representation $G$ of the branch covering is a tree, then $\phi|_B: B \to A$ is also represented with a tree $H$. As there are only two branching points in $A$, it means that $H$ has to be a 2-colored graph. This implies that the valence at any vertex of $H$ is lower or equal than 2. This makes $H$ isomorphic to the linear graph $L_{(2,m)}$ (see Definition \ref{def:linear_graph_Lm}) where $m =  deg(\beta_i)-1$.  Then the result follows from Lemma \ref{lem:on-Lm_m-powers-lifts}, which states as well, that the lifting of $\tau^{N}_{\alpha}$ is $\tau^{n_i}_{\beta_i}$
\end{proof}

The following corollary is an immediate consequence of the previous proposition.

\begin{cor}\label{cor:lifting-conditions}
Let $G$ be an $n$-colored graph that represents a covering map $\phi: E^2 \to D^2$, and let $\alpha$ be the arc on $D^2$ shown in Figure \ref{fig:arc-alpha-connecting-1-and-2}. Then 

\begin{enumerate}
    \item[a)] $\tau_{\alpha}^3:D^2 \to D^2$ lifts to a map on $E^2$ if the connected components of the $\{1,2\}$-subgraph of $G$ are isolated vertices or isomorphic to $L_2$. 
    \item[b)] $\tau_{\alpha}^2:D^2 \to D^2$ lifts to a map on $E^2$ if the connected components of the $\{1,2\}$-subgraph of $G$ are isolated vertices or isomorphic to $L_1$. 
\end{enumerate}
\end{cor}

\begin{figure}
    \centering
    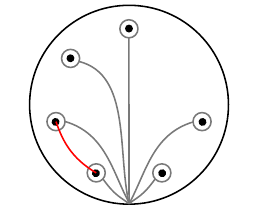
    \caption{Arc $\alpha$ connecting points $1$ and $2$.}
    \label{fig:arc-alpha-connecting-1-and-2}
\end{figure}
\begin{proof}

Let $N_\alpha$ be a small regular neighborhood around $\alpha$. Let $M = \phi^{-1}(N_\alpha)$ be the preimage of $N_\alpha$ under $\phi$. We will show that the graph representation of $\phi|_M$ is the $\{1,2\}$-subgraph of $G$. For that, recall that $G$ represents a function $\rho: \pi_1(D^2-\{p_1, p_2, \dots, p_n\}) \to S_m$ (see Definition \ref{def:graph-and-permuations}) where the image of the generators $\rho(\sigma_1), \rho(\sigma_2), \dots,  \rho(\sigma_n)$ correspond to edge-colors of $G$. In particular, $\rho(\sigma_1)$ and $\rho(\sigma_2)$ correspond to 1 and 2 colored edges of $G$. 

Now, to construct the graph representation of $\phi|_M: M \to N_\alpha$, we need to choose a set of generators for $\pi_1(N_\alpha - \{p_1, p_2\})$. Observe first that our choice of $\alpha$ satisfies that $\sigma_1$, $\alpha$, and $\sigma_2$, bound a triangular region $T \subset N_\alpha$. Let us take an arc $\gamma \subset T$ connecting a point $* \in \partial N_\alpha$ with the base point $\sigma_1 \cap \sigma_2$. As $T$ is contractible, the conjugations $\sigma_1^\gamma = \gamma \cdot \sigma_1 \cdot \gamma^{-1}$ and  $\sigma_2^\gamma = \gamma \cdot \sigma_2 \cdot \gamma^{-1}$ are isotopic to a pair of generators $\sigma'_1$and $\sigma'_2$ for $\pi_1(N_\alpha - \{p_1, p_2\})$. This immediately implies that the representation of $\phi|_M$ is equivalent to $\rho': \pi_1(N_\alpha - \{p_1, p_2\}) \to S_m$ that sends $\rho'(\sigma'_i)=\rho(\sigma_i)$ for $i=1,2$. 

Part (a) condition on the components of $\{1,2\}$-subgraph $G$ implies that $\phi$ restricted into connected components of $M$ leads to a homeomorphism or an $L_2$ branch covering of $N_\alpha$ . Then part (a) follows from Proposition \ref{prop:lifting-of-tk-and-degree}.  The proof for part (b) is similar.
\end{proof}

For the first part of the corollary, we can rephrase the condition on the $\{1,2\}$-subgraphs of $G$ by saying that every $1$-colored edge must be connected with one and only one $2$-colored edge; and in the other way around. And for the second part by saying that $1$-colored edges are disjoint from $2$-colored edges.

 The fact that $\alpha$, $\sigma_1$, and  $\sigma_2$ bound a region disjoint from the rest of the generators is the main reason because the $\{1,2\}$-subgraphs of $G$ represents the covering $\phi|_M:M \to N_\alpha$. 

Using the previous observations, we can see now that Corollary \ref{cor:lifting-conditions} gives us a condition to recognize the lifting of a square or a cube of a half-twist around any arc $\alpha$ connecting two branching points. It is enough to change the generators of $\pi_1(D^2-\{p_1, p_2, \dots, p_n\})$ in such a way that $\alpha$, together with the two generators around its endpoints enclose a triangular region disjoint from the other generators (see Figure \ref{fig:triangular-region}). Then, with the new base, we can draw the corresponding graph and check if the graph satisfies the conditions on Corollary \ref{cor:lifting-conditions}.

In particular, we can take the straight arc $\alpha_i$ that connects $p_i$ and $p_{i+1}$ and enclose a triangular region together with the generators $\sigma_i$ and $\sigma_{i+1}$. Then, if the $\{i, i+1\}$-subgraph of $G$  has only connected components isomorphic to $L_2$ or isolated points, then $\tau_{\alpha_i}^3$ lifts; similarly, if the connected components are isolated points or $L_1$ graphs, $\tau_{\alpha_i}^2$ lifts.

For an arbitrary arc, different from the arcs $\alpha_i$'s, it is not possible to just look at a subgraph of $G$. We actually need to modify $G$ accordingly using a new suitable set of generators for $\pi_1(D^2-\{p_1, p_2, \dots, p_n\})$ . But fortunately, we don't need to draw the full new graph $G$. We only need the subgraph corresponding to the generators containing the mentioned triangular region. This subgraph can be easily constructed as a 2-colored graph, the edges corresponding to the permutations given by the image of the two generators under $\rho: \pi_1(D^2-\{p_1, p_2, \dots, p_n\}) \to S_m$.

\begin{figure}
    \centering
    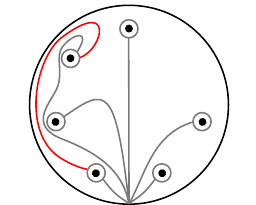
    \caption{Triangular region enclosed by $\alpha$ and the two generators around its endpoints}
    \label{fig:triangular-region}
\end{figure}

To summarize, given an arc $\alpha$ connecting two branching point $p$ and $q$ in\\ $\{p_1, \dots, p_n\} \subset D^2$ first we need to find a pair of arcs $\gamma_p$ and $\gamma_q$, such that:
\begin{enumerate}
    \item $\gamma_p$ connects $p$ and the base point
    \item $\gamma_q$ connects $q$ and the base point
    \item $\gamma_p$, $\gamma_q$, and $\alpha$ enclose a triangular region from $D^2$ without any other branching point on its interior.
\end{enumerate}
After finding those arcs, we use $\gamma_p$ to create a base generator $\sigma_p$ that travels along $\gamma_p$ until it almost touches $p$ and then loops around a small neighborhood of $p$ and returns back to the base point; similarly, we define $\sigma_p$. Then, we can create the graph $G_{\{p,q\}}$ associated with the permutations $\rho(\sigma_p)$ and $\rho(\sigma_q)$ as described in Definition \ref{def:graph-and-permuations}; here $\rho: \pi_1(D^2-  \{p_1, \dots, p_n\})$ is the presentation associated with the branch covering $\phi: E^2 \to D^2$

 \begin{prop}\label{prop:lifting-alpha-twist-conditions}
 Let $\alpha$ and $G_{\{p,q\}}$  as above, then
 \begin{enumerate}
    \item $\tau_{\alpha}^3:D^2 \to D^2$ lifts to a map on $E^2$ if the connected components of the $G_{\{p,q\}}$ are either isolated vertices or isomorphic to $L_2$. 
    \item $\tau_{\alpha}^2:D^2 \to D^2$ lifts to a map on $E^2$ if the connected components of the $G_{\{p,q\}}$ are either isolated vertices or isomorphic to $L_1$. 
 \end{enumerate}
 
 \end{prop}

\subsection{Branch coverings with liftable cube half-twists}
In this subsection we give a method to construct infinitely many coverings $\phi: D^2 \to D^2$ branched along $n$ points with the property that the triple half-twist around the arcs $\alpha_1, \dots, \alpha_n$ lifts to the covering space; the arcs $\alpha_1, \dots, \alpha_n$ are the straight arcs connecting consecutive points on the base disk (see Figure \ref{fig:arc-alpha-connecting-1-and-2} for a picture of $\alpha_1$).

As we saw on Proposition \ref{prop:lifting-alpha-twist-conditions}, the lifting of $\tau^3_{\alpha_i}$ can be decided by looking at the $\{i,i+1\}$-subgraph of $G$ associated with the covering $\phi$. This can be rephrased as follows:

\begin{prop}
If every  $i$-edge of $G$ is always next to one and only one $i+1$-edge if $i<n$ edge and with one and only one $i-1$-edge if $i >1$, then each $\tau^3_{\alpha_i}$ lift to the covering space associated with $G$.
\end{prop}

For example, the graph $L_n$ defined in \ref{def:linear_graph_Lm}  (see Fig  \ref{fig:example-L4} for an example when $m=4$) satisfies this condition for every $i=1, \dots, n$. This means that $\tau^3_{\alpha_i}$ lifts for every $i$. Observe that, in this case,  the lift of $\tau^3_{\alpha_i}$ consists of the composition of one half-twist around one arc components in $\phi^{-1} (\alpha_i)$ and the cubes of half-twist around the other arc components of $\phi^{-1} (\alpha_i)$.

\begin{figure}
    \centering
    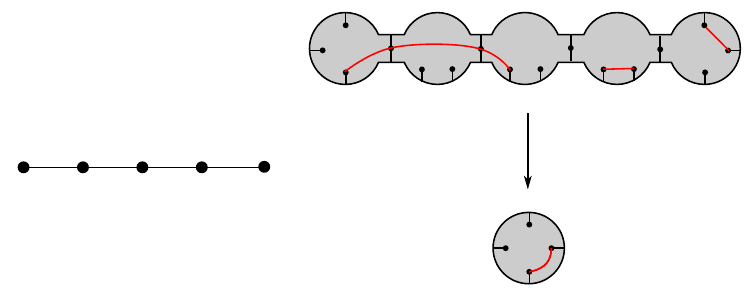
    \caption{The graph $L_4$ and its correspondent branched covering $\phi_5$. The connected components of $\phi_5^{-1}(\alpha)$ are $\widetilde{\alpha}, \overline{\alpha}_1$ and $\overline{\alpha}_2$. And the degrees are $deg(\widetilde{\alpha})=3$, and $deg(\overline{\alpha}_1) = deg(\overline{\alpha}_2) = 1$. The lift of $\tau^3_{\alpha}$ is the function $\tau_{\widetilde{\alpha}} \circ \tau^3_{\overline{\alpha}_1} \circ \tau^3_{\overline{\alpha}_2}$ }
    \label{fig:example-L4}
\end{figure}

\begin{defn}
An $n$-colored graph $G$ will be called $n$-\emph{consecutive-colored} if it satisfies that:
\begin{enumerate}
    \item Every $i$-colored edge is adjacent with one and only one $i+1$ edge if $i<n$,
    \item and is adjacent with one and only one $i-1$ if $i>1$.
\end{enumerate}
\end{defn}

Another way of putting this definition is by saying that the components of the $\{i,i+1\}$- subgraph of $G$ (the subgraph formed by the $i-$ and $(i+1)-$colored edges of $G$) are either isomorphic to $L_2$ or isolated points for all $i<n$.

The following statement will allow us to construct infinitely many $n$-consecutive-colored graphs. 

\begin{prop}\label{prop:n-consecutive-colored-graph-conditions}
Let $G_1$ and $G_2$ be two $n$-consecutive-colored graphs with marked vertices $v_i \in G_i$ for $i=1,2$ satisfying that no color around $v_1$ is consecutive or equal to a color around $v_2$. Then, the quotient graph $$G = \frac{G_1 \cup G_2 }{v_1 =v_2}$$
is an $n$-consecutive-colored graph as well.
\end{prop}
\begin{proof}
We only need to observe that around the vertex $v_1 = v_2$, we don't break the condition of $n$-consecutive-colorability. It is obvious that all $i$-edges keep together with at least one $(i+1)$-edge, but now they could be connected with one extra $(i+1)$, but that can only happen on the vertex $v_1 = v_2$. But the condition on them guarantees that no new pair of connected consecutive colored edges are created.
\end{proof}

Now, by taking copies of $L_n$ and identifying them along vertices with no consecutive colors or the same color around them, we can create many $n$-consecutive-colored graphs like the one shown in Figure \ref{fig:n-consecutive-colored-graph-example}.

\begin{figure}
    \centering
    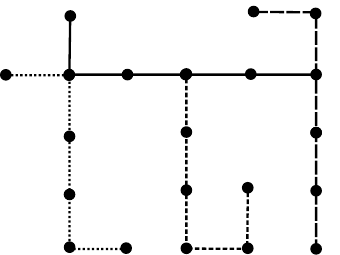
    \caption{A $5$-consecutive-colored graph made of 4 copies of $L_5$}
    \label{fig:n-consecutive-colored-graph-example}
\end{figure}
 
From the example shown in Figure \ref{fig:n-consecutive-colored-graph-example}, we can see that it is possible to create trees for any $n \geq 3$. It is clear that when $n=2$, there is only one connected tree. When  $n=3$, there are infinitely many threes, but they are all line graphs made of $L_3$'s. When $n \geq 4$, we can find a much bigger variety of trees. This kind of tree is the one that we will use in the following section.

\subsection{Symmetric graphs}
Colored graphs have a special kind of symmetry.  For example, we can observe that if we reflect the graph $G$ shown Figure \ref{fig:color-symmetric-graph-example} along a vertical line passing through the middle, we will get a transformation from $G$ onto $G$. Unfortunately, this transformation does not preserve the coloring. Luckily, it maps edges of the same color to edges of the same color (for instance, the two edges of color 4 are mapped to edges of color 1). We will say that $G$ is color-symmetric; the formal definition is as follows.

\begin{figure}
    \centering
    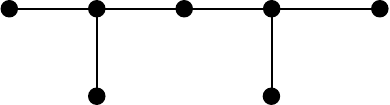
    \caption{An example of a color-symmetric graph}
    \label{fig:color-symmetric-graph-example}
\end{figure}{}

\begin{defn}
An $n$-colored-graph $G$ is going to be \emph{color-symmetric} if:
\begin{enumerate}
    \item There is a permutation $s: V(G) \to V(G)$ with $s(s(v)) = v$, and such that $e =[a,b]$ is an edge of $G$ if and only if $[s(a), s(b)]$ is an edge of $G$ as well. We write $s(e)=[s(a), s(b)]$.
    \item There is a permutation of colors $f: \{1, 2, \dots, n\} \to \{1, 2, \dots, n\}$ such that $color(s(e)) = f(color(e))$.
\end{enumerate}
We will refer to $s$ as a \emph{reflection} of $G$ and, to $f$, as a \emph{color-reflection} of $G$. 
\end{defn}

A simpler example is the graph $L_m$ which is color-symmetric with color-reflection $f(i) = m+1-i$.

\begin{prop}
If $H$ is an $n$-colored graph with a symmetry $s$ and color-reflection $f$. Take $v_0$ a vertex on $H$, and define $G$ as the quotient
$$G = \frac{H \cup H}{v_0 \sim s(v_0)}$$

Then $G$ is color-symmetric with $f$ as a color-reflection.
\end{prop}
\begin{proof}
Let $s': V(H \cup H ) \to V(H\cup H)$ be the function that sends a vertex $v$ in any copy of $H$ to $s(v)$ on the other copy. Thanks to the fact that $s$ has order two, it is clear that $s'$ descends to the quotient, it fixes the vertex $v_0$, and it is a symmetry for $G = H \cup H / (v_0 \sim s(v_0))$.  Like no new edges are created, the function $s'$ sends an edge $e$ in one copy of $H$ to an edge $s(e)$ on the other, this implies that $G$ is also color-symmetric with respect to $f$.
\end{proof}

The following more general result can be proven using the same type of arguments as above.
\begin{prop}\label{prop:creating-symmetrical-graphs-by-gluing}
Let $G$ and $H$ be two color-symmetric $n$-colored graphs with symmetry $s$ and $s'$, respectively, and both with the same color-reflection $f$. Now take two vertices $v \in H$ and $w \in G$  and construct the graph
$$K = \frac{(H\times 0) \cup (G \times 1)\cup (H\times 2)}{(v,0) \sim (w,1), (s(w),1) \sim (s'(v),2)}$$

Then $K$ is color-symmetric with reflection $r$ and color-reflection $f$, where 
\begin{equation*}
r(x,i) = \left\{
\begin{array}{ll}
     (s'(x),2)& \textrm{ if } \ i=0 \\
     (s(x),1)& \textrm{ if } \ i=1\\
     (s'(x),0)& \textrm{ if }\  i=2
\end{array}
\right.    
\end{equation*}{}

\end{prop}
\begin{proof}
 First of all, we have to show that $r$ is well defined and it is a symmetry of $K$. Let $q$ be the quotient function that maps $\hat{K} = (H\times 0) \cup (G \times 1)\cup (H\times 2)$ onto $K$.  Let $\rho$ be the permutation of the vertices $V(\hat{K})$ given by  
 \begin{equation*}
\rho(x,i) = \left\{
\begin{array}{ll}
     (s'(x),2)& \textrm{ if } \ i=0 \\
     (s(x),1)& \textrm{ if } \ i=1\\
     (s'(x),0)& \textrm{ if }\  i=2
\end{array}
\right.    
\end{equation*}{}

Observe that $q \circ \rho(v, 0) = q \circ \rho(w,1) $ and  $q \circ \rho(s(w), 1) = q \circ \rho(s'(v),2) $, so $\rho$ descent as $r$ to the quotient space $V(K)$, this means that $q \circ \rho =  r \circ q$.  Clearly, $\rho \circ \rho = I_{\hat{K}}$ (the identity on  $V(\hat{K})$), so $r \circ r = I_{V(K)}$.  As for the edges, we observe that $[x,y]$ is an edge of $\hat{K}$ if and only if $[\rho(x), \rho(y)]$  is an edge too; this is because $\rho$ restricted to a copy of $H$ ($H\times i$, with $i =0, 2$) behaves as graph isomorphisms between $H \times 0$ and $H \times 2$, and restricted to $G \times 1$ is a reflection. Then, it follows that $[q(x), q(y)]$ is an edge of $K$ is equivalent to $[r(x), r(y)]$ is an edge of $K$. So, $r$ is a symmetry of $K$.

Finally, to check that $r$ satisfy the relation $f \circ color(e) = color ( r(e) ) $ for every edge in $K$ it is enough to check that  $f \circ color(e) = color ( \rho(e) ) $ for every edge $e$ in $\hat{K}$. To prove the identity, we can proceed by cases. For example, for $e = [(x,0), (y,0)]$ in $H \times 0$ the color of $\rho(e) = [(s'(x), 2), s'(y),2]$ is the color of $[s'(x), s'(y)] = s'([x,y])$ in $H$, so it satisfy  that $color (\rho(e)) = color(s'([x,y])) = f(color[x,y]) = f (color(e))$. We can solve the other cases analogously.

\end{proof}

\subsection{The main disk branch covering: the graph $G_m$}

In this section, we describe the cornerstone of the construction of a universal transverse knot. This is a branch covering $\phi: D^2 \to D^2$ with a set of strategically selected arcs on both disks (the base and covering disk spaces) that will allow us to construct any braid on the above disk by means of twisting around the arcs on the base disk.

As we explained in the previous sections, we can describe disk coverings by means of an $m$-colored tree. The tree that we are going to use is constructed by taking copies of the linear $m$-colored graph $L_{m}$ (see Definition \ref{def:linear_graph_Lm}) and gluing them as follows:

\begin{defn}\label{def:graph_gm}Given any integer $m \geq 8$, we denote as $G_m$ the graph constructed by taking $L^1, L^2, \dots , L^{4m+2}$, $4m+2$ copies of the linear graph $L_{2m}$, and identifying them as follow:
\begin{enumerate}
    \item[I1:] For $i=1, 2, \dots, m-1$ identify $v_{i-1} \in L^i$ with $v_{i+2}  \in L^{i+1}$.
    \item[I2:] For $i=m+1, m+2, \dots, 2m-1$ identify $v_{i-2} \in L^i$ with $v_{i+1}  \in L^{i+1}$.
    \item[I3:] Identify $v_2 \in L^1$ with $v_{2m-2}\in L^{2m}$.
    \item[I4:] For $i =1, 2, \dots, m$ identify $v_i \in L^i$ with $v_{2m-5} \in L^{2m+i}$
    \item[I5:] For $i =m+1, m+2, \dots, 2m$ identify $v_{i-1} \in L^i$ with $v_{5} \in L^{2m+i}$
    \item[I6:] Identify $v_0 \in L^1$ with $v_6 \in L^{4m+1}$ and $v_{2m-6}  \in L^{4m+2}$ with $v_{2m} \in L^{2m}$
\end{enumerate}
\end{defn}

\begin{prop}\label{prop:properties-of-G_m}
The graph $G_m$ satisfies the following properties:
\begin{enumerate}
    \item It is a connected tree.
    \item It is a $2m$-colored graph 
    \item It is symmetric with color-reflection $f(i) = 2m+1 -i$.
    \item It has the $2m$-consecutive-coloring property.

\end{enumerate}
\end{prop}
\begin{proof}
To show that $G_m$ is connected, we do as follows. First, observe that I1 creates a connected graph $G^1$ with the $m$ copies $L^1, L^2, \dots, L^m$. Now, let $G^2$ be the connected graph created with ($L^{m+1}, L^{m+2}, \dots, L^{2m}$) applying I2-identifications. Then, I3 connects $G^1$ and $G^2$, creating a new graph $G^3$ that contains the first $2m$ copies of $L_{2m}$. The I3- and I4-identifications connect the graphs $L^{2m+1}, \dots , L^{4m}$ with $G^3$. And finally, I5 does the same with $L^{4m+1}$ and $L^{4m+2}$. On all the previous steps, a tree is being identified with another tree by point (a wedge of trees), which is always also a tree. That means that $G$ is a connected tree as claimed.

For the coloring of $G_m$, we use the coloring inherited by $L_{2m}$. To prove that this coloring is an $m$-coloring, we need to check that on all the identifications we made between a vertex $v_i \in L^p$ and a vertex $v_j \in L^q$ they do not have a common color around them. Now, $v_i$ and $v_j$ have a common color around them if and only if $|i-j|  \leq 1$. We can verify that on all identifications made from I1 to I6, we have, in fact, that $|i-j| \geq 3$. Note that the condition $m \geq 8$ needs to be put in place here to prove that $|i-j| \geq 3$ for I4- and I5-identifications. 

The property $|i -j| \geq 3$ implies not only that $v_i$ and $v_j$ have no common colors but also that they don't have consecutive colors in common. And by Proposition \ref{prop:n-consecutive-colored-graph-conditions} follows that $G$ is also a $2m$-consecutive-colored graph.

The symmetry is harder to prove. We need to apply Lemma \ref{prop:creating-symmetrical-graphs-by-gluing} several times. This can be accomplished by constructing $G_m$ following the same identifications but in a different order. The exact order is as follows:

\begin{enumerate}
    \item Use I3-identification first. The resulting graph $\tilde{G}^1$ is clearly $2m$-color-symmetric. The $2m$-reflection $\phi_1$ sends the vertices of $L^1$ to the ones on $L^{2m}$ with $\phi(v_i) = v_{2m-i}$
    \item We proceed to define $\tilde{G}^n$ inductively.
    \begin{itemize}
        \item Let us suppose that we have constructed $\tilde{G}^k$ through identifications of $L^1, L^2$ $, \dots L^k$ and $L^{2m}, \dots, L^{2m+1-k}$ with a reflection $\phi_k$ that maps $L_i$ to $L^{2m+1-i}$ reflecting the vertices and colors. 
        \item We define $\tilde{G}^{k+1}$ as the resulting graph from $\tilde{G}^k$, $L^{k+1}$, and $L^{2m-k}$ after identifying $v_{k-1} \in L^k \subset \tilde{G}^k$ with $v_{k+2} \in L^{k+1}$, and identifying $v_{2m-2-k} \in L^{2m-k}$ with $v_{2m+1-k} \in L^{2m+1-k} \subset \tilde{G}^k$. And we extend $\phi_k$ to $\tilde{G}^{k+1}$ by taking $L^{k+1}$ to $L^{2m-k}$ reflecting the vertices. We call this new function $\phi_{k+1}$
    \end{itemize}
    \item We proceed similarly as above. We now add to $\tilde{G}^n$ the next $2m$ copies of $L_{2m}$ but on pairs of the form $L^{2m+i}$ and $L^{4m+1-i}$ for $i=1, \dots, m$.
    \item Finally, we add the pair $L^{4m+1}$ and $L^{4m+2}$
\end{enumerate}

Identifying this way, on pairs, makes it possible to apply Lemma \ref{prop:creating-symmetrical-graphs-by-gluing} on each step. So, the graph $G_m$ is $2m$-color-symmetric.  As an example, the horizontal reflection along the vertical line passing through the middle of Figure \ref{fig:example-graph-G8} gives us the reflection of $G_8$ with color reflection $i \to (17-i)$.  

\end{proof}

So far, we have proven that graph $G_m$ has a lot of properties.  We are going to need those properties to prove our main theorem. However, we need to prove some other properties. But those properties will be more closely related to the branch covering represented by $G_m$ than to $G_m$ itself. 

\begin{figure}
    \centering
    \includegraphics[scale=0.5, angle=90]{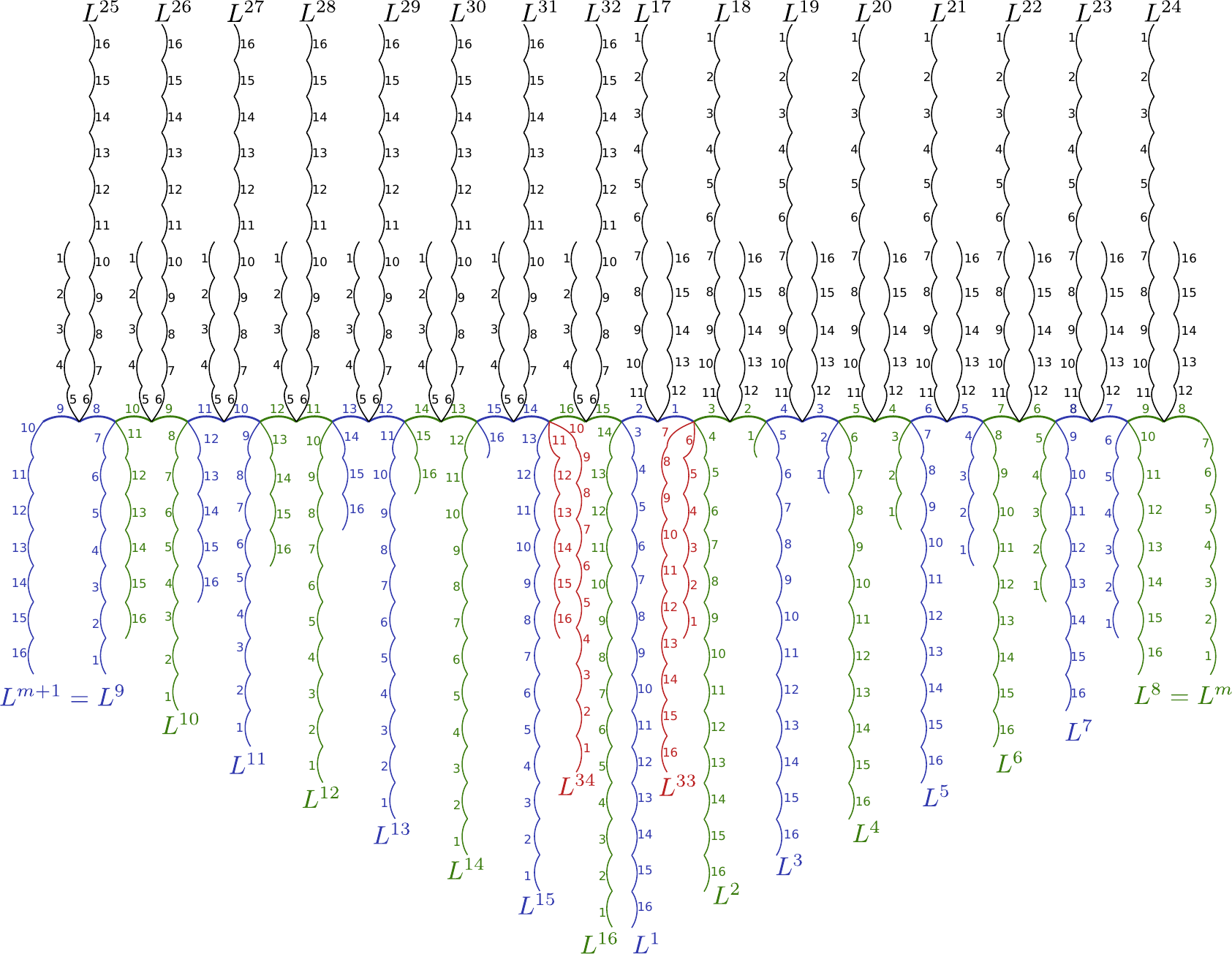}
    \caption{A complete drawing of $G_m$ for $m=8$}
    \label{fig:example-graph-G8}
\end{figure}

\section{Braids in $S^3$ as parts of the preimage of some branching set}\label{sec:constructing-branch-coverings}

In this section, we give a way to construct any given braid as a sub-braid of $\varphi^{-1}(L)$ for some covering map $\varphi:S^3 \to S^3$ branching along some link $L \subset S^3$. More formally, we are going to prove the following theorem:

\begin{thm}
\label{thm:main-theorem-translation-to-openbooks}
Let $L$ be a braid link in $S^3$. Then there is braid knot $K$ in $S^3$ and a covering $\varphi:S^3 \to S^3$ map branched along $K$ such that $\varphi$ maps disk pages into disk pages (of the corresponding standard open books of $S^3$ ) such that a positive stabilization of $L$ is a conjugate of some sublink of $\varphi^{-1}(K)$.   
\end{thm}

The detailed proof is given in Section \ref{sec:proof-of-main-theorem-translation-to-openbooks}, but we are going to do some preamble first.

We begin by taking the covering $\phi_m:\tilde{D}^2 \to D^2$ branched along a finite set of points $\{p_1, p_2, \dots, p_{2m}\}$, which has $G_m$ as its representation graph for some $m\geq 8$ . 

As explained in Section \ref{sec:description-with-graphs}, we can embed $G_m$ in the covering disk $\tilde{D}^2$ with the vertices of $G_m$ being the preimages of a point $*$ (the base point of $D^2$). Now, given $v \in G_m \subset \tilde{D}^2$ and any branching point $p_i \in D^2$, we define $v(i)$ in $\phi^{-1}(p_i)$ as follows: take $\gamma_i$ the straight line connecting $*$ with $p_i$, and consider $\tilde{\gamma}_i$ the lifting of $\gamma_i$ at $v$, we chose $v(i)$ as the other end of $\tilde{\gamma}_i$ (which is in $\phi^{-1}(p_i)$). Observe that $v(i) = w(j)$ for two distinct vertices $v, w \in G_m$  if and only if $i=j$ and the edge $[v, w]$ is colored with $i$. Also notice that given any vertex $v \in G_m$ the function $\phi$ maps the set $ \{v(1), v(2), \dots v(2m)\}$  one-to-one and onto $\{p_1, p_2, \dots, p_{2m}\}$. As an example of this notation consider Figure \ref{fig:example-L4}, the four leftmost vertices in the covering disk correspond to $\{v_0(1), v_0(2), v_0(3), v_0{4}\}$, the four at the right correspond to $\{v_1(1), v_1(2), v_1(3), v_1(4)\}$, and $v_0(1) = v_1(1)$.

For simplicity, we will denote $v_i \in L^j \subset G_m$ by $v_i^j$. Observe that $v_i^j = v_k^l$ if and only if they correspond to an identification of $v_i \in L^j$ with $v_k \in L^l$ in Definition \ref{def:graph_gm}. Now, we can name all the elements of $\phi^{-1}(p_i)$ as $v_j^k(i)$ for $k=1, \dots, 4m+2$ and $j=0,\dots, 2m$; of course, with many repetitions.

Using the above notation, we define $VH$ as the subset of vertices of $G_m$ listed below:
\begin{enumerate}
    \item $v_{i-2}^i, v_{i-1}^i, v_i^i$ for $i = m+1, m+2, \dots 2m$.
    \item $v_{i-1}^i, v_i^i, v_{i+1}^i$ for $i = 1, 2,\dots, m$
\end{enumerate}

On $VH$, we have 3 vertices from each $L^i$ for $i=1, \dots, 2m$. This makes a total of $6m$ vertices.  But we are over counting; $2m-1$ vertices are duplicated (check I1, I2, and I3 from \ref{def:graph_gm}). Then, the set $VH$ has exactly $4m+1$ vertices. 

Let $H_m$ denote the maximal subgraph of $G_m$ with vertices in $VH$. Observe that no vertex of $H_m$ has valence greater than two, so $H_m$ is a line graph with $4m$ vertices. In Figure \ref{fig:example-graph-G8} is the vertical subgraph pictured in bold.

\begin{defn}\label{def:subset-Q}
We denote by $Q$ the subset of $\phi^{-1}(\{p_1, \dots, p_{2m}\})$ containing the following $4m$ elements:
\begin{enumerate}
    \item $v^i_{i+1}(i)$ and $v^{m+1}_{m+1}(i)$ for $i=1, \dots, m$,
    \item $v^i_{i-2}(i)$ and $v^m_{m-1}(i)$ for $i=m+1, \dots, 2m$,
\end{enumerate}
\end{defn}

The first $2m$ elements of $Q$ are located at the two ends of $H_m$ ($v^{m+1}_{m+1}$ and $v^m_{m-1}$ ). The rest is in the middle of $H_m$.  Observe that $Q \cap \phi^{-1}(p_i)$ contains only two elements for every $i$, those elements are $\{v^i_i(i)$, $v^{m+1}_{m+1}(i)\}$ if $i \leq m$ or $\{v^i_{i-2}(i), v^m_{m-1}(i)\}$ if $i > m$.
 
Now consider the arcs $\alpha_1, \dots, \alpha_{2m-1}$ on $D^2$ connecting consecutive branching points with a straight line (we previously defined these arcs in Section \ref{sec:half-twist-lifting}). Ignore for now the arc $\alpha_{m}$. It is clear that $\phi^{-1}(\alpha_i)$ for every $i \neq m$ has only two connected components intersecting $Q$. This is because the ends of $\alpha_i$ are $p_i$ and $p_{i+1}$, so $\phi^{-1}(\alpha_i)$ can only intersect $Q$ at $\phi^{-1}(p_i) \cap Q$ or $\phi^{-1}(p_{i+1}) \cap Q$. We can easily describe these in two cases, when $i  < m$ and when $i >m$ as follows:
$$
\begin{array}{rl}
\phi^{-1}(p_i) \cap Q &=\left\{ \begin{array}{cc}
     \{v^i_{i+1}(i),\  v^{m+1}_{m+1}(i) \} & \textrm{for } i < m  \\
     \{v^i_{i-2}(i),\  v^m_{m-1}(i) \} & \textrm{for } i > m \\
     \end{array} \right.       \\
     & \\
\phi^{-1}(p_{i+1}) \cap Q &=\left\{ \begin{array}{cc}
     \{v^{i+1}_{i+2}(i+1),\  v^{m+1}_{m+1}(i+1) \} & \textrm{for } i < m  \\
     \{v^{i+1}_{i-1}(i+1),\  v^m_{m-1}(i+1) \} & \textrm{for } i> m \\
     \end{array} \right. 
\end{array}
 $$

For the case $i >m$ the points $v^{m+1}_{m+1}(i)$ and $v^{m+1}_{m+1}(i+1)$ belong to the same vertex $v^{m+1}_{m+1}$, so they are connected by an arc $\overline{\alpha}_i \subset \phi^{-1}(\alpha_i)$.  The arc $\overline{\alpha}_i$ covers 1-1 and onto $\alpha_i$. 

For the other two points $v^i_{i+1}(i)$ and $v^{i+1}_{i+2}(i+1)$, they are on different vertices. So they can not be connected by a 1:1 arc. But, by rule I1 from Definition \ref{def:graph_gm} we know that $v^{i+1}_{i+2} = v^{i}_{i-1}$. Then, by the definition of $L^i_m$, we know that $v^i_{i-1}, v^i_i$, and $v^i_{i+1}$ are consecutive vertices connected by edges of colors $i$ and $i+1$. Which means that  $v^i_{i-1}(i) = v^i_i(i)$ and $v^i_i(i+1) = v^i_{i+1}(i+1)$. So, taking the preimages of $\alpha_i$ contained in the discs corresponding to the vertices of $v^i_{i-1}, v^i_i$, and $v^i_{i+1}$ we get an arc $\widetilde{\alpha}_i$ that has ends in $v^{i+1}_{i+2}(i+1)=v^i_{i-1}(i+1)$ and $v^i_{i+1}(i)$ (see Figure \ref{fig:widetilde_alpha_i}). A similar phenomenon happens with $i > m$. 

In conclusion, there are two components of $\phi^{-1}(\alpha_i)$ having ends on $Q$. Both of them are arcs with both ends in $Q$,  one of the arcs $\widetilde{\alpha_i}$ covers 3:1 onto $\alpha_i$, and the other arc $\overline{\alpha_i}$ covers 1-1 and onto $\alpha_i$. 

\begin{figure}
    \centering
    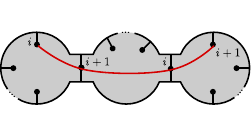
    \caption{The arc $\widetilde{\alpha_i}$ is the connected component of $\phi^{-1}(\alpha_i) $ the has ends in $v^i_{i+1}(i)$ and $v^{i+1}_{i+2}(i+1)$ for $i <m$}
    \label{fig:widetilde_alpha_i}
\end{figure}

For the particular case of $\alpha_{m}$, there are also two components of $\phi^{-1}(\alpha_m)$ intersecting $Q$. Both components are arcs covering 3:1 onto $\alpha_m$. We will denote those two arcs by $\widetilde{\alpha_{m,1}}$ and $\widetilde{\alpha_{m,2}}$; being $\widetilde{\alpha_{m,1}}$ the one connecting $v^{m+1}_{m+1}(m)$ with $v^{m+1}_{m-1}(m+1)$.

Observe that $ \cup_i  \widetilde{\alpha_i} \cup_i \overline{\alpha_i} \cup \widetilde{\alpha_{m,1}} \cup \widetilde{\alpha_{m,2}} $ has two connected components. In order to make it connected, we are going to add another arc. Let $\overline{\alpha_0}$ be the lifting of $\alpha_0 \subset D^2$ (the arc connecting $p_1$ and $p_{2m}$) at $v^1_2(1)$. Note that $\overline{\alpha_0}$ ends at $v^1_2(2m)$. After adding $\overline{\alpha_0}$ to our set of arcs, it becomes connected. There are other two connected components of $\phi^{-1}(\alpha_0)$ intersecting $Q$: one is connecting $v^{m+1}_{m+1}(1) \in Q$ with $v^{m+1}_{m+1}(2m) \not \in Q$ and the other is connecting $v^{m}_{m-1}(1) \not \in Q$ with $v^{m}_{m-1}(2m) \in Q$.  
\begin{rem}\label{rem:generators_of_braids}
 The half-twists around the set of arcs $\widetilde{\alpha_i}$'s,  $\overline{\alpha_i}$'s,  $\widetilde{\alpha_{m,1}}$  and $ \widetilde{\alpha_{m,2}}$ form a system of generators for the mapping class group of $D ^2$ with $Q$ as marked points.
\end{rem}

Now, we can have those half-twists as the lift of half-twists along arcs $\alpha_i$ (see  Corollary \ref{cor:lifting-conditions}). We have shown that $\tau^3_{\alpha_i}$ (the half-twist around $\alpha_i$) lifts. Moreover, it lifts to a single half-twist around $\widetilde{\alpha_i}$ followed with a cubic half-twist around $\overline{\alpha_i}$ ($i \neq 0, m$) plus some other twists not involving $Q$.

By just twisting around $\alpha_i$, it seems impossible to create all possible braid words on the covering disk. So, we need to get rid of that extra cubic twist. To accomplish that, we will make use of another set of arcs.

\subsection{The arcs  $\gamma_i$'s}\label{sec:arcs-gammai-definition}

The arcs that we are going to use are the arcs $\alpha_i$'s modified using some additional loops $\beta_i$'s. The exact definition of the $\beta_i$'s arcs is given below. For this definition, we require to orientate the arcs $\alpha_i$'s. Recall that  $\alpha_i$ is a straight arc that connects the branching points $P_i$ and $P_{i+1}$ (see Figure \ref{fig:arc-alpha-connecting-1-and-2} for $\alpha_1$). We give each $\alpha_i$ the orientation from $P_i$ to $P_{i+1}$.

\begin{defn}
Let $E$ be a surface with a finite set of marked points $P$. And let $\alpha$ be an embedded and oriented arc in $E$ such that the interior of $\alpha$ is in $E \backslash P$ and $\partial \alpha \subset P$. 

Now, let $\beta$ be a loop with base point $x$ in $\alpha - \partial \alpha $. We define the \emph{surgery of} $\alpha$ \emph{along} $\beta$ as the arc $\gamma$ resulting from the following construction:
  \begin{enumerate}
      \item We remove a little neighborhood of $x$ from the interior of $\alpha$ and create two oriented sub-arcs: $\alpha_1$ and $\alpha_2$, being $\alpha_1$ the arc that contains the starting point of $\alpha$.
      \item Then modify the base points of $\beta$ to begin at the end of $\alpha_1$ and to end at the beginning of $\alpha_2$.
      \item Finally, $\gamma = \alpha_1 \cdot \beta \cdot \alpha_2$.
  \end{enumerate}
\end{defn}

Now, we are going to define the loops $\beta_i$'s. Because they are loops, we can write them as a conjugate of a product of the generators $\sigma_1, \sigma_2, \dots, \sigma_{2m}$ (shown in Figure \ref{fig:disk-generators}). Because $\beta_i$ has to start in the middle of $\alpha_i$ we are going to need to move the base point. This implies that we need to conjugate $\sigma_i$ with an arc connecting the base point to the middle point of $\alpha_i$. Let $\omega_i$ be an arc connecting the base point to the middle point but without going around any branching in $P$.  

Now we define $\beta_i$ as follows:
\begin{equation}
\beta_i = \left\{ \begin{array}{cc}
      (\omega_i\Pi_{j=i+1}^{2m-1}\sigma_j)\sigma_{2m} \sigma_{2m-5} (\omega_i\Pi_{j=i+1}^{2m-1}\sigma_j)^{-1}& \mbox{ for } i= 1, \dots, m-1   \\
      (\omega_i\Pi_{j=i}^{2}\sigma_j)\sigma_{1} \sigma_{6} (\omega_i\Pi_{j=i}^{2}\sigma_j)^{-1} & \mbox{for } i =  m+1, \dots, 2m
\end{array} \right.    
\end{equation}

We can now define $\gamma_i$ as the surgery of $\alpha_i$ along $\beta_i$. A picture of $\beta_i$ is given in Figure \ref{fig:gamma_i}. Observe that we haven't defined $\beta_m$, hence there is no $\gamma_m$ defined. 

\begin{figure}
    \centering
    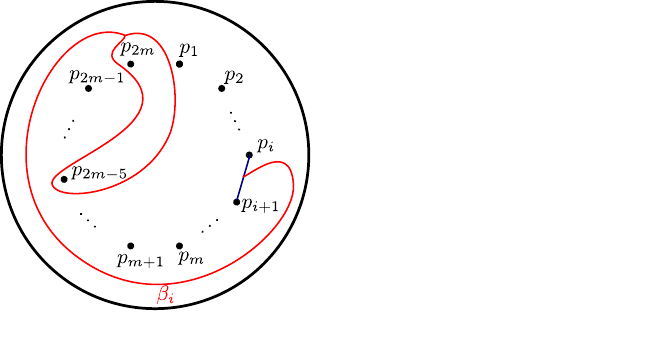
    \caption{The definition of the arc $\beta_i$ changes depending on $i>m$ or if $i<m$. The arc $\gamma_i$ is the surgery of $\alpha_i$ along $\beta_i$ for $i\neq m$.}
    \label{fig:gamma_i}
\end{figure}

\begin{prop}\label{prop:gamma_i-properties}
Let $\phi: D^2 \to D^2$ be the branch covering represented by $G_m$. And let $\gamma_i$ be the arc on the base disk defined as above for some $i$. Then we have the following:

\begin{enumerate}
    \item The square of the half-twist around $\gamma_i$ lifts to a composition of arc twists on the covering disk.
    \item There are only three connected components of $\phi^{-1}(\gamma_i)$ intersecting $Q$; one has both ends on $Q$, and the other two have only one.
    \item The component of $\phi^{-1}(\gamma_i)$ having both ends on $Q$ satisfies the following:
    \begin{itemize}
        \item It covers 1-1 and onto $\gamma_i$ under $\phi$.
        \item It is homotopic to $\overline{\alpha}_i$ relative to $Q$ (ignoring the rest of\\ $\phi^{-1}(\{p_1, p_2, \dots, p_n\})$).
        \item If $i <m$, it connects $v_{m+1}^{m+1}(i)$ with $v_{m+1}^{m+1}(i+1)$.
        \item If $i >m$, it connects $v_{m}^{m-1}(i)$ with $v_{m}^{m-1}(i+1)$.
    \end{itemize}
    \item The two components of $\phi^{-1}(\gamma_i)$ with only one end on $Q$ cover 1-1 and onto $\gamma_i$ under $\phi$.
\end{enumerate}
\end{prop}

Thanks to the previous proposition, the function $\tau^2_{\gamma_i}$ lifts to the covering disk and it lifts to something homotopic (relative to $Q$) to the square of half-twist around $\overline{\alpha_i}$. So, by adding this new set of squares of half-twists, now we can generate the whole braid group of the covering disk relative to $Q$.

Before proving Proposition \ref{prop:gamma_i-properties}, we are going to see how it is possible to use this new set of arcs to construct any braid. You can find the proof of Proposition \ref{prop:gamma_i-properties} in Section \ref{sec:proof-gamma_i_properties}.

\subsection{Constructing any braid as a sub-braid}\label{sec:constructing-any-braid}
Roughly speaking, what are going to do is to construct any given closed braid $L$ as a sublink of $\varphi^{-1}(L')$ with $L'$ a  braid with fewer components than $L$, and  $\varphi: S^3 \to S^3$ is a covering branched along $L'$. With this proven, Theorem \ref{thm:main-theorem-translation-to-openbooks} will follow, because we can iterate the process until we get a knot on the  branching set. 

In order to think of $\varphi^{-1}(L')$ as a braid, we need $\varphi$ to preserve the standard open book decomposition of $S^3$. Moreover, we need $\varphi$ to be induced by a covering map $\phi:D ^2 \to D^2$ and a covering transformation pair $(f,\tilde{f})$ (see Section  \ref{sec:description-with-graphs}). 

Let $P = \{p_1, \dots, p_n\}  \subset D^2$  be the branching set of $\phi$. Then, the map $\varphi$ is branching along a link $L'$ that can be identified with the projection of $P \times I$ into the quotient $M(D^2, f) \cong S^3$.  This implies that $L'$ is a closed braid with braid word given by $f$ (decompose $f$ as a product of half-twists to get the word). Similarly, $\varphi^{-1}(L')$ is also a closed braid but with word given by $\tilde{f}$. 

To construct a sub-braid of $\varphi^{-1}(L')$ we can take a subset $Q$ of $\phi^{-1}(P)$ such that $\tilde{f}(Q) = Q$. Then the image of $Q \times I$ in the quotient space $M(D^2,\tilde{f}) \cong S^3$ is a sub-braid of $\varphi^{-1}(L')$. In this case, the braid word is now given by $\tilde{f}$ relative to $Q$; this means that on the decomposition of $\tilde{f}$ as a product of half-twists, we can get rid of twists around arcs with endpoints out of $Q$.

Thanks to the above observations, it is possible to translate our problem into the construction of a pair $(f,\tilde{f})$ such that $\tilde{f}$ is a braid word for $L$ relative to $Q$, and the braid given by $f$ with fewer components than $L$. 

We are going to construct $(f,\tilde{f})$ as a composition of covering transformation pairs of the form $(\tau^k_a, h)$, where $h$ is the lift of $\tau^k_a$ relative to $Q$. As we only care about the homotopy type of $\tilde{f}$ relative to $Q$, we will write the pairs in reduced form: ignoring all the half-twist not involving vertices of $Q$, ignoring even powers involving only one vertex of $Q$, and at some cases, we will replace the arc (where we twisted around) for another homotopic arc in $D^2-Q$ relative to $Q$. We will call a pair $(f, h)$ \emph{$Q$-reduced transformation pair} if $h$ is the reduction of $\tilde{f}$ for some covering transformation pair $(f,\tilde{f})$.

The next lemma gives us the transformation pair that we will use for our construction. To simplify the notation, we will write $a^n$ to denote the function $\tau_a^n$: the $n$-th power of the half-twist around the arc $a$.

\begin{lem}\label{lem:basic-covering-maps} Let $\phi:D^2 \to D^2$ be the branch covering represented by $G_m$ and  let $Q$ be the subset of $\phi^{-1}(P)$ defined on \ref{def:subset-Q}. The following pairs are $Q$-reduced covering transformations of $\phi$. 
\begin{enumerate}
    \item $(\alpha_i^3, \widetilde{\alpha_i} \circ \overline{\alpha_i}^3)$ for $i \in \{1, 2,\dots, 2m-1\} \backslash \{m\}$
    \item $(\alpha_m^3, \widetilde{\alpha_{m,1}} \circ \widetilde{\alpha_{m,2}})$ 
    \item $(\alpha_0^2,\overline{\alpha_0}^2)$ 
    \item $(\gamma_i^2, \overline{\alpha_i}^2)$   $i \in \{1, 2,\dots, 2m-1\} \backslash \{m\}$ (we replace the lifting of $\gamma_i$ for $\overline{\alpha_i}$ because they are homotopic in $D^2 -Q$).
\end{enumerate}
\end{lem}

 At the beginning of Section \ref{sec:constructing-branch-coverings}, we remark that we can write any homotopic class  $\tilde{f}:(D^2,Q) \to (D^2,Q)$ as a product of (\emph{i.e.} a word on) $\widetilde{\alpha_i}$'s, $\overline{\alpha_i}$'s, $\widetilde{\alpha_{m,1}}$,  $\widetilde{\alpha_{m,2}}$ and $\overline{\alpha_0}$. The exact order of these generators is as follows:
 
 $\overline{\alpha_1},  \dots ,\overline{\alpha_{m-1}}, \widetilde{\alpha_{m,1}}, \widetilde{\alpha_{m+1}}, \dots, \widetilde{\alpha_{2m-1}}, \overline{\alpha_0}, \widetilde{\alpha_{1}}, \dots, \widetilde{\alpha_{m-1}},\widetilde{\alpha_{m,2}}, \overline{\alpha_{m+1}}, \dots, \overline{\alpha_{2m-1}}$\\ (see Figure \ref{fig:Q-partition-in-three-parts} for the case $m=8$).

The transformations given by Lemma \ref{lem:basic-covering-maps} are enough to construct a ``representative'' $\tilde{f}:(D^2,Q) \to (D^2,Q)$  of any braid $B$ under positive stabilization, conjugation and braid isotopy.  This is what we are going to prove in the following section.

\subsection{Proof of Theorem \ref{thm:main-theorem-translation-to-openbooks}}\label{sec:proof-of-main-theorem-translation-to-openbooks}

Before starting with the proof, we just need to make a small clarification about the way we make our braid diagrams and the corresponding braid permutation.

\begin{rem}
All the braid pictures made on this document were done by reading the braid word from top to bottom. And the permutation that we assign to it comes from reading these pictures from top to bottom as well. This means that the permutation of a braid is given by a function that maps a point from the top to the point in the bottom that it is connected to.
\end{rem}
This way of assigning permutations is unusual because we no longer get a group morphism between the braid group and the symmetric group ($\rho: B_n \to S_n$), it satisfies $\rho(ab) = \rho(b)\rho(a)$. As we don't actually make use of this group morphisms property we don't mind losing it. What we win with this change is that we can get the permutation and the braid word from reading the picture top to bottom.

As this is a constructive proof, we are going to split it into several steps; five steps to be precise. In the first two steps, we put the given link $L$ in a manageable form. In steps 3 and 4, we actually create $L$ as the lift of some braid $L'$ in the base 3-sphere. In step 5 we perform a final move that reduces the number of components of $L'$ and changes $L$ only by some positive stabilizations.

\subsubsection*{Step 1. Stabilize $L$ to get enough strands and the right permutation form}
Now recall that $L$ is the closure of some braid $B_L$ with $n$ strands.  The braid permutation of $B_L$ can be decomposed as a disjoint product of cyclic permutations $\eta_1 \cdot \eta_2 \cdots \eta_k$ where $k>1$ (recall that $L$ is not connected). Just by taking a conjugation of $B_L$ we can make $\eta_i$ to be the permutation $(r_{i-1}+1\quad r_{i-1}+2\ \cdots\ r_i)$ for every $i$, where  $0 = r_0 < r_1 < r_2 < \dots < r_k = n $. Observe that a positive stabilization of $B_L$ will increase both the number of strands of $B_L$, and the length of one $\eta_i$ by 1. So, using positive stabilizations, we can make $\eta_1$ be a permutation of half of the total number of strands of $B_L$. Let $m \in \mathbb{N}$ denote half of the number of strands, so $\eta_1 = (1\ 2\ \dots\ m)$ and the number of strands is $2m$. We may need to add some extra positive stabilizations to make $m \geq 8$.

As an example, think that the given link $L$ is the Hopf Link $L_{Hopf}$, and it is presented as the closure of the 2-strand braid shown in Figure \ref{fig:example-hopf-stabilized} a). $L_{Hopf}$ is already in the form described above, but with $m=1$. So, we need to perform some extra positive stabilization moves to get $m=8$, the resulting braid $B_L = B_{Hopf}$ is drawn in Figure \ref{fig:example-hopf-stabilized} b). In this case, the corresponding permutation for $B_{Hopf}$ is $\eta_1 \eta_2= (1\ 2\ 3\ \dots\ 8)(9\ 10\ 11\ \dots\ 16)$.  

\begin{figure}
    \centering
    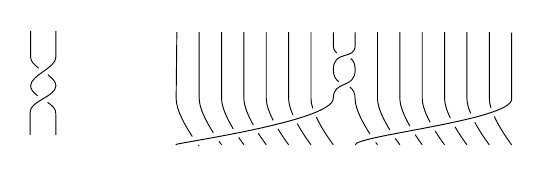
    \caption{The Hopf Link and a positive stabilization with $2m=16$ strands. We have marked $\sigma(1), \sigma(7), \sigma(10),$ and $\sigma(16)$}
    \label{fig:example-hopf-stabilized}
\end{figure}

So far, we have written $B_L$ as a closed braid with $2m$-strands, where the first $m$ strands belong to the same component of $L$. Now, we want to construct the braid word of $B_L$ as a product of half-twists around arcs in $(D^2,Q)$. For that, we need first to give a correspondence between the strands $1, 2, \dots 2m$ of $B_L$, and the elements of $Q$. As a correspondence, we will use the function $\sigma: \{1, \dots, 2m\} \to Q $ defined as follows:

\begin{equation}
    \sigma(i) = \left\{ \begin{array}{ll}
         v^{m+i}_{m-2+i}(m+i) & \textrm{if } i \leq  m \\
         v^{i-m}_{i-m+1}(i-m) &  \textrm{if } i >  m
    \end{array}
    \right.
\end{equation}
For the Hopf link example in Figure \ref{fig:example-hopf-stabilized} b) we have marked some of the strands with its corresponding value $\sigma(i) \in Q$ . 

Using the above correspondence, it is possible to write the word of $B_L$ as a product of half-twists around the arcs: $$\widetilde{\alpha_{m+1}}, \widetilde{\alpha_{m+2}}, \dots, \widetilde{\alpha_{2m-1}}, \overline{\alpha_0}, \widetilde{\alpha_{1}}, \widetilde{\alpha_{2}}, \dots, \widetilde{\alpha_{m-1}}.$$ We denote the corresponding braid word on $Q$ as $W_L$. In the Hopf Link example, we got that $W_L = W_{Hopf} = \overline{\alpha_0}^2 \widetilde{\alpha_{15}} \widetilde{\alpha_{14}} \cdots \widetilde{\alpha_{9}}  \widetilde{\alpha_{7}} \widetilde{\alpha_{6}} \cdots \widetilde{\alpha_{1}}$ (see Figure \ref{fig:example-hopf-stabilized})


\subsubsection*{Step 2. Factorize $B_L$ as a product of cyclic braids and Artin's generators}

Now, for each $i$, we can decompose the cyclic permutation $\eta_i$ as a product of transposition of the form $(k\ k+1)$. For each transposition, we can take one positive\footnote{The sign given to the half-twist is the same as the sign convention on the resulting crossing for the corresponding closed braid. The crossing sign convention is shown in Figure \ref{fig:crossing-signs}.} half-twist around the generator arc that connects the $k-$ and $k+1$-strands. The product of these half-twist creates a braid $C_i$ with permutation $\eta_i$ (see Figure \ref{fig:cycle_braid_Ci} for an example). Observe that the closure of $C_i$ is a positive stabilization of the trivial knot.

\begin{figure}
    \centering
    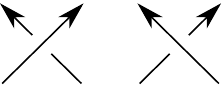
    \caption{Crossing sign convention}
    \label{fig:crossing-signs}
\end{figure}

\begin{figure}
    \centering
    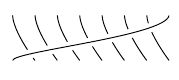
    \caption{Braid $C_1$ corresponding to the permutation $\eta_1=(1\ 2\ \dots\ 8)$}
    \label{fig:cycle_braid_Ci}
\end{figure}

Now the braid $V = (C_1 \cdot C_2 \cdots C_k)^{-1}  \cdot W_L$ has a trivial permutation, so it is a pure braid. It is well known that any pure braid can be written as a product of $A_{i,j}$ generators (See \cite{RolfsenTutorial} or \cite{ArtinBraids}). Where an $A_{i,j}$ generator is a braid that takes two strands and wraps them around each other while passing over all the strands in between (See Figure \ref{fig:artin_generators}). 

\begin{figure}
    \centering
    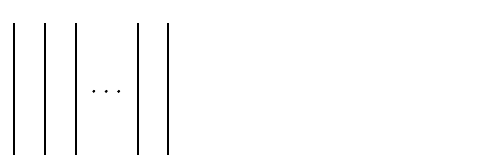
    \caption{An Artin's generator $A_{i,j}$}
    \label{fig:artin_generators}
\end{figure}

The above implies that we can write $W_L$ as a product of $C_1 \cdots C_k $ followed by a product of $A_{i,j}$ generators. For example, $W_{Hopf}$ is the product $C_1 \cdot C_2 \cdot A_{9,10}^{-1} \cdot A_{1,10} \cdot A_{9,10}$  as shown in Figure \ref{fig:example_step3_artin_decomp}.  

\begin{figure}
    \centering
    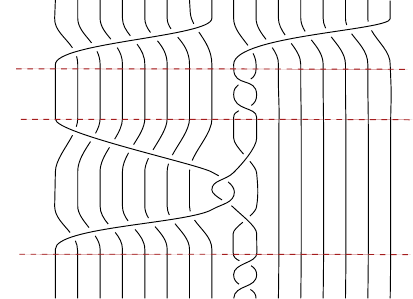
    \caption{Hopf link after some positive stabilization and decomposed as a product of $C_i$'s and $A_{i,j}$'s }
    \label{fig:example_step3_artin_decomp}
\end{figure}

\subsubsection*{Step 3. Obtaining $C_i$ as a lift of some braid}

Now, we are going to prove that we can obtain $C_i$ and the generators $A_{i,j}$ as compositions of lifts from half-twists.  But first, we need to split the strands of $Q$ in three groups $Q_0$, $Q_+$ and $Q_-$:
\begin{itemize}
    \item $Q_0 = \{v^{m+i}_{m+i-2}(m+i)\}_{i=1}^m  \cup \{v^j_{j+1}(j)\}_{j=1}^m$
    \item $Q_+ = \{v_{m-1}^m(m+i)\}_{i=1}^m$
    \item $Q_- = \{v^{m+1}_{m+1}(j)\}_{j=1}^m$
\end{itemize}

Informally speaking, we can say that $Q_0$ is a subset of vertices in the "middle" of $Q$. Observe that we have chosen to place the strands of $L$ in $Q_0$. In Figure \ref{fig:Q-partition-in-three-parts} you can find a picture showing the relation between the points of $Q$ and the lifting of the arcs $\alpha_i$ for the case $m=8$.

\begin{figure}
    \centering
    \includegraphics{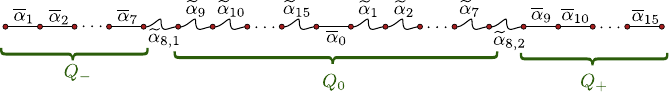}
    \caption{Partition of $Q$ into $Q_- \cup Q_0 \cup Q_+$ with the relevant lifts of $\alpha_i$ for $i=0, 1,\dots, 15$ and $m=8$. }
    \label{fig:Q-partition-in-three-parts}
\end{figure}

Let us start by creating $C_1$ on $Q_0$. For that, we take the product $\alpha^3_{m+1} \cdot \alpha^3_{m+2} \cdots \alpha^3_{2m-1}$. By Lemma \ref{lem:basic-covering-maps} part (1) the lifting relative to $Q$ is the product $\widetilde{\alpha}_{m+1} \cdot \widetilde{\alpha}_{m+2} \cdots \widetilde{\alpha}_{2m-1} \cdot \overline{\alpha}^3_{m+1} \cdot \overline{\alpha}^3_{m+2} \cdots \overline{\alpha}^3_{2m-1}$. This word is creating a copy of $C_1$ on $Q_0$ but another ``more twisted'' version of $C_1$ on $Q_+$. To remedy this, we are going to lift the following word instead:

$$\alpha^3_{m+1} \cdot \gamma^{-2}_{m+1}\cdot \alpha^3_{m+2} \cdot \gamma^{-2}_{m+2} \dots \alpha^3_{2m-1} \cdot \gamma^{-2}_{2m-1}$$

By Lemma \ref{lem:basic-covering-maps} the lifting of the previous word is:
$$( \widetilde{\alpha}_{m+1} \cdot \widetilde{\alpha}_{m+2}  \cdots \widetilde{\alpha}_{2m-1} ) \cdot ( \overline{\alpha}_{m+1} \cdot \overline{\alpha}_{m+2} \cdots \overline{\alpha}_{2m-1} )$$ So, we are obtaining two copies of $C_1$ on the covering space: one copy on $Q_0$ and the other on $Q_+$. 

In the Hopf link example, consider the first crossing from the top in Figure \ref{fig:example_step3_artin_decomp} (this corresponds to a half-twist around $\widetilde{\alpha}_{15}$). We can create this crossing as the lift of $\alpha^3_{15} \cdot \gamma^{-2}_{15}$. In the base disk, $\alpha^3_{15} \cdot \gamma^{-2}_{15}$ creates the braid shown in Figure \ref{fig:example_move_to_create_crossing}. This move lifts to  $\widetilde{\alpha}_{15} \cdot \overline{\alpha}_{15}$, creating the desired crossing plus a crossing involving two strands in $Q_+$, the ends of $\overline{\alpha}_{15}$. This way, taking the composition of $\alpha^3_{i} \cdot \gamma^{-2}_{i}$ for $i=15, 14, \dots, 7$ we can create $C_1$ in $Q_0$ and an identical copy in $Q_+$.
\begin{figure}
    \centering
    \includegraphics{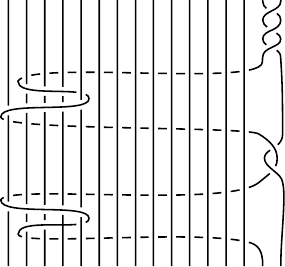}
    \caption{Braid obtained after performing the composition of half-twists given by $\alpha^3_{15} \cdot \gamma^{-2}_{15}$}
    \label{fig:example_move_to_create_crossing}
\end{figure}

Similarly, by taking lifts of $\alpha_i^3 \circ \gamma_i^{-2}$ for  $1\leq i<m$ we can create $C_2 \cdot C_3 \cdots C_k$ in the covering space and another copy in $Q_-$. On the base space, we obtain a braid with the same permutation as $C_1 \cdot C_2 \cdot C_3 \cdots C_k$, but with a lot more crossings: we get a braid like the one shown in Figure \ref{fig:example_move_to_create_crossing} per crossing of $C_i$. 

\subsubsection*{Step 4: Obtaining $A_{i,j}$ as a lift of some braid}
Now, to finally create $W_L$ in the strands of $Q_0$ we just need to create the $A_{i,j}$ generators (with $\sigma(i),\sigma(j) \in Q_0$) as the lift of some word in the base space. We are going to divide this problem into cases, but we first need to divide $Q_0$ into two parts:

\begin{itemize}
    \item $Q_0^+ = \{v^{m+i}_{m+i-2}(m+i)\}_{i=1}^m = \{\sigma(i)\}_{i=1}^m$
    \item $Q_0^- = \{v^i_{i+1}(i)\}_{i=1}^m = = \{\sigma(m+i)\}_{i=1}^m$.
\end{itemize}

\textbf{Case $\sigma(i),\sigma(j) \in Q_0^+$ or $\sigma(i),\sigma(j) \in Q_0^-$.}\\
This case implies that either $i, j \leq m$ or $i,j > m$. This means that $A_{i,j}$ involves only strings from $Q_0^+$ or $Q_0^-$, but not both. We consider first the case $i, j \leq m$.

To create $A_{i,j}$ with $i < j \leq m$ we are going to take the lift of $$D_{i,j} = \Lambda_{i,j} \cdot \alpha^6_{i+m} \cdot \gamma_{i+m}^{-6}\cdot \Lambda^{-1}_{i,j} \quad \textrm{ where } \quad \Lambda_{i,j} = \prod_{k=1}^{j-i-1} \alpha^3_{m+j-k}$$
 Here $\Lambda_{i,i+1}$ is the identity, so $D_{i,i+1} =  \alpha^6_{i+m} \cdot \gamma_{i+m}^{-6}$.

 By Lemma \ref{lem:basic-covering-maps}, $\alpha^6_{i+m}\circ \gamma^{-6}_{i+m}$ lifts to $$\widetilde{\alpha}^2_{i+m} \circ \overline{\alpha}^6_{i+m} \circ \overline{\alpha}^{-6}_{i+m}  = \widetilde{\alpha}^2_{i+m}.$$ And the lift of $\Lambda_{i,j}$ is 
 $$\prod_{k=1}^{j-i-1} \widetilde{\alpha}_{m+j-k} \circ \overline{\alpha}^3_{m+j-k} = \prod_{k=1}^{j-i-1} \widetilde{\alpha}_{m+j-k} \circ \prod_{k=1}^{j-i-1} \overline{\alpha}^3_{m+j-k}$$
 So, this means that the lift of $D_{i,j}$ is $$(\prod_{k=1}^{j-i-1} \widetilde{\alpha}_{m+j-k} )\cdot \widetilde{\alpha}^2_{i+m} (\prod_{k=1}^{j-i-1} \widetilde{\alpha}_{m+j-k} )^{-1} \cdot (\prod_{k=1}^{j-i-1} \overline{\alpha}^3_{m+j-k}) \cdot (\prod_{k=1}^{j-i-1} \overline{\alpha}^3_{m+j-k})^{-1} = A_{i,j}$$

The lift of $\alpha^6_{i+m}$ is being canceled with that of $\gamma^{-6}_{j-1}$ on $Q_+$, making space for the lift of  $\Lambda_{i,j}$ to be canceled with its own inverse on $Q_+$, and leaving the braid on $Q_+$ unaffected. This way we have created the braid generator $A_{i,j}$ on $Q$ as the lift of $D_{i,j}$.

The case $i,j >m$ is identical. We just need to modify $D_{i, j}$ as shown in the next formulas:

$$D_{i,j} =  \Lambda_{i,j} \cdot \alpha^6_{i-m} \cdot \gamma^{-6}_{i-m} \cdot \Lambda_{i,j}^{-1} \quad \textrm{where} \quad \Lambda_{i,j} = \prod_{k=1}^{j-i-1} \alpha^3_{j-m-k} .$$
Here, $j>i>m$ and if $j=i+1$ we have that $\Lambda_{i,i+1}$ is the identity, so $D_{i,i+1}= \alpha^6_{i-m} \cdot \gamma^{-6}_{i-m}$.

Back to the Hopf link example, the generator $A_{9,10}$ fits into this case. Observe that $A_{9,10}=\widetilde{\alpha}^2_1$ (check Figure \ref{fig:Q-partition-in-three-parts} ). As we proved, we need to lift $D_{9,10} = \alpha^6_{1} \gamma_{1}^{-6}$. We have drawn $D_{9,10}$ on Figure \ref{fig:example_move_to_create_artin_generator}. 

\begin{figure}
    \centering
    \includegraphics{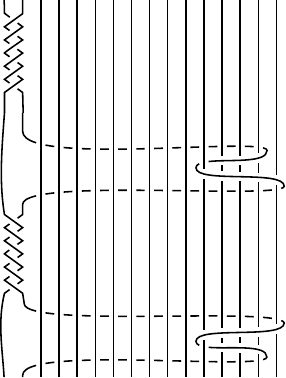}
    \caption{Move $D_{9,10}$ that lifts to $A_{9,10}$}
    \label{fig:example_move_to_create_artin_generator}
\end{figure}
  
\textbf{Case $i \in Q_0^+$ and $j \in Q^-_0$.}

The main difficulty with this case is that the two strands are separated by the generator $\overline{\alpha}_0$, which is not the lift of a cube of $\alpha_0^3$  as the other generators are. But we can go around this complication by writing $A_{i,j}$ as $\widetilde{\delta^+_i} \cdot \widetilde{\delta^-_j} \cdot  \overline{\alpha}^2_0 \cdot \widetilde{\delta^-_j}^{-1} \cdot \widetilde{\delta^+_i}^{-1}$ where

$$
\widetilde{\delta^+_i} = \left\{ 
\begin{array}{cc}
     \widetilde{\alpha}_i^{-1} \cdots \widetilde{\alpha}^{-1}_{2m-1}& \textrm{If } i \leq 2m-1 \\
      1 & \text{If } i = 2m 
\end{array}\right.
$$

$$
\widetilde{\delta^-_j} = \left\{ 
\begin{array}{cc}
     \widetilde{\alpha}_{j-1}\cdots \widetilde{\alpha}_{1}& \textrm{If } j > 1 \\
      1 & \text{If } j = 1 
\end{array}\right.
$$

We can create now $A_{i,j}$ as the lift of $\delta^+_i \cdot \delta^-_j \cdot \alpha^2_0 \cdot (\delta^-_j)^{-1} \cdot (\delta^+_i)^{-1}$ where $\delta^+_i$ and $\delta^-_j$ are the obvious choices:

$$
\delta^+_i = \left\{ 
\begin{array}{cc}
     \alpha^{-3}_i \cdots \alpha^{-3}_{2m-1}& \textrm{If } i \leq 2m-1 \\
      1 & \text{If } i = 2m 
\end{array}\right.
$$
 
$$
\delta^-_j = \left\{ 
\begin{array}{cc}
     \alpha^3_{j-1} \cdots \alpha^3_{1}& \textrm{If } j > 1 \\
      1 & \text{If } j = 1 
\end{array}\right.
$$ 
 
And it is enough to take a look at the lifts given by Lemma \ref{lem:basic-covering-maps} to see that we made the right choice. The lifting of  $\delta^+_i \cdot \delta^-_j \alpha^2_0 (\delta^-_j)^{-1} (\delta^+_i)^{-1}$ is $A_{i,j}$ relative to $Q$.

\subsubsection*{Step 5: Final move to reduce the number of components}\label{sec:step-5-final-move}
Observe now that we have created $L$ on the strands of $Q_0$, $C_1$ on $Q_+$, and $C_2 \cdots C_k$ on $Q_-$.  And on the downstairs braid, we have $S$ with the same number of components as $L$. Moreover, the braid permutation of $S$ is equal to that of $L$. 

Now, we want to reduce the number of components of $S$ without changing $L$ beyond some positive stabilizations. The move that we are going to use is  $\alpha^3_m$. So far, we haven't used $\alpha_m$ not even once, and by using it now, we are changing the permutation of $S$ by a multiplication of the transposition $(m\ m+1)$, implying the reduction of components by 1.

On the one hand, we take $\alpha_m^3$ to reduce the number of components of $S$, and on the other hand, it lifts to the composition of half-twists $\widetilde{\alpha}_{1,m}$ and $\widetilde{\alpha}_{2,m}$ relative to $Q$. But these two twists connect the braid $L$  (in $Q_0$) with $C_1$ (in $Q_+$) and $C_2$ (in $Q_-$), adding one positive crossing between them (see Figure \ref{fig:resulting-link-after-alpham}). 

As each of $C_1$ and $C_2$ is a positive stabilization of the trivial knot, we have that the newly obtained link $L'$ in $Q$ is a positive stabilization of $L$ (see Figure \ref{fig:resulting-link-after-alpham})

\begin{figure}
    \centering
    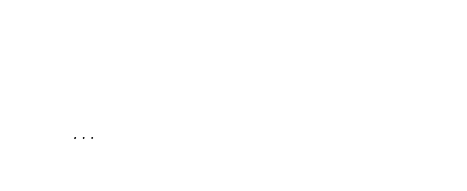
    \caption{Resulting link after adding the crossing $\widetilde{\alpha}_{1,m}$ and $\widetilde{\alpha}_{2,m}$ }
    \label{fig:resulting-link-after-alpham}
\end{figure}
This completes the construction of the link $L'$.

\subsubsection*{Final argument: iterate }
Finally, given a disconnected braid $L$, we have constructed $\phi: D^2 \to D^2$ branched along $P$ ($2m$ points for some $m \geq 8$) and a covering transformation $(f, \tilde{f})$ such that they induce a branch covering $\varphi: S^3 = M(D^2,\tilde{f}) \to S^3 = M(D^2,f)$ branched along a braid $L_1$ (the image of $P \times I$ on $M(D^2,f)$). And we also constructed a subset $Q \subset \phi^{-1}(P)$ such that its induced sub-braid $L'$ of $\varphi^{-1}(S)$ is equivalent to $L$. Moreover, the number of components of $L_1$ is one less than $L$. 

We can iterate this process, applying this algorithm on the new braid $L_1$ and obtaining $L_2$, and so on until we get $L_{k-1}$ with only one component. After composing all these branch coverings, we obtain a branch covering $\varphi: S^3 \to S^3$ branched along a knot $L_{k-1}$ with the properties that we wanted in the first place. This completes the proof of Theorem \ref{thm:main-theorem-translation-to-openbooks}.

\subsection{Proof of Proposition \ref{prop:gamma_i-properties}}\label{sec:proof-gamma_i_properties}

In this section, we will prove all the properties of $\gamma_i$-arcs claimed on Proposition \ref{prop:gamma_i-properties}. Our strategy is to describe the preimage of $\beta_i$ and later check that the properties from Proposition \ref{prop:gamma_i-properties} will be met after performing surgery on $\alpha_i$ along $\beta_i$.

More precisely, as we mentioned in Section \ref{sec:arcs-gammai-definition}, the arcs $\gamma_i$ are constructed from the $\alpha_i$ arcs  after doing surgery along the $\beta_i$ loops. This implies that  $\phi^{-1}(\gamma_i)$ is the surgery of $\phi^{-1}(\alpha_i)$ along $\phi^{-1}(\beta_i)$; where surgery means the same as before, but we now allow surgery along arcs (not just loops), which means that we have to remove neighborhoods on both ends of the arcs instead of just the base point.

We already understand the set of arcs on $\phi^{-1}(\alpha_i)$ pretty well. We can identify the arcs from $\phi^{-1}(\alpha_i)$ with the connected components of the $\{i, i+1\}$-subgraph of $G_m$. The isolated points of this graph correspond to arc components being mapped $1-1$ onto $\alpha_i$, and the other connected graphs components correspond to what we will refer to as \emph{triple arcs}; arc components of $\phi^{-1}(\alpha_i)$ that cover $3:1$ onto $\alpha_i$ under $\phi$.

Now, what's left is to understand the preimages of $\beta_i$ and how they interact with the preimages of $\alpha_i$. Remember that $\beta_i$ is not defined for $i = m$, so the next proposition is for all $i = 1, \dots, 2m-1$ except $i=m$.

\begin{prop}\label{prop:beta_i_representation}
The permutation associated with the arc $\beta_i$ is a disjoint product of $8m+2$ transpositions and a 3-cycle.
\end{prop}
\begin{proof}
 Recall that $\beta_i$ is a loop on $D^2 \backslash P$, so it has an associated permutation under the representation given by $G_m$.
 
 The first observation is that $\beta_i$ is a conjugation of $\sigma_{2m} \cdot \sigma_{2m-5}$ if $i<m$ and of $\sigma_1 \cdot \sigma_6$ if $i>m$. This means that it has the same permutation structure as that of $\sigma_{2m} \cdot \sigma_{2m-5}$ or $\sigma_1 \cdot \sigma_6$. By looking at the $\{1,6\}$-subgraph of  $G_m$, we can see that $\sigma_1 \cdot \sigma_6$ is a disjoint product of $8m+2$ transpositions and a 3-cycle: the $\{1,6\}$-subgraph is formed with $8m+2$ disjoint edges except for a pair coming from the identification made on I6 (on definition of $G_m$). The same happens when $i<m$, $\beta_i$  is a product of $8m+2$ disjoint transpositions and a 3-cycle. 
 \end{proof}

The previous observation helps us to get a much better picture of the arcs involved in the preimage of $\beta_i$. As in definition \ref{def:graph-and-permuations}, we can represent $\beta_i$ as a new set of directed edges in $G_m$; in fact, the $(8m+2)$ transpositions can be represented as undirected edges and the one 3-cycle as three directed edges. We add this new set of edges to the $\{i,i+1\}$-subgraph of $G_m$, let $\hat{G}^i_m$ be the resulting graph. Observe that $\hat{G}^i_m$ only has edges with colors $i$, $i+1$, and $\beta_i$ (three of which are directed). By removing the $\beta_i$-colored arcs of $\hat{G}^i_m$, we obtain $G^i_m$ (the $\{i,i+1\}$-subgraph of $G_m$).



\begin{figure}
    \centering
    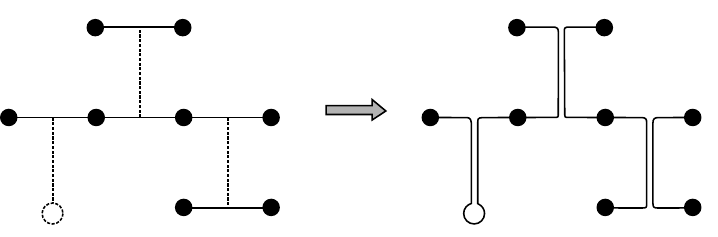
    \caption{Preimages of $\beta_i$ coming out of a triple arc.}
    \label{fig:surgery-to-triple-arc-along-beta-arcs}
\end{figure}

We now proceed to draw the edges corresponding to the permutation of $\beta_i$. For simplicity, we start with the case  $i>m$, the case $i<m$ can be constructed using the color-symmetry of $G_m$ (by $f(i) = 2m+1-i$) and the fact that $\beta_i$ is the symmetric of $\beta_{2m+1-i}$. 


To draw edges representing the $\beta_i$-permutation on $G_m$ we choose any point $v$ in the graph $G_m$, and then we move through the edges following the sequence of colors: 
\begin{equation}\label{sequence-of-colors}
i\ ,\ i-1\ ,\ \dots\ ,\ 3\ ,\ 2\ ,\ 1\ ,\ 6\ ,\ 2\ ,\ 3\ ,\ \dots\ ,\ i-1\ ,\ i
\end{equation}
(recall that $\beta_i = \sigma_i \cdot \sigma_{i-1} \cdots \sigma_2 \cdot \sigma_1 \cdot \sigma_6 \cdot \sigma^{-1}_2 \cdot \sigma^{-1}_3 \cdots \sigma^{-1}_i$). By ``moving'', we mean to move to a contiguous vertex through an edge with the color in turn, and if there is none, we stay at the current vertex (we say, it got fixed). After following the whole sequence, we will end up at a vertex $w$. If $v = w$, we say that the lifting at $v$ is a loop, and we don't add anything to the graph $G^i_m$. But if $v \neq w$, we draw a directed edge connecting $v$ with $w$, and we will tag it with $\beta_i$. We said that the lifting of $\beta_i$ on $v$ ends at $w$. We replace any pair of directed $\beta_i$-tagged edges $[w,v]$ and $[w,v]$ with a single undirected $\beta_i$-tagged edge.  

Let $H$ be a connected component of $\hat{G}^i_m$. Observe that if $H$ contains only isolated vertices of $G^i_m$, then after surgery along $\beta_i$ we will obtain arcs covering 1:1 and onto $\gamma_i$. So  Part (1) of Proposition \ref{prop:gamma_i-properties} is valid in this case. Now, the remaining case is when $H$ contains a three-vertex component of $G^i_m$.

Now, by the way that $G_m$ is constructed, each three-vertices component of $G^i_m$ (the $\{i,i+1\}$-subgraph) is contained in a unique copy $L^k_{2m}$ of $L_{2m}$. So, the  non-isolated vertices components of  $G^i_m$ are made of three vertices ($v^k_{i-1}, v^k_i$ and $v^k_{i+1}$) and two edges: the $i$-colored edge $[v^k_{i-1}, v^k_i]$  and the $i+1$-colored edge $[v^k_{i}, v^k_{i+1}]$. 

The first thing we are going to prove is that the lift of $\beta_i$ at $v^k_i$ is an isolated vertex in $G^i_k$. And later, we will work with the lift of $\beta_i$ at $v^k_{i-1}$ and $v^k_{ i+1}$; which, in most cases, is a loop or ends at an isolated vertex of $G^i_m$. So, for most cases, the connected components look like in Figure \ref{fig:surgery-to-triple-arc-along-beta-arcs}, implying that after surgery, for most cases, we have only arcs that are mapped 3:1 or 1:1 onto $\gamma_i$.  But we will discuss this in further detail later.

\subsubsection{Computing the lift at $v^k_i$}

We now compute the lifting of $\beta_i$ at vertex $v^k_i$. Starting at  $v^k_i$ we follow the sequence of colors \eqref{sequence-of-colors}. After following the first part of the sequence $i, i-1, \dots, 2, 1$, we will end up at the vertex $v^k_0 \in L^k_{2m}$; it doesn't matter if there are vertices identified in the middle of $L^k_{2m}$ the sequence will lead us to the same place because at each movement of sub-sequence we find the color in turn.  If there are no identifications on $G_m$ involving $v^k_0$, after following the rest of the sequence $6, 2, 3, \dots, i$ we remain still on $v^k_0$. Implying that the lift of $\beta_i$ at $v^k_i$ ends at $v^k_0$; an isolated vertex of $G^i_m$ as we claimed. 

The only value of $k$ where $v^k_0$ is identified with some other vertex is when $k=1$. Identifications I1 and I6 say that $v^1_0 \in L^1$ is identified with $v^2_3 \in L^2$ and $v^{4m+1}_6 \in L^{4m+1}$ (See Figure \ref{fig:identification_v0_v3_v6}). Now starting on $v^k_0$ we follow the rest of the sequence $6, 2, 3, 4, \dots, i$. It is not hard to see that after following the first five terms ($6, 2, 3, 4, and 5$), we end at the vertex $v^{4m+1}_4$, which is not part of any identification on $G_m$. So, we will stay at $v^{4m+1}_4$ the rest of the sequence: $6, 7, \dots, i$. This proves that, in the case $k=1$, the lifting of $\beta_i$ at $v^k_i$  is also an isolated vertex of $\{i, i+1\}$-subgraph as well. This completes the computation of the lift of $\beta_i$ at $v^k_i$. See Table \ref{tab:lifting-at-vk_i} for a resume of what we just computed. 

\begin{figure}
    \centering
    \includegraphics{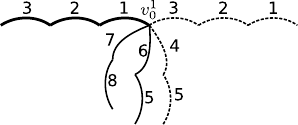}
    \caption{Identification of $v_0$ with $v_3$ and $v_6$}
    \label{fig:identification_v0_v3_v6}
\end{figure}

\begin{table}
    \centering
    \begin{tabular}{|c|c|c|}
    \hline
    When &   starting at & ends at \\
    \hline
     $k \neq 1$     & $v^k_i$& $v^k_0$ \\ 
     \hline
     $k = 1$     & $v^1_i$& $v^{4m+1}_4$\\
     \hline
    \end{tabular}
    \caption{The lift of $\beta_i$ at the vertex $v^k_i$}
    \label{tab:lifting-at-vk_i}
\end{table}

\subsubsection{Computing the lift at $v^k_{i-1}$ and $v^k_{i+1}$}

Now we proceed to compute the lift at $v^k_{i-1}$, the lift at $v^k_{i+1}$ solves similarly. So, we start by stepping on a vertex $v^k_{i-1}$, and then we follow the sequence \eqref{sequence-of-colors}. Observe that we move immediately from $v^k_{i-1}$ to $v^k_i$.  If there are no identifications on $G_m$ involving $v^k_i$ we will stay there all the rest of the sequence until we reach the color $i$ again, which will take us back to $v^k_{i-1}$. Implying, in this case, that the lifting of $\beta_i$ at $v^k_{i-1}$ is a loop.

So, the difficulty is when $v^k_i$ is part of one of the identifications on the construction of $G_m$. The only possible cases where $v^k_i$ with $2m > i >m$ is identified with some other vertex are listed below:
\begin{enumerate}
    \item On I2 $v^{i+2}_i$ is identified with $v^{i+3}_{i+3}$ with $ m +1 \leq i \leq 2m-3$,
    \item On I2 $v^i_i$ is identified with $v^{i-1}_{i-3}$ when $m + 2 \leq i \leq 2m-1$ or
    \item On I3 $v^{2m}_{2m-2}$ is identified with $v^{1}_{2}$, so this is only for $i=2m-2$.
    \item On I5 $v^{i+1}_{i}$ is identified with $v^{2m+i+1}_{5}$ with $ m +1 \leq i \leq 2m-1$
    \item On I6 $v^{4m+2}_{2m-6}$ is identified with $v^{2m}_{2m}$, this is only for $i=2m-6$.
    \item On I4 $v^{2m+j}_{2m-5}$ is identified with $v^j_j$, this only applies for $i=2m-5$ and $ m \geq j \geq 1$.
\end{enumerate}

We are going to work on each case separately.

\textbf{Case: $k=i+2$ and $i \leq 2m-3$}. In this case, $v^{i+2}_i$ is identified with $v^{i+3}_{i+3}$ using I2 (change $i$ with $i+2$).
Observe that by following the sequence \eqref{sequence-of-colors} starting at $v^{i+2}_{i-1}$ we will move first to $w = v^{i+2}_i = v^{i+3}_{i+3}$ which is either a vertex with 4 edges of colors $i$, $i+1$, $i+3$, and $i+4$ if $i\leq 2m-4$ or a valence 5 vertex for $i=2m-3$. In the former case, none of the colors around  $w$ appears on the sequence \eqref{sequence-of-colors}, meaning that we will stay on $w$ the whole sequence, but at the end (on color $i$), we move back to $v^{i+2}_{i-1}$. So, the lifting of $\beta_i$, in this case, is a loop.

For the case when $i=2m-3$, we observe that on $v^{i+3}_{i+3} = v^{2m}_{2m}$ there is also another identification by Rule 6, so we have now the identification with vertex $v^{4m+2}_{2m-6}$ implying that the colors around $w = v^{2m-1}_{2m-3} = v^{2m}_{2m} = v^{4m+2}_{2m-6} $ are $2m-3, 2m-2, 2m, 2m-6$, and $2m-5$. Implying that when following the sequence  \eqref{sequence-of-colors} (with $i=2m-3$) at $v^{2m-3}_{2m-4}$ we will first move to $w$ and then to $v^{4m+2}_{2m-5}$ and we will stay there during the sub-sequence $2m-6, 2m-7, \dots 2, 1, 6, 2, 3, \dots 2m-6$ and the final part of the sequence $2m-5, 2m-4, 2m-3$ we move back to $v^{2m-3}_{2m-4}$ (See Figure \ref{fig:case-vertex_2m-4_L2m-1-beta2m-3}). So, it is also a loop in this case.

\begin{figure}
    \centering
    \includegraphics{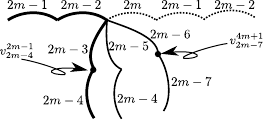}
    \caption{Local view of vertex $v^{2m-1}_{2m-3}$ on $G_m$}
    \label{fig:case-vertex_2m-4_L2m-1-beta2m-3}
\end{figure}

\textbf{Case: $k=i$ and $i \geq m+2$}. In this case, $v^i_i$ is identified with $v^{i-1}_{i-3}$ in I2 (change $i$ with $i-1$). A neighborhood of $v^i_i$ is shown in Figure \ref{fig:neighborhood-of-v_i-in-L_i}. Again, we just start at $v^k_{i-1}= v^i_{i-1}$ and follow the sequence \eqref{sequence-of-colors}, and it is not hard to convince ourselves that we will get the following path: 
$$v^i_{i-1} \to v^i_i =_{\small{I2}} v^{i-1}_{i-3} \to v^{i-1}_{i-2}=_{\small{I5}}v^{2m+i-1}_5 \to v^{2m+i-1}_6 \to v^{2m+i-1}_5 \to v^{2m+i-1}_4$$

This means that the lift of $\beta_i$ ends at an isolated vertex in the $\{i,i+1\}$-subgraph.

\begin{figure}
    \centering
    \includegraphics{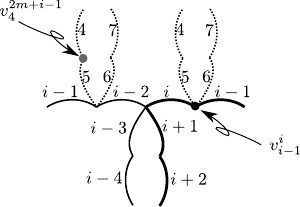}
    \caption{Neighborhood of $v^i_i$ for $i>m$}
    \label{fig:neighborhood-of-v_i-in-L_i}
\end{figure}

\textbf{Case: $i=2m-2$ and $k=2m$}. In this case, $v^{2m}_{2m-2}$ is identified with $v^{1}_{2}$ using I3. Again we will start at $v^k_{i-1} = v^{2m}_{2m-3}$ and follow the main sequence \eqref{sequence-of-colors} (with $i=2m-2$). For that, we observe that at the vertex $v^{2m}_{2m-2}$, the graph $G_m$ looks like in Figure \ref{fig:neighborhood-of-v_2m-2-in-L_2m}. We can easily check that the resulting sequence of vertices is $$v^{2m}_{2m-3} \to v^{2m}_{2m-2}= v^{1}_{2}  \to v^1_{3} \to v^{1}_{2} = v^{2m}_{2m-2} \to v^{2m}_{2m-3}$$

This means that the lifting of $\beta_i$ (for $i=2m-1$) at $v^{2m}_{2m-2}$ is a loop.

\begin{figure}
    \centering
    \includegraphics{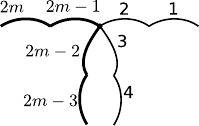}
    \caption{Neighborhood of $v^{2m}_{2m-2}$.}
    \label{fig:neighborhood-of-v_2m-2-in-L_2m}
\end{figure}

\textbf{Case: $k=i+1$ and $i \leq 2m-1$}. In this case, $v^{i+1}_{i}$ is identified with $v^{2m+i+1}_{5}$ by  I5 (change $i$ with $i+1$).
Now in Figure \ref{fig:neighborhood-of-v_i-in-L_i+1}, we have drawn a neighborhood of $v^{i+1}_{i}$. Now, starting at vertex $v^{i+1}_{i-1}$ we follow the sequence of colors \eqref{sequence-of-colors}, and we get the path:
$$v^{i+1}_{i-1} \to v^{i+1}_{i}= v^{2m+i+1}_{5}  \to v^{2m+i+1}_{6} \to v^{2m+i+1}_{5} \to  v^{2m+i+1}_{4}.$$
We will end up at $v^{2m+i+1}_{4}$. So, the lift of $\beta_i$ at $v^{i+1}_{i-1}$ ends at the isolated vertex $v^{2m+i+1}_{4}$.
\begin{figure}
    \centering
    \includegraphics{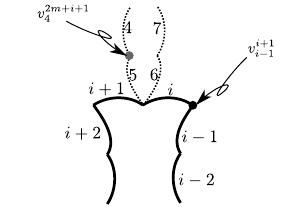}
    \caption{Neighborhood of $v^{i+1}_{i}$.}
    \label{fig:neighborhood-of-v_i-in-L_i+1}
\end{figure}

\textbf{Case: $i=2m-6$ and $k=4m+2$}. In this case, $v^{4m+2}_{2m-6}$ is identified with $v^{2m}_{2m}$ using I6. But recall that at that vertex, we also have the identification given by I2: $v^{2m}_{2m}$ is identified with $v^{2m-1}_{2m-3}$. So, locally we are in the same situation as in Figure \ref{fig:case-vertex_2m-4_L2m-1-beta2m-3} but now we are starting at $v^{4m+2}_{2m-7}$ (recall that $m \geq 8$, so $i-1=2m-7 > m$). Observe that the colors around $w = v^{4m+2}_{2m-6} = v^{2m-1}_{2m-3}=v^{2m}_{2m}$ are all bigger than $2m-6$. So, when we follow the sequence of colors \eqref{sequence-of-colors} (for $i=2m-6$), we are going to move first to $w$ and stay there the rest of the sequence, except that, by the end, we move back to $v^{4m+2}_{2m-7}$. Implying that the lift of $\beta_i$ at $v^{4m+2}_{2m-7}$ is a loop.

\textbf{Case: $k=2m+j$, $i = 2m-5$ and $j =1, \dots, m$}. In this case, $v^k_i$ is identified with $v^j_j$ using I4. We have many sub-cases, but most of them look like in figure \ref{fig:neighborhood-of-v_j-in-L_j}. If we follow the first part of the color sequence \eqref{sequence-of-colors} (with $i=2m-5$) starting at $v^k_{i-1} = v^{2m+j}_{2m-6}$ we will move all the way to vertex $v^{j-1}_{j-1}$. If neither $j$ or $j-1$ is equal to $1$ or $6$, then it will loop back to the vertex $v^k_{i-1}$ after following the rest of the sequence. The cases where it fails to come back are when $j= 2, 6, 7$. In these cases, the path ends at $v^1_{2m-5}$, $v^4_2$ and $v^6_4$, respectively. These vertices are isolated in $G^i_m$ except for $v^1_{2m-5}$, which is the middle point of a triple arc (recall that $i=2m-5$).

\begin{figure}
    \centering
    \includegraphics{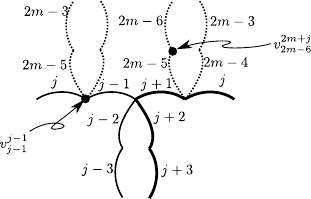}
    \caption{Neighborhood of $v^{2m+j}_{2m-5}$.}
    \label{fig:neighborhood-of-v_j-in-L_j}
\end{figure}

\bgroup
\def\arraystretch{1.5}%
\begin{table}
    \centering
    \begin{tabular}{| m{14em} | c | c |}
    \hline
    When &   starting at $v^k_{i-1}$& ends at \\
    \hline 
     $k=i+2$ and $ m< i < 2m-2$   & $v^{i+2}_{i-1}$ & $v^{i+2}_{i-1}$ \\ 
    \hline
     $k=i$ and $m+1 < i < 2m$   & $v^{i}_{i-1}$ & $v^{2m+i-1}_4$ \\ 
    \hline
     $k=2m$ and $i=2m-1$   & $v^{2m}_{2m-2}$ & $v^{2m}_{2m-2}$ \\ 
    \hline
     $k=i+1$ and $m<i<2m$   & $v^{i+1}_{i-1}$ & $v^{2m+i+1}_{4}$ \\ 
    \hline
     $k=4m+2$ and $i=2m-6$   & $v^{4m+2}_{2m-7}$ & $v^{4m+2}_{2m-7}$ \\ 
    \hline
     $k=2m+j$, $i=2m-5$ and $j \neq 2, 6, 7$ with $j \leq m$ & $v^{2m+j}_{2m-6}$ & $v^{2m+j}_{2m-6}$ \\ 
    \hline
     $k=2m+2$, $i=2m-5$ & $v^{2m+2}_{2m-6}$ & $v^1_{2m-5}$* \\ 
    \hline
     $k=2m+6$, $i=2m-5$ & $v^{2m+6}_{2m-6}$ & $v^4_2$ \\ 
    \hline
     $k=2m+7$, $i=2m-5$ & $v^{2m+7}_{2m-6}$ & $v^6_4$ \\ 
    \hline
     Other cases ($v^k_i$ is not identified with anyone)  & $v^{k}_{i-1}$ & $v^{k}_{i-1}$ \\ 
    \hline
    \end{tabular}
    \caption{The lifting of $\beta_i$ at the vertex $v^{k}_{i-1}$}
    \label{tab:lifting-at-vk_i-1}
\end{table}
\egroup

We have covered all the cases for the lift of $ \beta_i$ at $v^k_{i-1}$. Table \ref{tab:lifting-at-vk_i-1} shows a resume for all our computations. Similarly, we can compute the lift of $\beta_i$ at the vertex $v^k_{i+1}$; the result is shown in Table \ref{tab:lifting-at-vk_i+1}.

\bgroup
\def\arraystretch{1.5}%
\begin{table}
    \centering
    \begin{tabular}{|m{14em}|c|c|}
    \hline
    When &   starting at $v^k_{i+1}$ & ends at \\
    \hline
     $k=i+3$ and $ m< i < 2m-4$   & $v^{i+3}_{i+1}$ & $v^{i+3}_{i+1}$ \\ 
    \hline
     $k=i+3$ and $ i= 2m-4$   & $v^{2m-1}_{2m-3} (=v^{2m}_{2m} = v^{4m+2}_{2m-5}) $ & $v^{2m-1}_{2m-3}$ \\
    \hline
     $k=i+3$ and $ i= 2m-3$   & $v^{2m}_{2m-2} (=v^{1}_{2}) $ & $v^{2m}_{2m-2}$ \\
    \hline
     $k=i+1$ and$ m< i < 2m-1$   & $v^{i+1}_{i+1}$ & $v^{2m+i}_4$ \\ 
    \hline
     $k=i+1$ and $i=2m-1$   & $v^{2m}_{2m} (=v^{2m-1}_{2m-3}= v^{4m+2}_{2m-5})$ & $v^{4m-1}_4$ \\ 
    \hline
     $k=4m+2$ and $i=2m-7$   & $v^{4m+2}_{2m-6}$ & $v^{4m+2}_{2m-6}$ \\ 
    \hline    
     $k=2m+j$ and $i=2m-6$ with $j \neq 2,6,7$ and $j<m$  & $v^{2m+j}_{2m-5} (=v^{j}_{j})$ & $v^{4m-1}_4$ \\ 
    \hline
     $k=2m+2$ and $i=2m-6$  & $v^{2m+2}_{2m-5} (=v^{2}_{2})$ & $v^{1}_{2m-6}$* \\ 
    \hline
     $k=2m+6$ and $i=2m-6$  & $v^{2m+6}_{2m-5} (=v^{2}_{2})$ & $v^{5}_{2}$ \\ 
    \hline
     $k=2m+7$ and $i=2m-6$  & $v^{2m+7}_{2m-5} (=v^{2}_{2})$ & $v^{7}_{4}$ \\ 
    \hline
     $v^k_{i+1}$ is not identified   & $v^{k}_{i+1}$ & $v^{k}_{i+1}$ \\ 
    \hline
    \end{tabular}
    \caption{The lifting of $\beta_i$ at the vertex $v^{k}_{i+1}$ }
    \label{tab:lifting-at-vk_i+1}
\end{table}
\egroup


\subsubsection{Proving part 1 of Proposition \ref{prop:gamma_i-properties}}
Now, from Tables \ref{tab:lifting-at-vk_i}, \ref{tab:lifting-at-vk_i-1}, and \ref{tab:lifting-at-vk_i+1}, we can easily prove part 1 of Proposition \ref{prop:gamma_i-properties}. We need to prove that the square half-twist around $\gamma_i$ lifts to the covering disk. By Theorem  \ref{prop:lifting-alpha-twist-conditions}, it will be enough to prove that the connected components of the preimages of $\gamma_i$ cover either 2:1 or 1:1 onto $\gamma_i$. Those preimages are obtained by performing surgery on the preimages of $\alpha_i$ along the preimages of $\beta_i$.

Recall that the preimages of $\alpha_i$ are single arcs (1:1 onto it) and triple arcs (3:1 and onto $\alpha_i$). We also know that the preimages of $\beta_i$ are all 2:1 (represent transposition) and just one triple arc (represents a 3-cycle). 

Let us start with the case when only single arcs of $\alpha_i$ are connected with preimages of $\beta_i$. After surgery, we obtain only single arcs covering $\gamma_i$ (see Figure \ref{fig:surgery-to-single-arcs-along-beta-arcs}).

\begin{figure}
    \centering
    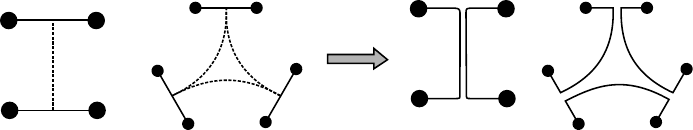
    \caption{Surgery to single $\alpha_i$ arcs along of $\beta_i$ preimages}
    \label{fig:surgery-to-single-arcs-along-beta-arcs}
\end{figure}

Now the interesting case is when we perform surgery on triple arcs of $\alpha_i$. In  this case, we have to look at Tables \ref{tab:lifting-at-vk_i}, \ref{tab:lifting-at-vk_i-1}, and \ref{tab:lifting-at-vk_i+1}, which give us the lift of $\beta_i$ at the three possible starting points of a triple arc. 

Observe first that the lift of $\beta_i$ at $v^k_i$ is always an isolated vertex of $G^i_m$, and at $v^k_{i-1}$ or $v^k_{i+1}$ it is, in most cases, a loop or an isolated vertex of $G^i_m$. Let us assume that no 3-cycle of $\beta_i$ is involved. Then the surgery along $\beta_i$ will separate the triple arc into two double arcs (that cover 2:1 onto $\gamma_i$) and some extra single arcs. In Figure \ref{fig:surgery-to-triple-arc-along-beta-arcs}, there is depicted an example when the lift at $v^k_{i-1}$ is a loop and at $v^k_{i+1}$ is an isolated vertex. 

The remaining cases are when there is a 3-cycle or when the lift at $v^k_{i-1}$ or $v^k_{i+1}$ is another not isolated vertex. On Tables \ref{tab:lifting-at-vk_i-1} and \ref{tab:lifting-at-vk_i+1} we marked the later cases with an asterisk. 

If there is a 3-cycle connecting a triple arc with two single arcs, the conclusion is the same as before; we get only double or single arcs after surgery. Because in this case, the surgery along a 3-cycle is basically the same as that on transposition. It is similar to the phenomenon in Figure \ref{fig:surgery-to-triple-arc-along-beta-arcs} but replacing one single arc with a triple arc. 

So, the next case is when a 3-cycle $\beta_i$ connects at least two triple arcs of $\alpha_i$. First, observe that the case with three triple arcs does not occur. Because, in these cases, a lift marked with an asterisk must occur (see Tables \ref{tab:lifting-at-vk_i}, \ref{tab:lifting-at-vk_i-1}, and \ref{tab:lifting-at-vk_i+1}). But the ones marked with asterisk end at the middle of a triple arc, and the lift on middle arcs is always a single arc (an isolated vertex in $G^i_m$, see Table \ref{tab:lifting-at-vk_i} ). So, we are left with the case when a 3-cycle is connecting two triple arcs and one single arc. But moreover, this case can only come from the asterisk cases on Tables \ref{tab:lifting-at-vk_i-1} and \ref{tab:lifting-at-vk_i+1}. 

After analyzing the two cases marked with an asterisk, we can observe that the 3-cycle follows the sequence of vertices: $v^k_{i-1} = v^{2m+1}_{2m-6} \to v^{k'}_{i} = v^1_{2m-5} \to v^{4m+1}_4 $ and loops back to $v^{2m+1}_{2m-6}$. As this case is for $i = 2m-5$, it means that $v^{2m+1}_{2m-6}$ represents one side of a triple arc, $v^1_{2m-5}$ the middle of a triple arc, and $v^{4m+1}_4$ represents a single arc. Now we need to check the lifting on the other four vertices of the two triple arcs, that is, on $v^{2m+1}_{2m-5}, v^{2m+1}_{2m-4}, v^1_{2m-4} $ and $v^1_{2m-6}$. From the tables, we can see that the lift on those vertices is a loop except on $v^{2m+1}_{2m-5}$ which is an isolated vertex $v^{2m+1}_{0}$. From this configuration, we can see that after surgery we get only single or double arcs covering $\gamma_i$ (See Figure \ref{fig:surgery-to-triple-arc-with-asterisk})

\begin{figure}
    \centering
    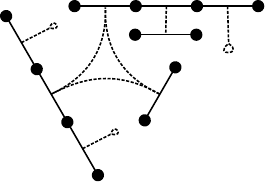
    \caption{Connections between triple arcs involved in the asterisk case of Table \ref{tab:lifting-at-vk_i-1}.}
    \label{fig:surgery-to-triple-arc-with-asterisk}
\end{figure}

We can work out the other asterisk case similarly. And this completes all the cases, so we have completed the proof for part 1 of Proposition \ref{prop:gamma_i-properties}.

\subsubsection{Proving parts 2-4 of Proposition \ref{prop:gamma_i-properties}}

To prove the other three parts of Proposition \ref{prop:gamma_i-properties}, we need to involve $Q$, the selected subset of $\phi^{-1}(\{p_1, \dots, p_m\})$. Recall that $Q$ has two preimages of $p_i$, and both are regular points (not real branching points). Recall that the ends of $\gamma_i$ are $p_i$ and $p_{i+1}$, and we have shown that the connected components of $\phi^{-1}(\gamma_i)$ are double arcs or single arcs.

The points of $Q$ that are connected by $\phi^{-1}(\gamma_i)$ are the ones corresponding to preimages of $p_i$ and $p_{i+1}$. Those vertices are precisely $$v^{i}_{i-2}(i) = v^{i+1}_{i+1}(i),\ v^{i+1}_{i-1}(i+1),\ v^m_{m-1}(i),\ \textrm{and }\  v^m_{m-1}(i+1)$$ when $2m-1> i > m$  (when $0<i<m-1$, we use the symmetric vertices $v^i_{i+1}(i), v^{i+1}_{i+2}(i+1) = v^{i}_{i-1}(i+2), v^{m+1}_{m+1}(i)$, and $v^{m+1}_{m+1}(i+1)$). We already know that these vertices are the ends of $\widetilde{\alpha}_i$ and $\overline{\alpha}_i$, so we only need to study how they get modified by $\phi^{-1}(\beta_i)$ . 

We will show next, that the surgery along the lifts of $\beta_i$ will break $\widetilde{\alpha}_i$ into three arcs, two of which are single arcs that contain one end on $Q$ and the other end out of $Q$.  It is enough to compute the lift at the vertices $v^{i+1}_{i}, v^{i+1}_{i-1}$ and $v^{i+1}_{i+1}$ of $G_m^i$ . But we already computed this on Tables \ref{tab:lifting-at-vk_i}, \ref{tab:lifting-at-vk_i-1}, and \ref{tab:lifting-at-vk_i+1}, and it is always an isolated vertex. This implies that indeed $\beta_i$ breaks $\widetilde{\alpha}_i$ into three arcs, as we claimed. This gives us two components of $\phi^{-1}(\gamma_i)$ intersecting $Q$, which are two of the three components described in part 2 of Proposition \ref{prop:gamma_i-properties}. And these two components clearly satisfy part 4 of Proposition \ref{prop:gamma_i-properties}; each component contains only one vertex in $Q$. 

To get the other component, we have to look at the lift of $\beta_i$ on $\overline{\alpha}_i$. Recall that $\overline{\alpha}_i$ has ends $ v^m_{m-1}(i)$ and $v^m_{m-1}(i+1)$. This means that we need to compute the lift of $\beta_i$ at $v^{m}_{m-1}$. After computing the lift, we found that it is a loop, which concludes the proof of part 2. 

Now for part 3 of \ref{prop:gamma_i-properties}, observe that the sequence of colors of $\beta_i$ only travels from $v^m_{m-1}$ to $v^m_{m}$ and back to $v^m_{m-1}$. This means that the arc $\beta_i$ wraps around $v^m_{m}(1)$ and $v^m_{m}(6)$, which are points that do not belong to $Q$. Implying that $\beta_i$ bounds a disk in $D^2-Q$, so it is isotopic to its base point. This finally implies that the lift of $\gamma_i$ is also isotopic to $\overline{\alpha}_i$ relative to the boundary as stated in part 3.

\bibliographystyle{plain}
\bibliography{bibliografia}
\end{document}

%% file: images/disk_generators.pdf_tex
\begingroup%
  \makeatletter%
  \providecommand\color[2][]{%
    \errmessage{(Inkscape) Color is used for the text in Inkscape, but the package 'color.sty' is not loaded}%
    \renewcommand\color[2][]{}%
  }%
  \providecommand\transparent[1]{%
    \errmessage{(Inkscape) Transparency is used (non-zero) for the text in Inkscape, but the package 'transparent.sty' is not loaded}%
    \renewcommand\transparent[1]{}%
  }%
  \providecommand\rotatebox[2]{#2}%
  \newcommand*\fsize{\dimexpr\f@size pt\relax}%
  \newcommand*\lineheight[1]{\fontsize{\fsize}{#1\fsize}\selectfont}%
  \ifx\svgwidth\undefined%
    \setlength{\unitlength}{122.64218668bp}%
    \ifx\svgscale\undefined%
      \relax%
    \else%
      \setlength{\unitlength}{\unitlength * \real{\svgscale}}%
    \fi%
  \else%
    \setlength{\unitlength}{\svgwidth}%
  \fi%
  \global\let\svgwidth\undefined%
  \global\let\svgscale\undefined%
  \makeatother%
  \begin{picture}(1,0.80217638)%
    \lineheight{1}%
    \setlength\tabcolsep{0pt}%
    \put(0,0){\includegraphics[width=\unitlength,page=1]{disk_generators.pdf}}%
    \put(0.64141574,0.5280462){\color[rgb]{0,0,0}\rotatebox{-53.058763}{\makebox(0,0)[t]{\lineheight{1.25}\smash{\begin{tabular}[t]{c}$\dots$\end{tabular}}}}}%
    \put(0.38132866,0.17729259){\color[rgb]{0,0,0}\makebox(0,0)[rt]{\lineheight{1.25}\smash{\begin{tabular}[t]{r}$\sigma_1$\end{tabular}}}}%
    \put(0.2579418,0.38689056){\color[rgb]{0,0,0}\makebox(0,0)[rt]{\lineheight{1.25}\smash{\begin{tabular}[t]{r}$\sigma_2$\end{tabular}}}}%
    \put(0.70449633,0.38689056){\color[rgb]{0,0,0}\makebox(0,0)[lt]{\lineheight{1.25}\smash{\begin{tabular}[t]{l}$\sigma_{n-1}$\end{tabular}}}}%
    \put(0.61143691,0.17729259){\color[rgb]{0,0,0}\makebox(0,0)[lt]{\lineheight{1.25}\smash{\begin{tabular}[t]{l}$\sigma_{n}$\end{tabular}}}}%
    \put(0.32468109,0.62859082){\color[rgb]{0,0,0}\makebox(0,0)[rt]{\lineheight{1.25}\smash{\begin{tabular}[t]{r}$\sigma_3$\end{tabular}}}}%
    \put(0.50678645,0.74196064){\color[rgb]{0,0,0}\makebox(0,0)[t]{\lineheight{1.25}\smash{\begin{tabular}[t]{c}$\sigma_4$\end{tabular}}}}%
    \put(0,0){\includegraphics[width=\unitlength,page=2]{disk_generators.pdf}}%
  \end{picture}%
\endgroup%

%% file: images/ejemplo1a.pdf_tex
\begingroup%
  \makeatletter%
  \providecommand\color[2][]{%
    \errmessage{(Inkscape) Color is used for the text in Inkscape, but the package 'color.sty' is not loaded}%
    \renewcommand\color[2][]{}%
  }%
  \providecommand\transparent[1]{%
    \errmessage{(Inkscape) Transparency is used (non-zero) for the text in Inkscape, but the package 'transparent.sty' is not loaded}%
    \renewcommand\transparent[1]{}%
  }%
  \providecommand\rotatebox[2]{#2}%
  \newcommand*\fsize{\dimexpr\f@size pt\relax}%
  \newcommand*\lineheight[1]{\fontsize{\fsize}{#1\fsize}\selectfont}%
  \ifx\svgwidth\undefined%
    \setlength{\unitlength}{161.16082619bp}%
    \ifx\svgscale\undefined%
      \relax%
    \else%
      \setlength{\unitlength}{\unitlength * \real{\svgscale}}%
    \fi%
  \else%
    \setlength{\unitlength}{\svgwidth}%
  \fi%
  \global\let\svgwidth\undefined%
  \global\let\svgscale\undefined%
  \makeatother%
  \begin{picture}(1,0.68158956)%
    \lineheight{1}%
    \setlength\tabcolsep{0pt}%
    \put(0,0){\includegraphics[width=\unitlength,page=1]{ejemplo1a.pdf}}%
    \put(0.25010612,0.27915119){\color[rgb]{0,0,0}\makebox(0,0)[t]{\lineheight{1.25}\smash{\begin{tabular}[t]{c}$1$\end{tabular}}}}%
    \put(0.5036003,0.27598639){\color[rgb]{0,0,0}\makebox(0,0)[t]{\lineheight{1.25}\smash{\begin{tabular}[t]{c}$3$\end{tabular}}}}%
    \put(0.38192335,0.46950804){\color[rgb]{0,0,0}\makebox(0,0)[lt]{\lineheight{1.25}\smash{\begin{tabular}[t]{l}$2$\end{tabular}}}}%
    \put(0.38405152,0.17832473){\color[rgb]{0,0,0}\makebox(0,0)[lt]{\lineheight{1.25}\smash{\begin{tabular}[t]{l}$4$\end{tabular}}}}%
    \put(0.77035757,0.26899635){\color[rgb]{0,0,0}\makebox(0,0)[t]{\lineheight{1.25}\smash{\begin{tabular}[t]{c}$1$\end{tabular}}}}%
    \put(0.08576014,0.34374787){\color[rgb]{0,0,0}\makebox(0,0)[rt]{\lineheight{1.25}\smash{\begin{tabular}[t]{r}$V_1$\end{tabular}}}}%
    \put(0.35434525,0.36077513){\color[rgb]{0,0,0}\makebox(0,0)[rt]{\lineheight{1.25}\smash{\begin{tabular}[t]{r}$V_2$\end{tabular}}}}%
    \put(0.34925705,0.59836525){\color[rgb]{0,0,0}\makebox(0,0)[rt]{\lineheight{1.25}\smash{\begin{tabular}[t]{r}$V_6$\end{tabular}}}}%
    \put(0.61567253,0.35565239){\color[rgb]{0,0,0}\makebox(0,0)[rt]{\lineheight{1.25}\smash{\begin{tabular}[t]{r}$V_3$\end{tabular}}}}%
    \put(0.91527336,0.34091446){\color[rgb]{0,0,0}\makebox(0,0)[lt]{\lineheight{1.25}\smash{\begin{tabular}[t]{l}$V_4$\end{tabular}}}}%
    \put(0.34650194,0.05712886){\color[rgb]{0,0,0}\makebox(0,0)[rt]{\lineheight{1.25}\smash{\begin{tabular}[t]{r}$V_5$\end{tabular}}}}%
  \end{picture}%
\endgroup%

%% file: images/ejemplo1b.pdf_tex
\begingroup%
  \makeatletter%
  \providecommand\color[2][]{%
    \errmessage{(Inkscape) Color is used for the text in Inkscape, but the package 'color.sty' is not loaded}%
    \renewcommand\color[2][]{}%
  }%
  \providecommand\transparent[1]{%
    \errmessage{(Inkscape) Transparency is used (non-zero) for the text in Inkscape, but the package 'transparent.sty' is not loaded}%
    \renewcommand\transparent[1]{}%
  }%
  \providecommand\rotatebox[2]{#2}%
  \newcommand*\fsize{\dimexpr\f@size pt\relax}%
  \newcommand*\lineheight[1]{\fontsize{\fsize}{#1\fsize}\selectfont}%
  \ifx\svgwidth\undefined%
    \setlength{\unitlength}{161.16931465bp}%
    \ifx\svgscale\undefined%
      \relax%
    \else%
      \setlength{\unitlength}{\unitlength * \real{\svgscale}}%
    \fi%
  \else%
    \setlength{\unitlength}{\svgwidth}%
  \fi%
  \global\let\svgwidth\undefined%
  \global\let\svgscale\undefined%
  \makeatother%
  \begin{picture}(1,1.27748137)%
    \lineheight{1}%
    \setlength\tabcolsep{0pt}%
    \put(0,0){\includegraphics[width=\unitlength,page=1]{ejemplo1b.pdf}}%
    \put(0.20008521,0.8928846){\color[rgb]{0,0,0}\makebox(0,0)[lt]{\lineheight{1.25}\smash{\begin{tabular}[t]{l}1\end{tabular}}}}%
    \put(0.72011978,0.8928846){\color[rgb]{0,0,0}\makebox(0,0)[lt]{\lineheight{1.25}\smash{\begin{tabular}[t]{l}1\end{tabular}}}}%
    \put(0.40251344,0.63278312){\color[rgb]{0,0,0}\makebox(0,0)[lt]{\lineheight{1.25}\smash{\begin{tabular}[t]{l}1\end{tabular}}}}%
    \put(0.4020132,1.14321452){\color[rgb]{0,0,0}\makebox(0,0)[lt]{\lineheight{1.25}\smash{\begin{tabular}[t]{l}1\end{tabular}}}}%
    \put(0.52444082,0.0924567){\color[rgb]{0,0,0}\makebox(0,0)[lt]{\lineheight{1.25}\smash{\begin{tabular}[t]{l}1\end{tabular}}}}%
    \put(0.48629945,0.04860076){\color[rgb]{0,0,0}\makebox(0,0)[lt]{\lineheight{1.25}\smash{\begin{tabular}[t]{l}2\end{tabular}}}}%
    \put(0.43911825,0.09248099){\color[rgb]{0,0,0}\makebox(0,0)[lt]{\lineheight{1.25}\smash{\begin{tabular}[t]{l}3\end{tabular}}}}%
    \put(0.48581477,0.12933994){\color[rgb]{0,0,0}\makebox(0,0)[lt]{\lineheight{1.25}\smash{\begin{tabular}[t]{l}4\end{tabular}}}}%
    \put(0.09409025,0.85275615){\color[rgb]{0,0,0}\makebox(0,0)[lt]{\lineheight{1.25}\smash{\begin{tabular}[t]{l}2\end{tabular}}}}%
    \put(0.43911825,0.09248099){\color[rgb]{0,0,0}\makebox(0,0)[lt]{\lineheight{1.25}\smash{\begin{tabular}[t]{l}3\end{tabular}}}}%
    \put(0.04654205,0.89290889){\color[rgb]{0,0,0}\makebox(0,0)[lt]{\lineheight{1.25}\smash{\begin{tabular}[t]{l}3\end{tabular}}}}%
    \put(0.09360557,0.92893502){\color[rgb]{0,0,0}\makebox(0,0)[lt]{\lineheight{1.25}\smash{\begin{tabular}[t]{l}4\end{tabular}}}}%
    \put(0.3555978,0.58925352){\color[rgb]{0,0,0}\makebox(0,0)[lt]{\lineheight{1.25}\smash{\begin{tabular}[t]{l}2\end{tabular}}}}%
    \put(0.30828161,0.6328074){\color[rgb]{0,0,0}\makebox(0,0)[lt]{\lineheight{1.25}\smash{\begin{tabular}[t]{l}3\end{tabular}}}}%
    \put(0.35511312,0.72899687){\color[rgb]{0,0,0}\makebox(0,0)[lt]{\lineheight{1.25}\smash{\begin{tabular}[t]{l}4\end{tabular}}}}%
    \put(0.3555978,1.05853219){\color[rgb]{0,0,0}\makebox(0,0)[lt]{\lineheight{1.25}\smash{\begin{tabular}[t]{l}2\end{tabular}}}}%
    \put(0.30777708,1.14323874){\color[rgb]{0,0,0}\makebox(0,0)[lt]{\lineheight{1.25}\smash{\begin{tabular}[t]{l}3\end{tabular}}}}%
    \put(0.35511312,1.18387844){\color[rgb]{0,0,0}\makebox(0,0)[lt]{\lineheight{1.25}\smash{\begin{tabular}[t]{l}4\end{tabular}}}}%
    \put(0.50688824,0.89290889){\color[rgb]{0,0,0}\makebox(0,0)[lt]{\lineheight{1.25}\smash{\begin{tabular}[t]{l}3\end{tabular}}}}%
    \put(0.61813215,0.85275615){\color[rgb]{0,0,0}\makebox(0,0)[lt]{\lineheight{1.25}\smash{\begin{tabular}[t]{l}2\end{tabular}}}}%
    \put(0.6176476,0.92505645){\color[rgb]{0,0,0}\makebox(0,0)[lt]{\lineheight{1.25}\smash{\begin{tabular}[t]{l}4\end{tabular}}}}%
    \put(0.87079377,0.92860788){\color[rgb]{0,0,0}\makebox(0,0)[lt]{\lineheight{1.25}\smash{\begin{tabular}[t]{l}2\end{tabular}}}}%
    \put(0.92170735,0.89290889){\color[rgb]{0,0,0}\makebox(0,0)[lt]{\lineheight{1.25}\smash{\begin{tabular}[t]{l}3\end{tabular}}}}%
    \put(0.87030909,0.85308329){\color[rgb]{0,0,0}\makebox(0,0)[lt]{\lineheight{1.25}\smash{\begin{tabular}[t]{l}4\end{tabular}}}}%
  \end{picture}%
\endgroup%

%% file: images/arc_twist.pdf_tex
\begingroup%
  \makeatletter%
  \providecommand\color[2][]{%
    \errmessage{(Inkscape) Color is used for the text in Inkscape, but the package 'color.sty' is not loaded}%
    \renewcommand\color[2][]{}%
  }%
  \providecommand\transparent[1]{%
    \errmessage{(Inkscape) Transparency is used (non-zero) for the text in Inkscape, but the package 'transparent.sty' is not loaded}%
    \renewcommand\transparent[1]{}%
  }%
  \providecommand\rotatebox[2]{#2}%
  \newcommand*\fsize{\dimexpr\f@size pt\relax}%
  \newcommand*\lineheight[1]{\fontsize{\fsize}{#1\fsize}\selectfont}%
  \ifx\svgwidth\undefined%
    \setlength{\unitlength}{245.035732bp}%
    \ifx\svgscale\undefined%
      \relax%
    \else%
      \setlength{\unitlength}{\unitlength * \real{\svgscale}}%
    \fi%
  \else%
    \setlength{\unitlength}{\svgwidth}%
  \fi%
  \global\let\svgwidth\undefined%
  \global\let\svgscale\undefined%
  \makeatother%
  \begin{picture}(1,0.34869548)%
    \lineheight{1}%
    \setlength\tabcolsep{0pt}%
    \put(0,0){\includegraphics[width=\unitlength,page=1]{arc_twist.pdf}}%
    \put(0.50001528,0.24766722){\color[rgb]{0,0,0}\makebox(0,0)[t]{\lineheight{1.25}\smash{\begin{tabular}[t]{c}$\tau_{\alpha}$\end{tabular}}}}%
    \put(0.50001811,0.1567439){\color[rgb]{0,0,0}\makebox(0,0)[t]{\lineheight{1.25}\smash{\begin{tabular}[t]{c}$\to$\end{tabular}}}}%
  \end{picture}%
\endgroup%

%% file: images/ejemplo_L2_phi3.pdf_tex
\begingroup%
  \makeatletter%
  \providecommand\color[2][]{%
    \errmessage{(Inkscape) Color is used for the text in Inkscape, but the package 'color.sty' is not loaded}%
    \renewcommand\color[2][]{}%
  }%
  \providecommand\transparent[1]{%
    \errmessage{(Inkscape) Transparency is used (non-zero) for the text in Inkscape, but the package 'transparent.sty' is not loaded}%
    \renewcommand\transparent[1]{}%
  }%
  \providecommand\rotatebox[2]{#2}%
  \newcommand*\fsize{\dimexpr\f@size pt\relax}%
  \newcommand*\lineheight[1]{\fontsize{\fsize}{#1\fsize}\selectfont}%
  \ifx\svgwidth\undefined%
    \setlength{\unitlength}{255.92236136bp}%
    \ifx\svgscale\undefined%
      \relax%
    \else%
      \setlength{\unitlength}{\unitlength * \real{\svgscale}}%
    \fi%
  \else%
    \setlength{\unitlength}{\svgwidth}%
  \fi%
  \global\let\svgwidth\undefined%
  \global\let\svgscale\undefined%
  \makeatother%
  \begin{picture}(1,0.47165744)%
    \lineheight{1}%
    \setlength\tabcolsep{0pt}%
    \put(0,0){\includegraphics[width=\unitlength,page=1]{ejemplo_L2_phi3.pdf}}%
    \put(0.7697961,0.23344912){\color[rgb]{0,0,0}\makebox(0,0)[lt]{\lineheight{1.25}\smash{\begin{tabular}[t]{l}$\phi_3$\end{tabular}}}}%
    \put(0.77930186,0.05695498){\color[rgb]{0,0,0}\makebox(0,0)[lt]{\lineheight{1.25}\smash{\begin{tabular}[t]{l}1\end{tabular}}}}%
    \put(0.72833449,0.05674887){\color[rgb]{0,0,0}\makebox(0,0)[lt]{\lineheight{1.25}\smash{\begin{tabular}[t]{l}2\end{tabular}}}}%
    \put(0.76292697,0.0166259){\color[rgb]{0,0,0}\makebox(0,0)[t]{\lineheight{1.25}\smash{\begin{tabular}[t]{c}$\alpha$\end{tabular}}}}%
    \put(0.64743742,0.40140025){\color[rgb]{0,0,0}\makebox(0,0)[lt]{\lineheight{1.25}\smash{\begin{tabular}[t]{l}1\end{tabular}}}}%
    \put(0.55736273,0.40119419){\color[rgb]{0,0,0}\makebox(0,0)[lt]{\lineheight{1.25}\smash{\begin{tabular}[t]{l}2\end{tabular}}}}%
    \put(0.85904577,0.40119419){\color[rgb]{0,0,0}\makebox(0,0)[lt]{\lineheight{1.25}\smash{\begin{tabular}[t]{l}2\end{tabular}}}}%
    \put(0.95137615,0.40140025){\color[rgb]{0,0,0}\makebox(0,0)[lt]{\lineheight{1.25}\smash{\begin{tabular}[t]{l}1\end{tabular}}}}%
    \put(0.76059118,0.43243357){\color[rgb]{0,0,0}\makebox(0,0)[t]{\lineheight{1.25}\smash{\begin{tabular}[t]{c}$\widetilde{\alpha}$\end{tabular}}}}%
    \put(0.11689693,0.2362469){\color[rgb]{0,0,0}\makebox(0,0)[t]{\lineheight{1.25}\smash{\begin{tabular}[t]{c}$1$\end{tabular}}}}%
    \put(0.28892948,0.2362469){\color[rgb]{0,0,0}\makebox(0,0)[t]{\lineheight{1.25}\smash{\begin{tabular}[t]{c}$2$\end{tabular}}}}%
    \put(0.20036759,0.29257085){\color[rgb]{0,0,0}\makebox(0,0)[t]{\lineheight{1.25}\smash{\begin{tabular}[t]{c}$v_1$\end{tabular}}}}%
    \put(0.35860998,0.29257085){\color[rgb]{0,0,0}\makebox(0,0)[t]{\lineheight{1.25}\smash{\begin{tabular}[t]{c}$v_2$\end{tabular}}}}%
    \put(0.0443022,0.29257085){\color[rgb]{0,0,0}\makebox(0,0)[t]{\lineheight{1.25}\smash{\begin{tabular}[t]{c}$v_0$\end{tabular}}}}%
    \put(0.20145609,0.39386565){\color[rgb]{0,0,0}\makebox(0,0)[t]{\lineheight{1.25}\smash{\begin{tabular}[t]{c}$L_2$\end{tabular}}}}%
  \end{picture}%
\endgroup%

%% file: images/arc_alpha_connecting_1_and_2.pdf_tex
\begingroup%
  \makeatletter%
  \providecommand\color[2][]{%
    \errmessage{(Inkscape) Color is used for the text in Inkscape, but the package 'color.sty' is not loaded}%
    \renewcommand\color[2][]{}%
  }%
  \providecommand\transparent[1]{%
    \errmessage{(Inkscape) Transparency is used (non-zero) for the text in Inkscape, but the package 'transparent.sty' is not loaded}%
    \renewcommand\transparent[1]{}%
  }%
  \providecommand\rotatebox[2]{#2}%
  \newcommand*\fsize{\dimexpr\f@size pt\relax}%
  \newcommand*\lineheight[1]{\fontsize{\fsize}{#1\fsize}\selectfont}%
  \ifx\svgwidth\undefined%
    \setlength{\unitlength}{122.64218668bp}%
    \ifx\svgscale\undefined%
      \relax%
    \else%
      \setlength{\unitlength}{\unitlength * \real{\svgscale}}%
    \fi%
  \else%
    \setlength{\unitlength}{\svgwidth}%
  \fi%
  \global\let\svgwidth\undefined%
  \global\let\svgscale\undefined%
  \makeatother%
  \begin{picture}(1,0.80217638)%
    \lineheight{1}%
    \setlength\tabcolsep{0pt}%
    \put(0,0){\includegraphics[width=\unitlength,page=1]{arc_alpha_connecting_1_and_2.pdf}}%
    \put(0.64141574,0.5280462){\color[rgb]{0,0,0}\rotatebox{-53.058763}{\makebox(0,0)[t]{\lineheight{1.25}\smash{\begin{tabular}[t]{c}$\dots$\end{tabular}}}}}%
    \put(0.34229138,0.08514573){\color[rgb]{0,0,0}\makebox(0,0)[rt]{\lineheight{1.25}\smash{\begin{tabular}[t]{r}$1$\end{tabular}}}}%
    \put(0.2579418,0.38689056){\color[rgb]{0,0,0}\makebox(0,0)[rt]{\lineheight{1.25}\smash{\begin{tabular}[t]{r}$2$\end{tabular}}}}%
    \put(0.70449633,0.38689056){\color[rgb]{0,0,0}\makebox(0,0)[lt]{\lineheight{1.25}\smash{\begin{tabular}[t]{l}$n-1$\end{tabular}}}}%
    \put(0.61143691,0.17729259){\color[rgb]{0,0,0}\makebox(0,0)[lt]{\lineheight{1.25}\smash{\begin{tabular}[t]{l}$n$\end{tabular}}}}%
    \put(0.32468109,0.62859082){\color[rgb]{0,0,0}\makebox(0,0)[rt]{\lineheight{1.25}\smash{\begin{tabular}[t]{r}$3$\end{tabular}}}}%
    \put(0.50678645,0.74196064){\color[rgb]{0,0,0}\makebox(0,0)[t]{\lineheight{1.25}\smash{\begin{tabular}[t]{c}$4$\end{tabular}}}}%
    \put(0.26724449,0.22566897){\color[rgb]{0,0,0}\makebox(0,0)[lt]{\lineheight{1.25}\smash{\begin{tabular}[t]{l}$\alpha$\end{tabular}}}}%
  \end{picture}%
\endgroup%

%% file: images/triangular-region.pdf_tex
\begingroup%
  \makeatletter%
  \providecommand\color[2][]{%
    \errmessage{(Inkscape) Color is used for the text in Inkscape, but the package 'color.sty' is not loaded}%
    \renewcommand\color[2][]{}%
  }%
  \providecommand\transparent[1]{%
    \errmessage{(Inkscape) Transparency is used (non-zero) for the text in Inkscape, but the package 'transparent.sty' is not loaded}%
    \renewcommand\transparent[1]{}%
  }%
  \providecommand\rotatebox[2]{#2}%
  \newcommand*\fsize{\dimexpr\f@size pt\relax}%
  \newcommand*\lineheight[1]{\fontsize{\fsize}{#1\fsize}\selectfont}%
  \ifx\svgwidth\undefined%
    \setlength{\unitlength}{122.64218668bp}%
    \ifx\svgscale\undefined%
      \relax%
    \else%
      \setlength{\unitlength}{\unitlength * \real{\svgscale}}%
    \fi%
  \else%
    \setlength{\unitlength}{\svgwidth}%
  \fi%
  \global\let\svgwidth\undefined%
  \global\let\svgscale\undefined%
  \makeatother%
  \begin{picture}(1,0.80217638)%
    \lineheight{1}%
    \setlength\tabcolsep{0pt}%
    \put(0,0){\includegraphics[width=\unitlength,page=1]{triangular-region.pdf}}%
    \put(0.66587696,0.50358497){\color[rgb]{0,0,0}\rotatebox{-53.058763}{\makebox(0,0)[t]{\lineheight{1.25}\smash{\begin{tabular}[t]{c}$\dots$\end{tabular}}}}}%
    \put(0.10165286,0.4860783){\color[rgb]{0,0,0}\makebox(0,0)[rt]{\lineheight{1.25}\smash{\begin{tabular}[t]{r}$\alpha$\end{tabular}}}}%
    \put(0.3050801,0.52300973){\color[rgb]{0,0,0}\makebox(0,0)[lt]{\lineheight{1.25}\smash{\begin{tabular}[t]{l}$\sigma_2$\end{tabular}}}}%
    \put(0,0){\includegraphics[width=\unitlength,page=2]{triangular-region.pdf}}%
    \put(0.34504087,0.07052491){\color[rgb]{0,0,0}\makebox(0,0)[rt]{\lineheight{1.25}\smash{\begin{tabular}[t]{r}$\sigma_1$\end{tabular}}}}%
  \end{picture}%
\endgroup%

%% file: images/example-L4.pdf_tex
\begingroup%
  \makeatletter%
  \providecommand\color[2][]{%
    \errmessage{(Inkscape) Color is used for the text in Inkscape, but the package 'color.sty' is not loaded}%
    \renewcommand\color[2][]{}%
  }%
  \providecommand\transparent[1]{%
    \errmessage{(Inkscape) Transparency is used (non-zero) for the text in Inkscape, but the package 'transparent.sty' is not loaded}%
    \renewcommand\transparent[1]{}%
  }%
  \providecommand\rotatebox[2]{#2}%
  \newcommand*\fsize{\dimexpr\f@size pt\relax}%
  \newcommand*\lineheight[1]{\fontsize{\fsize}{#1\fsize}\selectfont}%
  \ifx\svgwidth\undefined%
    \setlength{\unitlength}{361.3932524bp}%
    \ifx\svgscale\undefined%
      \relax%
    \else%
      \setlength{\unitlength}{\unitlength * \real{\svgscale}}%
    \fi%
  \else%
    \setlength{\unitlength}{\svgwidth}%
  \fi%
  \global\let\svgwidth\undefined%
  \global\let\svgscale\undefined%
  \makeatother%
  \begin{picture}(1,0.39356525)%
    \lineheight{1}%
    \setlength\tabcolsep{0pt}%
    \put(0,0){\includegraphics[width=\unitlength,page=1]{example-L4.pdf}}%
    \put(0.07192847,0.15674502){\color[rgb]{0,0,0}\makebox(0,0)[t]{\lineheight{1.25}\smash{\begin{tabular}[t]{c}$1$\end{tabular}}}}%
    \put(0.15469634,0.15674502){\color[rgb]{0,0,0}\makebox(0,0)[t]{\lineheight{1.25}\smash{\begin{tabular}[t]{c}$2$\end{tabular}}}}%
    \put(0.2932622,0.15674502){\color[rgb]{0,0,0}\makebox(0,0)[lt]{\lineheight{1.25}\smash{\begin{tabular}[t]{l}$4$\end{tabular}}}}%
    \put(0.23051573,0.15674502){\color[rgb]{0,0,0}\makebox(0,0)[t]{\lineheight{1.25}\smash{\begin{tabular}[t]{c}$3$\end{tabular}}}}%
    \put(0.11208767,0.1838435){\color[rgb]{0,0,0}\makebox(0,0)[t]{\lineheight{1.25}\smash{\begin{tabular}[t]{c}$v_1$\end{tabular}}}}%
    \put(0.18822088,0.1838435){\color[rgb]{0,0,0}\makebox(0,0)[t]{\lineheight{1.25}\smash{\begin{tabular}[t]{c}$v_2$\end{tabular}}}}%
    \put(0.26733926,0.1838435){\color[rgb]{0,0,0}\makebox(0,0)[t]{\lineheight{1.25}\smash{\begin{tabular}[t]{c}$v_3$\end{tabular}}}}%
    \put(0.34825093,0.1838435){\color[rgb]{0,0,0}\makebox(0,0)[t]{\lineheight{1.25}\smash{\begin{tabular}[t]{c}$v_4$\end{tabular}}}}%
    \put(0.03700186,0.1838435){\color[rgb]{0,0,0}\makebox(0,0)[t]{\lineheight{1.25}\smash{\begin{tabular}[t]{c}$v_0$\end{tabular}}}}%
    \put(0.70539495,0.19419844){\color[rgb]{0,0,0}\makebox(0,0)[lt]{\lineheight{1.25}\smash{\begin{tabular}[t]{l}$\phi_5$\end{tabular}}}}%
    \put(0.74728461,0.03054131){\color[rgb]{0,0,0}\makebox(0,0)[lt]{\lineheight{1.25}\smash{\begin{tabular}[t]{l}$\alpha$\end{tabular}}}}%
    \put(0.49613754,0.32795999){\color[rgb]{0,0,0}\makebox(0,0)[lt]{\lineheight{1.25}\smash{\begin{tabular}[t]{l}1\end{tabular}}}}%
    \put(0.45940437,0.30631882){\color[rgb]{0,0,0}\makebox(0,0)[lt]{\lineheight{1.25}\smash{\begin{tabular}[t]{l}2\end{tabular}}}}%
    \put(0.64598931,0.32781406){\color[rgb]{0,0,0}\makebox(0,0)[lt]{\lineheight{1.25}\smash{\begin{tabular}[t]{l}2\end{tabular}}}}%
    \put(0.67268271,0.30872621){\color[rgb]{0,0,0}\makebox(0,0)[lt]{\lineheight{1.25}\smash{\begin{tabular}[t]{l}1\end{tabular}}}}%
    \put(0.58325607,0.34894066){\color[rgb]{0,0,0}\makebox(0,0)[t]{\lineheight{1.25}\smash{\begin{tabular}[t]{c}$\widetilde{\alpha}$\end{tabular}}}}%
    \put(0.71315016,0.05691989){\color[rgb]{0,0,0}\makebox(0,0)[lt]{\lineheight{1.25}\smash{\begin{tabular}[t]{l}1\end{tabular}}}}%
    \put(0.69614045,0.0373616){\color[rgb]{0,0,0}\makebox(0,0)[lt]{\lineheight{1.25}\smash{\begin{tabular}[t]{l}2\end{tabular}}}}%
    \put(0.67509915,0.05693072){\color[rgb]{0,0,0}\makebox(0,0)[lt]{\lineheight{1.25}\smash{\begin{tabular}[t]{l}3\end{tabular}}}}%
    \put(0.6959243,0.07336858){\color[rgb]{0,0,0}\makebox(0,0)[lt]{\lineheight{1.25}\smash{\begin{tabular}[t]{l}4\end{tabular}}}}%
    \put(0.67509915,0.05693072){\color[rgb]{0,0,0}\makebox(0,0)[lt]{\lineheight{1.25}\smash{\begin{tabular}[t]{l}3\end{tabular}}}}%
    \put(0.43388806,0.31871465){\color[rgb]{0,0,0}\makebox(0,0)[lt]{\lineheight{1.25}\smash{\begin{tabular}[t]{l}3\end{tabular}}}}%
    \put(0.46095561,0.33968888){\color[rgb]{0,0,0}\makebox(0,0)[lt]{\lineheight{1.25}\smash{\begin{tabular}[t]{l}4\end{tabular}}}}%
    \put(0.59547815,0.30830771){\color[rgb]{0,0,0}\makebox(0,0)[lt]{\lineheight{1.25}\smash{\begin{tabular}[t]{l}3\end{tabular}}}}%
    \put(0.55486625,0.30669323){\color[rgb]{0,0,0}\makebox(0,0)[lt]{\lineheight{1.25}\smash{\begin{tabular}[t]{l}4\end{tabular}}}}%
    \put(0.73728591,0.32317845){\color[rgb]{0,0,0}\makebox(0,0)[lt]{\lineheight{1.25}\smash{\begin{tabular}[t]{l}3\end{tabular}}}}%
    \put(0.70836385,0.30994293){\color[rgb]{0,0,0}\makebox(0,0)[lt]{\lineheight{1.25}\smash{\begin{tabular}[t]{l}4\end{tabular}}}}%
    \put(0.85651861,0.32164225){\color[rgb]{0,0,0}\makebox(0,0)[lt]{\lineheight{1.25}\smash{\begin{tabular}[t]{l}4\end{tabular}}}}%
    \put(0.83097046,0.30872621){\color[rgb]{0,0,0}\makebox(0,0)[lt]{\lineheight{1.25}\smash{\begin{tabular}[t]{l}1\end{tabular}}}}%
    \put(0.7957154,0.30908618){\color[rgb]{0,0,0}\makebox(0,0)[lt]{\lineheight{1.25}\smash{\begin{tabular}[t]{l}2\end{tabular}}}}%
    \put(0.91631143,0.34502777){\color[rgb]{0,0,0}\makebox(0,0)[lt]{\lineheight{1.25}\smash{\begin{tabular}[t]{l}1\end{tabular}}}}%
    \put(0.94588122,0.31292748){\color[rgb]{0,0,0}\makebox(0,0)[lt]{\lineheight{1.25}\smash{\begin{tabular}[t]{l}2\end{tabular}}}}%
    \put(0.91877682,0.30257464){\color[rgb]{0,0,0}\makebox(0,0)[lt]{\lineheight{1.25}\smash{\begin{tabular}[t]{l}3\end{tabular}}}}%
    \put(0.86437567,0.26316403){\color[rgb]{0,0,0}\makebox(0,0)[t]{\lineheight{1.25}\smash{\begin{tabular}[t]{c}$\overline{\alpha}_1$\end{tabular}}}}%
    \put(0.97524779,0.37168712){\color[rgb]{0,0,0}\makebox(0,0)[lt]{\lineheight{1.25}\smash{\begin{tabular}[t]{l}$\overline{\alpha}_2$\end{tabular}}}}%
    \put(0.19079781,0.22331283){\color[rgb]{0,0,0}\makebox(0,0)[t]{\lineheight{1.25}\smash{\begin{tabular}[t]{c}$L_4$\end{tabular}}}}%
  \end{picture}%
\endgroup%

%% file: images/n-consecutive-colored-graph-example.pdf_tex
\begingroup%
  \makeatletter%
  \providecommand\color[2][]{%
    \errmessage{(Inkscape) Color is used for the text in Inkscape, but the package 'color.sty' is not loaded}%
    \renewcommand\color[2][]{}%
  }%
  \providecommand\transparent[1]{%
    \errmessage{(Inkscape) Transparency is used (non-zero) for the text in Inkscape, but the package 'transparent.sty' is not loaded}%
    \renewcommand\transparent[1]{}%
  }%
  \providecommand\rotatebox[2]{#2}%
  \newcommand*\fsize{\dimexpr\f@size pt\relax}%
  \newcommand*\lineheight[1]{\fontsize{\fsize}{#1\fsize}\selectfont}%
  \ifx\svgwidth\undefined%
    \setlength{\unitlength}{163.17482776bp}%
    \ifx\svgscale\undefined%
      \relax%
    \else%
      \setlength{\unitlength}{\unitlength * \real{\svgscale}}%
    \fi%
  \else%
    \setlength{\unitlength}{\svgwidth}%
  \fi%
  \global\let\svgwidth\undefined%
  \global\let\svgscale\undefined%
  \makeatother%
  \begin{picture}(1,0.74914521)%
    \lineheight{1}%
    \setlength\tabcolsep{0pt}%
    \put(0.23743246,0.61093511){\color[rgb]{0,0,0}\makebox(0,0)[t]{\lineheight{1.25}\smash{\begin{tabular}[t]{c}$1$\end{tabular}}}}%
    \put(0.30526785,0.49035659){\color[rgb]{0,0,0}\makebox(0,0)[t]{\lineheight{1.25}\smash{\begin{tabular}[t]{c}$2$\end{tabular}}}}%
    \put(0.61515966,0.49522605){\color[rgb]{0,0,0}\makebox(0,0)[lt]{\lineheight{1.25}\smash{\begin{tabular}[t]{l}$4$\end{tabular}}}}%
    \put(0.46950331,0.49279129){\color[rgb]{0,0,0}\makebox(0,0)[t]{\lineheight{1.25}\smash{\begin{tabular}[t]{c}$3$\end{tabular}}}}%
    \put(0.80721004,0.49653744){\color[rgb]{0,0,0}\makebox(0,0)[lt]{\lineheight{1.25}\smash{\begin{tabular}[t]{l}$5$\end{tabular}}}}%
    \put(0,0){\includegraphics[width=\unitlength,page=1]{n-consecutive-colored-graph-example.pdf}}%
    \put(0.29260796,0.03039635){\color[rgb]{0,0,0}\makebox(0,0)[t]{\lineheight{1.25}\smash{\begin{tabular}[t]{c}$1$\end{tabular}}}}%
    \put(0.16820132,0.08613288){\color[rgb]{0,0,0}\makebox(0,0)[t]{\lineheight{1.25}\smash{\begin{tabular}[t]{c}$2$\end{tabular}}}}%
    \put(0.17035239,0.25420103){\color[rgb]{0,0,0}\makebox(0,0)[t]{\lineheight{1.25}\smash{\begin{tabular}[t]{c}$3$\end{tabular}}}}%
    \put(0.13949324,0.42132625){\color[rgb]{0,0,0}\makebox(0,0)[lt]{\lineheight{1.25}\smash{\begin{tabular}[t]{l}$4$\end{tabular}}}}%
    \put(0.08892034,0.54365133){\color[rgb]{0,0,0}\makebox(0,0)[lt]{\lineheight{1.25}\smash{\begin{tabular}[t]{l}$5$\end{tabular}}}}%
    \put(0.58413352,0.43102448){\color[rgb]{0,0,0}\makebox(0,0)[t]{\lineheight{1.25}\smash{\begin{tabular}[t]{c}$1$\end{tabular}}}}%
    \put(0.57917068,0.27610363){\color[rgb]{0,0,0}\makebox(0,0)[t]{\lineheight{1.25}\smash{\begin{tabular}[t]{c}$2$\end{tabular}}}}%
    \put(0.58473351,0.0961093){\color[rgb]{0,0,0}\makebox(0,0)[t]{\lineheight{1.25}\smash{\begin{tabular}[t]{c}$3$\end{tabular}}}}%
    \put(0.61895498,0.02447979){\color[rgb]{0,0,0}\makebox(0,0)[lt]{\lineheight{1.25}\smash{\begin{tabular}[t]{l}$4$\end{tabular}}}}%
    \put(0.73517509,0.08602765){\color[rgb]{0,0,0}\makebox(0,0)[lt]{\lineheight{1.25}\smash{\begin{tabular}[t]{l}$5$\end{tabular}}}}%
    \put(0.84689056,0.72439382){\color[rgb]{0,0,0}\makebox(0,0)[t]{\lineheight{1.25}\smash{\begin{tabular}[t]{c}$1$\end{tabular}}}}%
    \put(0.96263513,0.61189109){\color[rgb]{0,0,0}\makebox(0,0)[t]{\lineheight{1.25}\smash{\begin{tabular}[t]{c}$2$\end{tabular}}}}%
    \put(0.96404174,0.44967118){\color[rgb]{0,0,0}\makebox(0,0)[t]{\lineheight{1.25}\smash{\begin{tabular}[t]{c}$3$\end{tabular}}}}%
    \put(0.93906059,0.27452896){\color[rgb]{0,0,0}\makebox(0,0)[lt]{\lineheight{1.25}\smash{\begin{tabular}[t]{l}$4$\end{tabular}}}}%
    \put(0.94049172,0.0883367){\color[rgb]{0,0,0}\makebox(0,0)[lt]{\lineheight{1.25}\smash{\begin{tabular}[t]{l}$5$\end{tabular}}}}%
  \end{picture}%
\endgroup%

%% file: images/color-symmetric-graph-example.pdf_tex
\begingroup%
  \makeatletter%
  \providecommand\color[2][]{%
    \errmessage{(Inkscape) Color is used for the text in Inkscape, but the package 'color.sty' is not loaded}%
    \renewcommand\color[2][]{}%
  }%
  \providecommand\transparent[1]{%
    \errmessage{(Inkscape) Transparency is used (non-zero) for the text in Inkscape, but the package 'transparent.sty' is not loaded}%
    \renewcommand\transparent[1]{}%
  }%
  \providecommand\rotatebox[2]{#2}%
  \newcommand*\fsize{\dimexpr\f@size pt\relax}%
  \newcommand*\lineheight[1]{\fontsize{\fsize}{#1\fsize}\selectfont}%
  \ifx\svgwidth\undefined%
    \setlength{\unitlength}{186.78563636bp}%
    \ifx\svgscale\undefined%
      \relax%
    \else%
      \setlength{\unitlength}{\unitlength * \real{\svgscale}}%
    \fi%
  \else%
    \setlength{\unitlength}{\svgwidth}%
  \fi%
  \global\let\svgwidth\undefined%
  \global\let\svgscale\undefined%
  \makeatother%
  \begin{picture}(1,0.2697177)%
    \lineheight{1}%
    \setlength\tabcolsep{0pt}%
    \put(0,0){\includegraphics[width=\unitlength,page=1]{color-symmetric-graph-example.pdf}}%
    \put(0.1397035,0.19255617){\color[rgb]{0,0,0}\makebox(0,0)[t]{\lineheight{1.25}\smash{\begin{tabular}[t]{c}$1$\end{tabular}}}}%
    \put(0.37541207,0.19255617){\color[rgb]{0,0,0}\makebox(0,0)[t]{\lineheight{1.25}\smash{\begin{tabular}[t]{c}$2$\end{tabular}}}}%
    \put(0.25492881,0.10216194){\color[rgb]{0,0,0}\makebox(0,0)[lt]{\lineheight{1.25}\smash{\begin{tabular}[t]{l}$4$\end{tabular}}}}%
    \put(0.59133237,0.19255617){\color[rgb]{0,0,0}\makebox(0,0)[t]{\lineheight{1.25}\smash{\begin{tabular}[t]{c}$3$\end{tabular}}}}%
    \put(0.83368434,0.19255617){\color[rgb]{0,0,0}\makebox(0,0)[t]{\lineheight{1.25}\smash{\begin{tabular}[t]{c}$4$\end{tabular}}}}%
    \put(0.70521575,0.10216194){\color[rgb]{0,0,0}\makebox(0,0)[lt]{\lineheight{1.25}\smash{\begin{tabular}[t]{l}$1$\end{tabular}}}}%
  \end{picture}%
\endgroup%

%% file: images/widetilde_alpha_i.pdf_tex
\begingroup%
  \makeatletter%
  \providecommand\color[2][]{%
    \errmessage{(Inkscape) Color is used for the text in Inkscape, but the package 'color.sty' is not loaded}%
    \renewcommand\color[2][]{}%
  }%
  \providecommand\transparent[1]{%
    \errmessage{(Inkscape) Transparency is used (non-zero) for the text in Inkscape, but the package 'transparent.sty' is not loaded}%
    \renewcommand\transparent[1]{}%
  }%
  \providecommand\rotatebox[2]{#2}%
  \newcommand*\fsize{\dimexpr\f@size pt\relax}%
  \newcommand*\lineheight[1]{\fontsize{\fsize}{#1\fsize}\selectfont}%
  \ifx\svgwidth\undefined%
    \setlength{\unitlength}{120.18738033bp}%
    \ifx\svgscale\undefined%
      \relax%
    \else%
      \setlength{\unitlength}{\unitlength * \real{\svgscale}}%
    \fi%
  \else%
    \setlength{\unitlength}{\svgwidth}%
  \fi%
  \global\let\svgwidth\undefined%
  \global\let\svgscale\undefined%
  \makeatother%
  \begin{picture}(1,0.45403607)%
    \lineheight{1}%
    \setlength\tabcolsep{0pt}%
    \put(0,0){\includegraphics[width=\unitlength,page=1]{widetilde_alpha_i.pdf}}%
    \put(0.15340097,0.36502144){\makebox(0,0)[t]{\lineheight{1.25}\smash{\begin{tabular}[t]{c}$v^i_{i+1}$\end{tabular}}}}%
    \put(0.513922,0.36502144){\makebox(0,0)[t]{\lineheight{1.25}\smash{\begin{tabular}[t]{c}$v^i_{i}$\end{tabular}}}}%
    \put(0.73903244,0.36502144){\makebox(0,0)[lt]{\lineheight{1.25}\smash{\begin{tabular}[t]{l}$v^i_{i-1} = v^{i+1}_{i+2}$\end{tabular}}}}%
    \put(0.50390829,0.08698957){\color[rgb]{0.83137255,0,0}\makebox(0,0)[t]{\lineheight{1.25}\smash{\begin{tabular}[t]{c}$\widetilde{\alpha}_i$\end{tabular}}}}%
  \end{picture}%
\endgroup%

%% file: images/gamma_i.pdf_tex
\begingroup%
  \makeatletter%
  \providecommand\color[2][]{%
    \errmessage{(Inkscape) Color is used for the text in Inkscape, but the package 'color.sty' is not loaded}%
    \renewcommand\color[2][]{}%
  }%
  \providecommand\transparent[1]{%
    \errmessage{(Inkscape) Transparency is used (non-zero) for the text in Inkscape, but the package 'transparent.sty' is not loaded}%
    \renewcommand\transparent[1]{}%
  }%
  \providecommand\rotatebox[2]{#2}%
  \newcommand*\fsize{\dimexpr\f@size pt\relax}%
  \newcommand*\lineheight[1]{\fontsize{\fsize}{#1\fsize}\selectfont}%
  \ifx\svgwidth\undefined%
    \setlength{\unitlength}{309.26323358bp}%
    \ifx\svgscale\undefined%
      \relax%
    \else%
      \setlength{\unitlength}{\unitlength * \real{\svgscale}}%
    \fi%
  \else%
    \setlength{\unitlength}{\svgwidth}%
  \fi%
  \global\let\svgwidth\undefined%
  \global\let\svgscale\undefined%
  \makeatother%
  \begin{picture}(1,0.54275785)%
    \lineheight{1}%
    \setlength\tabcolsep{0pt}%
    \put(0,0){\includegraphics[width=\unitlength,page=1]{gamma_i.pdf}}%
    \put(0.24072404,0.00302541){\color[rgb]{0,0,0}\makebox(0,0)[t]{\lineheight{1.25}\smash{\begin{tabular}[t]{c}$i<m$\end{tabular}}}}%
    \put(0,0){\includegraphics[width=\unitlength,page=2]{gamma_i.pdf}}%
    \put(0.75928612,0.00302541){\color[rgb]{0,0,0}\makebox(0,0)[t]{\lineheight{1.25}\smash{\begin{tabular}[t]{c}$i>m$\end{tabular}}}}%
  \end{picture}%
\endgroup%

%% file: images/example_hopf_stabilized.pdf_tex
\begingroup%
  \makeatletter%
  \providecommand\color[2][]{%
    \errmessage{(Inkscape) Color is used for the text in Inkscape, but the package 'color.sty' is not loaded}%
    \renewcommand\color[2][]{}%
  }%
  \providecommand\transparent[1]{%
    \errmessage{(Inkscape) Transparency is used (non-zero) for the text in Inkscape, but the package 'transparent.sty' is not loaded}%
    \renewcommand\transparent[1]{}%
  }%
  \providecommand\rotatebox[2]{#2}%
  \newcommand*\fsize{\dimexpr\f@size pt\relax}%
  \newcommand*\lineheight[1]{\fontsize{\fsize}{#1\fsize}\selectfont}%
  \ifx\svgwidth\undefined%
    \setlength{\unitlength}{262.5bp}%
    \ifx\svgscale\undefined%
      \relax%
    \else%
      \setlength{\unitlength}{\unitlength * \real{\svgscale}}%
    \fi%
  \else%
    \setlength{\unitlength}{\svgwidth}%
  \fi%
  \global\let\svgwidth\undefined%
  \global\let\svgscale\undefined%
  \makeatother%
  \begin{picture}(1,0.34929746)%
    \lineheight{1}%
    \setlength\tabcolsep{0pt}%
    \put(0,0){\includegraphics[width=\unitlength,page=1]{example_hopf_stabilized.pdf}}%
    \put(0.07965604,0.02776177){\color[rgb]{0,0,0}\makebox(0,0)[t]{\lineheight{1.25}\smash{\begin{tabular}[t]{c}a)\end{tabular}}}}%
    \put(0.62825234,0.00877969){\color[rgb]{0,0,0}\makebox(0,0)[t]{\lineheight{1.25}\smash{\begin{tabular}[t]{c}b)\end{tabular}}}}%
    \put(0.32567577,0.33083419){\color[rgb]{0,0,0}\makebox(0,0)[t]{\lineheight{1.25}\smash{\begin{tabular}[t]{c}${\scriptscriptstyle v^{9}_{7}(9)}$\end{tabular}}}}%
    \put(0.56753304,0.33083417){\makebox(0,0)[t]{\lineheight{1.25}\smash{\begin{tabular}[t]{c}${\scriptscriptstyle v^{15}_{13}(15)}$\end{tabular}}}}%
    \put(0.68966029,0.33191284){\makebox(0,0)[t]{\lineheight{1.25}\smash{\begin{tabular}[t]{c}${\scriptscriptstyle v^{2}_{3}(2)}$\end{tabular}}}}%
    \put(0.93568196,0.33083419){\makebox(0,0)[t]{\lineheight{1.25}\smash{\begin{tabular}[t]{c}${\scriptscriptstyle v^{8}_{9}(8)}$\end{tabular}}}}%
    \put(0.32170076,0.29501964){\color[rgb]{0,0,0}\makebox(0,0)[t]{\lineheight{1.25}\smash{\begin{tabular}[t]{c}${\scriptscriptstyle 1}$\end{tabular}}}}%
    \put(0.36254112,0.29501964){\color[rgb]{0,0,0}\makebox(0,0)[t]{\lineheight{1.25}\smash{\begin{tabular}[t]{c}${\scriptscriptstyle 2}$\end{tabular}}}}%
    \put(0.40338142,0.29501964){\color[rgb]{0,0,0}\makebox(0,0)[t]{\lineheight{1.25}\smash{\begin{tabular}[t]{c}${\scriptscriptstyle 3}$\end{tabular}}}}%
    \put(0.44422178,0.29501964){\color[rgb]{0,0,0}\makebox(0,0)[t]{\lineheight{1.25}\smash{\begin{tabular}[t]{c}${\scriptscriptstyle 4}$\end{tabular}}}}%
    \put(0.4850621,0.29501964){\color[rgb]{0,0,0}\makebox(0,0)[t]{\lineheight{1.25}\smash{\begin{tabular}[t]{c}${\scriptscriptstyle 5}$\end{tabular}}}}%
    \put(0.52590245,0.29501964){\color[rgb]{0,0,0}\makebox(0,0)[t]{\lineheight{1.25}\smash{\begin{tabular}[t]{c}${\scriptscriptstyle 6}$\end{tabular}}}}%
    \put(0.56674283,0.29501964){\color[rgb]{0,0,0}\makebox(0,0)[t]{\lineheight{1.25}\smash{\begin{tabular}[t]{c}${\scriptscriptstyle 7}$\end{tabular}}}}%
    \put(0.6075832,0.29501964){\color[rgb]{0,0,0}\makebox(0,0)[t]{\lineheight{1.25}\smash{\begin{tabular}[t]{c}${\scriptscriptstyle 8}$\end{tabular}}}}%
    \put(0.64842353,0.29501964){\color[rgb]{0,0,0}\makebox(0,0)[t]{\lineheight{1.25}\smash{\begin{tabular}[t]{c}${\scriptscriptstyle 9}$\end{tabular}}}}%
    \put(0.68926381,0.29501964){\color[rgb]{0,0,0}\makebox(0,0)[t]{\lineheight{1.25}\smash{\begin{tabular}[t]{c}${\scriptscriptstyle 10}$\end{tabular}}}}%
    \put(0.7301042,0.29501964){\color[rgb]{0,0,0}\makebox(0,0)[t]{\lineheight{1.25}\smash{\begin{tabular}[t]{c}${\scriptscriptstyle 11}$\end{tabular}}}}%
    \put(0.77094454,0.29501964){\color[rgb]{0,0,0}\makebox(0,0)[t]{\lineheight{1.25}\smash{\begin{tabular}[t]{c}${\scriptscriptstyle 12}$\end{tabular}}}}%
    \put(0.8117848,0.29501964){\color[rgb]{0,0,0}\makebox(0,0)[t]{\lineheight{1.25}\smash{\begin{tabular}[t]{c}${\scriptscriptstyle 13}$\end{tabular}}}}%
    \put(0.85262518,0.29501964){\color[rgb]{0,0,0}\makebox(0,0)[t]{\lineheight{1.25}\smash{\begin{tabular}[t]{c}${\scriptscriptstyle 14}$\end{tabular}}}}%
    \put(0.89346559,0.29501964){\color[rgb]{0,0,0}\makebox(0,0)[t]{\lineheight{1.25}\smash{\begin{tabular}[t]{c}${\scriptscriptstyle 15}$\end{tabular}}}}%
    \put(0.93430591,0.29501964){\color[rgb]{0,0,0}\makebox(0,0)[t]{\lineheight{1.25}\smash{\begin{tabular}[t]{c}${\scriptscriptstyle 16}$\end{tabular}}}}%
    \put(0,0){\includegraphics[width=\unitlength,page=2]{example_hopf_stabilized.pdf}}%
  \end{picture}%
\endgroup%

%% file: images/crossing-signs.pdf_tex
\begingroup%
  \makeatletter%
  \providecommand\color[2][]{%
    \errmessage{(Inkscape) Color is used for the text in Inkscape, but the package 'color.sty' is not loaded}%
    \renewcommand\color[2][]{}%
  }%
  \providecommand\transparent[1]{%
    \errmessage{(Inkscape) Transparency is used (non-zero) for the text in Inkscape, but the package 'transparent.sty' is not loaded}%
    \renewcommand\transparent[1]{}%
  }%
  \providecommand\rotatebox[2]{#2}%
  \newcommand*\fsize{\dimexpr\f@size pt\relax}%
  \newcommand*\lineheight[1]{\fontsize{\fsize}{#1\fsize}\selectfont}%
  \ifx\svgwidth\undefined%
    \setlength{\unitlength}{106.10848735bp}%
    \ifx\svgscale\undefined%
      \relax%
    \else%
      \setlength{\unitlength}{\unitlength * \real{\svgscale}}%
    \fi%
  \else%
    \setlength{\unitlength}{\svgwidth}%
  \fi%
  \global\let\svgwidth\undefined%
  \global\let\svgscale\undefined%
  \makeatother%
  \begin{picture}(1,0.49340288)%
    \lineheight{1}%
    \setlength\tabcolsep{0pt}%
    \put(0,0){\includegraphics[width=\unitlength,page=1]{crossing-signs.pdf}}%
    \put(0.1773805,0.040299){\color[rgb]{0,0,0}\makebox(0,0)[t]{\lineheight{1.25}\smash{\begin{tabular}[t]{c}$+$\end{tabular}}}}%
    \put(0.81641037,0.02244898){\color[rgb]{0,0,0}\makebox(0,0)[t]{\lineheight{1.25}\smash{\begin{tabular}[t]{c}$-$\end{tabular}}}}%
  \end{picture}%
\endgroup%

%% file: images/cycle_braid_Bi.pdf_tex
\begingroup%
  \makeatletter%
  \providecommand\color[2][]{%
    \errmessage{(Inkscape) Color is used for the text in Inkscape, but the package 'color.sty' is not loaded}%
    \renewcommand\color[2][]{}%
  }%
  \providecommand\transparent[1]{%
    \errmessage{(Inkscape) Transparency is used (non-zero) for the text in Inkscape, but the package 'transparent.sty' is not loaded}%
    \renewcommand\transparent[1]{}%
  }%
  \providecommand\rotatebox[2]{#2}%
  \newcommand*\fsize{\dimexpr\f@size pt\relax}%
  \newcommand*\lineheight[1]{\fontsize{\fsize}{#1\fsize}\selectfont}%
  \ifx\svgwidth\undefined%
    \setlength{\unitlength}{85.91382665bp}%
    \ifx\svgscale\undefined%
      \relax%
    \else%
      \setlength{\unitlength}{\unitlength * \real{\svgscale}}%
    \fi%
  \else%
    \setlength{\unitlength}{\svgwidth}%
  \fi%
  \global\let\svgwidth\undefined%
  \global\let\svgscale\undefined%
  \makeatother%
  \begin{picture}(1,0.44345738)%
    \lineheight{1}%
    \setlength\tabcolsep{0pt}%
    \put(0.06790517,0.38309845){\color[rgb]{0,0,0}\makebox(0,0)[t]{\lineheight{1.25}\smash{\begin{tabular}[t]{c}$1$\end{tabular}}}}%
    \put(0,0){\includegraphics[width=\unitlength,page=1]{cycle_braid_Bi.pdf}}%
    \put(0.19359955,0.38309845){\color[rgb]{0,0,0}\makebox(0,0)[t]{\lineheight{1.25}\smash{\begin{tabular}[t]{c}$2$\end{tabular}}}}%
    \put(0.31929392,0.38309845){\color[rgb]{0,0,0}\makebox(0,0)[t]{\lineheight{1.25}\smash{\begin{tabular}[t]{c}$3$\end{tabular}}}}%
    \put(0.4449883,0.38309845){\color[rgb]{0,0,0}\makebox(0,0)[t]{\lineheight{1.25}\smash{\begin{tabular}[t]{c}$4$\end{tabular}}}}%
    \put(0.57068273,0.38309845){\color[rgb]{0,0,0}\makebox(0,0)[t]{\lineheight{1.25}\smash{\begin{tabular}[t]{c}$5$\end{tabular}}}}%
    \put(0.69637704,0.38309845){\color[rgb]{0,0,0}\makebox(0,0)[t]{\lineheight{1.25}\smash{\begin{tabular}[t]{c}$6$\end{tabular}}}}%
    \put(0.82207136,0.38309845){\color[rgb]{0,0,0}\makebox(0,0)[t]{\lineheight{1.25}\smash{\begin{tabular}[t]{c}$7$\end{tabular}}}}%
    \put(0.94725051,0.38309845){\color[rgb]{0,0,0}\makebox(0,0)[t]{\lineheight{1.25}\smash{\begin{tabular}[t]{c}$8$\end{tabular}}}}%
  \end{picture}%
\endgroup%

%% file: images/artin_generators.pdf_tex
\begingroup%
  \makeatletter%
  \providecommand\color[2][]{%
    \errmessage{(Inkscape) Color is used for the text in Inkscape, but the package 'color.sty' is not loaded}%
    \renewcommand\color[2][]{}%
  }%
  \providecommand\transparent[1]{%
    \errmessage{(Inkscape) Transparency is used (non-zero) for the text in Inkscape, but the package 'transparent.sty' is not loaded}%
    \renewcommand\transparent[1]{}%
  }%
  \providecommand\rotatebox[2]{#2}%
  \newcommand*\fsize{\dimexpr\f@size pt\relax}%
  \newcommand*\lineheight[1]{\fontsize{\fsize}{#1\fsize}\selectfont}%
  \ifx\svgwidth\undefined%
    \setlength{\unitlength}{228.78398985bp}%
    \ifx\svgscale\undefined%
      \relax%
    \else%
      \setlength{\unitlength}{\unitlength * \real{\svgscale}}%
    \fi%
  \else%
    \setlength{\unitlength}{\svgwidth}%
  \fi%
  \global\let\svgwidth\undefined%
  \global\let\svgscale\undefined%
  \makeatother%
  \begin{picture}(1,0.32575954)%
    \lineheight{1}%
    \setlength\tabcolsep{0pt}%
    \put(0,0){\includegraphics[width=\unitlength,page=1]{artin_generators.pdf}}%
    \put(0.02886121,0.29154537){\color[rgb]{0,0,0}\makebox(0,0)[t]{\lineheight{1.25}\smash{\begin{tabular}[t]{c}$i$\end{tabular}}}}%
    \put(0.35218494,0.29408956){\color[rgb]{0,0,0}\makebox(0,0)[t]{\lineheight{1.25}\smash{\begin{tabular}[t]{c}$j$\end{tabular}}}}%
    \put(0,0){\includegraphics[width=\unitlength,page=2]{artin_generators.pdf}}%
    \put(0.64783537,0.29154537){\color[rgb]{0,0,0}\makebox(0,0)[t]{\lineheight{1.25}\smash{\begin{tabular}[t]{c}$i$\end{tabular}}}}%
    \put(0.97115911,0.29408956){\color[rgb]{0,0,0}\makebox(0,0)[t]{\lineheight{1.25}\smash{\begin{tabular}[t]{c}$j$\end{tabular}}}}%
    \put(0.49534902,0.12987195){\color[rgb]{0,0,0}\makebox(0,0)[t]{\lineheight{1.25}\smash{\begin{tabular}[t]{c}$\to$\end{tabular}}}}%
  \end{picture}%
\endgroup%

%% file: images/example_step3_artin_decomp.pdf_tex
\begingroup%
  \makeatletter%
  \providecommand\color[2][]{%
    \errmessage{(Inkscape) Color is used for the text in Inkscape, but the package 'color.sty' is not loaded}%
    \renewcommand\color[2][]{}%
  }%
  \providecommand\transparent[1]{%
    \errmessage{(Inkscape) Transparency is used (non-zero) for the text in Inkscape, but the package 'transparent.sty' is not loaded}%
    \renewcommand\transparent[1]{}%
  }%
  \providecommand\rotatebox[2]{#2}%
  \newcommand*\fsize{\dimexpr\f@size pt\relax}%
  \newcommand*\lineheight[1]{\fontsize{\fsize}{#1\fsize}\selectfont}%
  \ifx\svgwidth\undefined%
    \setlength{\unitlength}{197.93980876bp}%
    \ifx\svgscale\undefined%
      \relax%
    \else%
      \setlength{\unitlength}{\unitlength * \real{\svgscale}}%
    \fi%
  \else%
    \setlength{\unitlength}{\svgwidth}%
  \fi%
  \global\let\svgwidth\undefined%
  \global\let\svgscale\undefined%
  \makeatother%
  \begin{picture}(1,0.72378973)%
    \lineheight{1}%
    \setlength\tabcolsep{0pt}%
    \put(0,0){\includegraphics[width=\unitlength,page=1]{example_step3_artin_decomp.pdf}}%
    \put(0.13350034,0.64547263){\makebox(0,0)[rt]{\lineheight{1.25}\smash{\begin{tabular}[t]{r}$C_1 \cdot C_2$\end{tabular}}}}%
    \put(-0.00167743,0.50125355){\makebox(0,0)[lt]{\lineheight{1.25}\smash{\begin{tabular}[t]{l}$A_{9,10}^{-1}$\end{tabular}}}}%
    \put(-0.00167743,0.2708121){\makebox(0,0)[lt]{\lineheight{1.25}\smash{\begin{tabular}[t]{l}$A_{1,10}$\end{tabular}}}}%
    \put(-0.00167743,0.03947773){\makebox(0,0)[lt]{\lineheight{1.25}\smash{\begin{tabular}[t]{l}$A_{9,10}$\end{tabular}}}}%
  \end{picture}%
\endgroup%

%% file: images/resulting-link-after-alpham.pdf_tex
\begingroup%
  \makeatletter%
  \providecommand\color[2][]{%
    \errmessage{(Inkscape) Color is used for the text in Inkscape, but the package 'color.sty' is not loaded}%
    \renewcommand\color[2][]{}%
  }%
  \providecommand\transparent[1]{%
    \errmessage{(Inkscape) Transparency is used (non-zero) for the text in Inkscape, but the package 'transparent.sty' is not loaded}%
    \renewcommand\transparent[1]{}%
  }%
  \providecommand\rotatebox[2]{#2}%
  \newcommand*\fsize{\dimexpr\f@size pt\relax}%
  \newcommand*\lineheight[1]{\fontsize{\fsize}{#1\fsize}\selectfont}%
  \ifx\svgwidth\undefined%
    \setlength{\unitlength}{223.86308805bp}%
    \ifx\svgscale\undefined%
      \relax%
    \else%
      \setlength{\unitlength}{\unitlength * \real{\svgscale}}%
    \fi%
  \else%
    \setlength{\unitlength}{\svgwidth}%
  \fi%
  \global\let\svgwidth\undefined%
  \global\let\svgscale\undefined%
  \makeatother%
  \begin{picture}(1,0.40796154)%
    \lineheight{1}%
    \setlength\tabcolsep{0pt}%
    \put(0,0){\includegraphics[width=\unitlength,page=1]{resulting-link-after-alpham.pdf}}%
    \put(0.1291808,0.01730362){\color[rgb]{0,0,0}\makebox(0,0)[t]{\lineheight{1.25}\smash{\begin{tabular}[t]{c}$C_1$\end{tabular}}}}%
    \put(0,0){\includegraphics[width=\unitlength,page=2]{resulting-link-after-alpham.pdf}}%
    \put(0.47742898,0.14463477){\color[rgb]{0,0,0}\makebox(0,0)[t]{\lineheight{1.25}\smash{\begin{tabular}[t]{c}$L$\end{tabular}}}}%
    \put(0.89040037,0.0068779){\color[rgb]{0,0,0}\makebox(0,0)[t]{\lineheight{1.25}\smash{\begin{tabular}[t]{c}$C_2$\end{tabular}}}}%
  \end{picture}%
\endgroup%

%% file: images/surgery-to-triple-arc-along-beta-arcs.pdf_tex
\begingroup%
  \makeatletter%
  \providecommand\color[2][]{%
    \errmessage{(Inkscape) Color is used for the text in Inkscape, but the package 'color.sty' is not loaded}%
    \renewcommand\color[2][]{}%
  }%
  \providecommand\transparent[1]{%
    \errmessage{(Inkscape) Transparency is used (non-zero) for the text in Inkscape, but the package 'transparent.sty' is not loaded}%
    \renewcommand\transparent[1]{}%
  }%
  \providecommand\rotatebox[2]{#2}%
  \newcommand*\fsize{\dimexpr\f@size pt\relax}%
  \newcommand*\lineheight[1]{\fontsize{\fsize}{#1\fsize}\selectfont}%
  \ifx\svgwidth\undefined%
    \setlength{\unitlength}{344.36690924bp}%
    \ifx\svgscale\undefined%
      \relax%
    \else%
      \setlength{\unitlength}{\unitlength * \real{\svgscale}}%
    \fi%
  \else%
    \setlength{\unitlength}{\svgwidth}%
  \fi%
  \global\let\svgwidth\undefined%
  \global\let\svgscale\undefined%
  \makeatother%
  \begin{picture}(1,0.32973424)%
    \lineheight{1}%
    \setlength\tabcolsep{0pt}%
    \put(0,0){\includegraphics[width=\unitlength,page=1]{surgery-to-triple-arc-along-beta-arcs.pdf}}%
    \put(0.07802001,0.08810328){\color[rgb]{0,0,0}\makebox(0,0)[lt]{\lineheight{1.25}\smash{\begin{tabular}[t]{l}$\beta'$\end{tabular}}}}%
    \put(0.13703027,0.18528313){\color[rgb]{0,0,0}\makebox(0,0)[t]{\lineheight{1.25}\smash{\begin{tabular}[t]{c}$B$\end{tabular}}}}%
    \put(0.25463085,0.18528313){\color[rgb]{0,0,0}\makebox(0,0)[t]{\lineheight{1.25}\smash{\begin{tabular}[t]{c}$A$\end{tabular}}}}%
    \put(0.3768426,0.18528313){\color[rgb]{0,0,0}\makebox(0,0)[t]{\lineheight{1.25}\smash{\begin{tabular}[t]{c}$B$\end{tabular}}}}%
    \put(0.02104758,0.18528313){\color[rgb]{0,0,0}\makebox(0,0)[t]{\lineheight{1.25}\smash{\begin{tabular}[t]{c}$A$\end{tabular}}}}%
    \put(0.32002874,0.08303124){\color[rgb]{0,0,0}\makebox(0,0)[lt]{\lineheight{1.25}\smash{\begin{tabular}[t]{l}$\beta'''$\end{tabular}}}}%
    \put(0.24207131,0.0070372){\color[rgb]{0,0,0}\makebox(0,0)[t]{\lineheight{1.25}\smash{\begin{tabular}[t]{c}$B$\end{tabular}}}}%
    \put(0.39157587,0.00341433){\color[rgb]{0,0,0}\makebox(0,0)[t]{\lineheight{1.25}\smash{\begin{tabular}[t]{c}$A$\end{tabular}}}}%
    \put(0.12638071,0.3120841){\color[rgb]{0,0,0}\makebox(0,0)[t]{\lineheight{1.25}\smash{\begin{tabular}[t]{c}$A$\end{tabular}}}}%
    \put(0.26091033,0.30773663){\color[rgb]{0,0,0}\makebox(0,0)[t]{\lineheight{1.25}\smash{\begin{tabular}[t]{c}$B$\end{tabular}}}}%
    \put(0.19902441,0.23301867){\color[rgb]{0,0,0}\makebox(0,0)[lt]{\lineheight{1.25}\smash{\begin{tabular}[t]{l}$\beta''$\end{tabular}}}}%
    \put(0.67692365,0.10449858){\color[rgb]{0,0,0}\makebox(0,0)[lt]{\lineheight{1.25}\smash{\begin{tabular}[t]{l}$2:1$\end{tabular}}}}%
    \put(0.72466218,0.18528313){\color[rgb]{0,0,0}\makebox(0,0)[t]{\lineheight{1.25}\smash{\begin{tabular}[t]{c}$B$\end{tabular}}}}%
    \put(0.84226269,0.18528313){\color[rgb]{0,0,0}\makebox(0,0)[t]{\lineheight{1.25}\smash{\begin{tabular}[t]{c}$A$\end{tabular}}}}%
    \put(0.96447444,0.18528313){\color[rgb]{0,0,0}\makebox(0,0)[t]{\lineheight{1.25}\smash{\begin{tabular}[t]{c}$B$\end{tabular}}}}%
    \put(0.60867943,0.18528313){\color[rgb]{0,0,0}\makebox(0,0)[t]{\lineheight{1.25}\smash{\begin{tabular}[t]{c}$A$\end{tabular}}}}%
    \put(0.91483354,0.09122889){\color[rgb]{0,0,0}\makebox(0,0)[lt]{\lineheight{1.25}\smash{\begin{tabular}[t]{l}$1:1$\end{tabular}}}}%
    \put(0.82970322,0.0070372){\color[rgb]{0,0,0}\makebox(0,0)[t]{\lineheight{1.25}\smash{\begin{tabular}[t]{c}$B$\end{tabular}}}}%
    \put(0.97920771,0.00341433){\color[rgb]{0,0,0}\makebox(0,0)[t]{\lineheight{1.25}\smash{\begin{tabular}[t]{c}$A$\end{tabular}}}}%
    \put(0.71401262,0.3120841){\color[rgb]{0,0,0}\makebox(0,0)[t]{\lineheight{1.25}\smash{\begin{tabular}[t]{c}$A$\end{tabular}}}}%
    \put(0.84854218,0.30773663){\color[rgb]{0,0,0}\makebox(0,0)[t]{\lineheight{1.25}\smash{\begin{tabular}[t]{c}$B$\end{tabular}}}}%
    \put(0.79997744,0.23506809){\color[rgb]{0,0,0}\makebox(0,0)[lt]{\lineheight{1.25}\smash{\begin{tabular}[t]{l}$2:1$\end{tabular}}}}%
    \put(0.49288376,0.20135488){\color[rgb]{0,0,0}\makebox(0,0)[t]{\lineheight{1.25}\smash{\begin{tabular}[t]{c}Surgery\end{tabular}}}}%
  \end{picture}%
\endgroup%

%% file: images/surgery-to-single-arcs-along-beta-arcs.pdf_tex
\begingroup%
  \makeatletter%
  \providecommand\color[2][]{%
    \errmessage{(Inkscape) Color is used for the text in Inkscape, but the package 'color.sty' is not loaded}%
    \renewcommand\color[2][]{}%
  }%
  \providecommand\transparent[1]{%
    \errmessage{(Inkscape) Transparency is used (non-zero) for the text in Inkscape, but the package 'transparent.sty' is not loaded}%
    \renewcommand\transparent[1]{}%
  }%
  \providecommand\rotatebox[2]{#2}%
  \newcommand*\fsize{\dimexpr\f@size pt\relax}%
  \newcommand*\lineheight[1]{\fontsize{\fsize}{#1\fsize}\selectfont}%
  \ifx\svgwidth\undefined%
    \setlength{\unitlength}{332.49674399bp}%
    \ifx\svgscale\undefined%
      \relax%
    \else%
      \setlength{\unitlength}{\unitlength * \real{\svgscale}}%
    \fi%
  \else%
    \setlength{\unitlength}{\svgwidth}%
  \fi%
  \global\let\svgwidth\undefined%
  \global\let\svgscale\undefined%
  \makeatother%
  \begin{picture}(1,0.18759105)%
    \lineheight{1}%
    \setlength\tabcolsep{0pt}%
    \put(0,0){\includegraphics[width=\unitlength,page=1]{surgery-to-single-arcs-along-beta-arcs.pdf}}%
    \put(0.67704935,0.11362702){\color[rgb]{0,0,0}\makebox(0,0)[lt]{\lineheight{1.25}\smash{\begin{tabular}[t]{l}$1:1$\end{tabular}}}}%
    \put(0.51287925,0.1285908){\color[rgb]{0,0,0}\makebox(0,0)[t]{\lineheight{1.25}\smash{\begin{tabular}[t]{c}Surgery\end{tabular}}}}%
    \put(0.30569215,0.10885123){\color[rgb]{0,0,0}\makebox(0,0)[rt]{\lineheight{1.25}\smash{\begin{tabular}[t]{r}$\beta$\end{tabular}}}}%
    \put(0.08117444,0.09770768){\color[rgb]{0,0,0}\makebox(0,0)[lt]{\lineheight{1.25}\smash{\begin{tabular}[t]{l}$\beta$\end{tabular}}}}%
    \put(0.32529849,0.05366421){\color[rgb]{0,0,0}\makebox(0,0)[t]{\lineheight{1.25}\smash{\begin{tabular}[t]{c}$\beta$\end{tabular}}}}%
    \put(0.34384333,0.1088512){\color[rgb]{0,0,0}\makebox(0,0)[lt]{\lineheight{1.25}\smash{\begin{tabular}[t]{l}$\beta$\end{tabular}}}}%
    \put(0.89188383,0.04263003){\color[rgb]{0,0,0}\makebox(0,0)[t]{\lineheight{1.25}\smash{\begin{tabular}[t]{c}$1:1$\end{tabular}}}}%
    \put(0.91085509,0.11743257){\color[rgb]{0,0,0}\makebox(0,0)[lt]{\lineheight{1.25}\smash{\begin{tabular}[t]{l}$1:1$\end{tabular}}}}%
    \put(0.86912277,0.11607604){\color[rgb]{0,0,0}\makebox(0,0)[rt]{\lineheight{1.25}\smash{\begin{tabular}[t]{r}$1:1$\end{tabular}}}}%
    \put(0.65934682,0.11305867){\color[rgb]{0,0,0}\makebox(0,0)[rt]{\lineheight{1.25}\smash{\begin{tabular}[t]{r}$1:1$\end{tabular}}}}%
  \end{picture}%
\endgroup%

%% file: images/surgery-to-triple-arc-with-asterisk.pdf_tex
\begingroup%
  \makeatletter%
  \providecommand\color[2][]{%
    \errmessage{(Inkscape) Color is used for the text in Inkscape, but the package 'color.sty' is not loaded}%
    \renewcommand\color[2][]{}%
  }%
  \providecommand\transparent[1]{%
    \errmessage{(Inkscape) Transparency is used (non-zero) for the text in Inkscape, but the package 'transparent.sty' is not loaded}%
    \renewcommand\transparent[1]{}%
  }%
  \providecommand\rotatebox[2]{#2}%
  \newcommand*\fsize{\dimexpr\f@size pt\relax}%
  \newcommand*\lineheight[1]{\fontsize{\fsize}{#1\fsize}\selectfont}%
  \ifx\svgwidth\undefined%
    \setlength{\unitlength}{126.86480785bp}%
    \ifx\svgscale\undefined%
      \relax%
    \else%
      \setlength{\unitlength}{\unitlength * \real{\svgscale}}%
    \fi%
  \else%
    \setlength{\unitlength}{\svgwidth}%
  \fi%
  \global\let\svgwidth\undefined%
  \global\let\svgscale\undefined%
  \makeatother%
  \begin{picture}(1,0.68580242)%
    \lineheight{1}%
    \setlength\tabcolsep{0pt}%
    \put(0,0){\includegraphics[width=\unitlength,page=1]{surgery-to-triple-arc-with-asterisk.pdf}}%
    \put(0.20027428,0.61960324){\color[rgb]{0,0,0}\makebox(0,0)[rt]{\lineheight{1.25}\smash{\begin{tabular}[t]{r}$\beta$\end{tabular}}}}%
    \put(0.39917157,0.32155295){\color[rgb]{0,0,0}\makebox(0,0)[t]{\lineheight{1.25}\smash{\begin{tabular}[t]{c}$\beta$\end{tabular}}}}%
    \put(0.64052369,0.60386857){\color[rgb]{0,0,0}\makebox(0,0)[lt]{\lineheight{1.25}\smash{\begin{tabular}[t]{l}$\beta$\end{tabular}}}}%
  \end{picture}%
\endgroup%